\documentclass[12pt]{amsart}

\usepackage{amssymb,amsmath,amsbsy,eucal,amsfonts,mathrsfs,latexsym}
\usepackage{color,graphicx,eucal,tikz,amscd,stmaryrd}
\usetikzlibrary{calc}
\usetikzlibrary{through}

\newcommand{\boldx}{\boldsymbol{x}}
\newcommand{\boldy}{\boldsymbol{y}}

\newtheorem{theorem}{Theorem}[section]
\newtheorem{lemma}{Lemma}[section]

\theoremstyle{definition}

\theoremstyle{remark}
\newtheorem{remark}[theorem]{Remark}
\newtheorem{definition}[theorem]{Definition}

\begin{document}

\title[A finite element method for non-divergence form]
{Analysis of a finite element method for second order uniformly elliptic PDEs in non-divergence form}

\author{Weifeng Qiu}
\address{Department of Mathematics, City University of Hong Kong}
\email{weifeqiu@cityu.edu.hk}


\date{}

\dedicatory{}

\begin{abstract}
We propose one finite element method for both second order linear uniformly elliptic PDE in non-divergence form 
and the uniformly elliptic Hamilton-Jacobi-Bellman (HJB) equation. 
For both linear elliptic PDE in non-divergence form and the HJB equation, 
we prove the well-posedness of strong solution in $W^{2,p}(\Omega)$ and optimal convergence in discrete 
$W^{2,p}$-norm of the finite element approximation to the strong solution for $1<p\leq 2$ 
on convex polyhedra in $\mathbb{R}^{d}$ ($d=2,3$). If the domain is a two dimensional non-convex 
polygon, $p$ is valid in a more restricted region. Furthermore, we relax the assumptions on the continuity 
of coefficients of the HJB equation, which have been widely used in literature.
\end{abstract}


\maketitle

\section{Introduction}
We consider the well-posedness of the strong solution in $W^{2,p}(\Omega)$ and optimal convergence to the strong 
solution in discrete $W^{2,p}$-norm of a finite element method for both second order linear uniformly elliptic PDE 
in non-divergence form and the uniformly elliptic Hamilton-Jacobi-Bellman equation on Lipschitz polyhedral domain 
$\Omega$ in $\mathbb{R}^{d}$ ($d=2,3$). The Lipschitz polyhedral domain $\Omega$ is assumed to be convex 
if $d=3$, and can be non-convex if $d=2$.

The linear uniformly elliptic PDE in non-divergence form is find the solution $u$ satisfying
\begin{align}
\label{nondiv_pde_original}
A:D^{2}u + \boldsymbol{b}\cdot\nabla u + cu = f \text{ in } \Omega, \qquad 
u = 0 \text{ on } \partial\Omega.
\end{align}
We assume that $A \in [L^{\infty}(\Omega)]^{d\times d}$ is a symmetric and uniformly positive definite 
matrix-valued function on $\Omega$, $\boldsymbol{b}\in [L^{\infty}(\Omega)]^{d}$, and $c \in L^{\infty}(\Omega)$ 
with $c \leq 0$.

The uniformly elliptic Hamilton-Jacobi-Bellman (HJB) equation is to find the solution $u$ satisfying 
\begin{align}
\label{hjb_eqs_original}
\sup_{\alpha \in \Lambda} [A^{\alpha}:D^{2}u 
+ \boldsymbol{b}^{\alpha}\cdot \nabla u + c^{\alpha}u - f^{\alpha}] = 0 \text{ in } \Omega,
\qquad u = 0 \text{ on } \partial\Omega.
\end{align}
Here the non-empty set $\Lambda$ is called index set.
We assume that there are positive constants $\nu \leq \overline{\nu}$ such that 
for any $\boldsymbol{\xi} \in \mathbb{R}^{d}$ and any $\alpha \in \Lambda$,
\begin{align}
\label{hjb_uniform_ellipticity}
& \nu \vert \boldsymbol{\xi}\vert^{2} \leq \Sigma_{1 \leq i,j \leq d} 
a_{ij}^{\alpha}(\boldx) \xi_{i} \xi_{j} \leq \overline{\nu} 
\vert \boldsymbol{\xi}\vert^{2} \text{ and } 
c^{\alpha}(\boldx) \leq 0, \quad \forall 
\boldx \in \Omega \text{ almost everywhere}; \\
\label{hjb_coeffs_bounds}
& \sup_{\alpha \in \Lambda} \Vert \boldsymbol{b}^{\alpha}
\Vert_{L^{\infty}(\Omega)} + \sup_{\alpha \in \Lambda} 
\Vert c^{\alpha}\Vert_{L^{\infty}(\Omega)}  < + \infty.
\end{align}

The linear elliptic PDE in non-divergence form (\ref{nondiv_pde_original}) is the starting point to 
study the elliptic HJB equation (\ref{hjb_eqs_original}), because (\ref{hjb_eqs_original}) becomes 
(\ref{nondiv_pde_original}) if the index set $\Lambda$ contains a single element.
In addition, (\ref{nondiv_pde_original}) appears in the linearization and numerical methods of fully non-linear 
second order PDEs \cite{Caffarelli1,FN1,Neilan1}. 
The HJB equation (\ref{hjb_eqs_original}) provides a fundamental characterization of the value functions 
associated with stochastic control problems, which have widespread applications across engineering, physics, 
economics, and finance \cite{FS1,JensenSmears2013}.

\subsection{Literature review}

Unlike partial differential equations expressed in divergence form, the theory governing linear equations 
in non-divergence form, and more broadly, fully nonlinear PDEs, relies on alternative solution frameworks, 
including classical solutions, viscosity solutions, and strong solutions. 

There are several papers 
\cite{BS1991,BZ1,CamilliJakobsen2009,CL1,DK1,FL1,FS1,JensenSmears2013,
Kocan1,Oberman1,NZ1,SalgadoZhang2019}
focusing on convergence to the viscosity solutions.  
In the Barles–Souganidis framework \cite{BS1991}, monotonicity serves as a fundamental principle underpinning 
numerical schemes to ensure convergence to viscosity solutions of fully nonlinear partial differential equations. 
This concept is also applicable to the monotone finite difference methods for the HJB equation \cite{BZ1,FS1}. 
A monotone finite element like scheme for the HJB equation was proposed in \cite{CamilliJakobsen2009}. 
In \cite{NZ1}, Nochetto and Zhang introduced a two-scale method based on the integro-differential 
formulation of (\ref{nondiv_pde_original}), where a discrete Alexandroff–Bakelman–Pucci estimate 
was established. This two-scale method \cite{NZ1} was extended to the Isaacs equation \cite{SalgadoZhang2019}.
However, up to our knowledge, all these numerical methods converge to the viscosity with low 
order accuracy. In addition, most of them have stencils which increase 
under mesh refinement, such that the corresponding linear system may not be sparse. 

With respect to numerical methods for the convergence to the strong solutions, 
there are three main categories based on their techniques for the linear elliptic PDEs 
in non-divergence form. 
\begin{itemize}

\item[(I)] \textbf{$A$ can be discontinuous but $\gamma A$ dominated by $I_{d}$ where 
$\gamma$ is a positive weight function}. 
In \cite{SmearsSuli2013}, Smears and S\"uli introduced and analyzed a Discontinuous Galerkin 
(DG) method for (\ref{nondiv_pde_original}). The main assumptions are 
the coefficients $A$ is dominated by $I_{d}$ after being multiplied by a 
positive weight function $\gamma$ (Cordes condition) and the domain $\Omega$ is convex. 
The stability and optimal rate of convergence in discrete $H^{2}$-norm were 
proven in \cite{SmearsSuli2013}. Following this direction, there are several papers 
\cite{BHW2021,Gallistl1,Wu2021}. The technique in \cite{SmearsSuli2013} 
has been extended to the HJB equation (\ref{hjb_eqs_original}) in 
\cite{SmearsSuli2014,Gallistl2,Wu2021}. 
It is not clear how to prove optimal rate of convergence in  
discrete $W^{2,p}$-norm ($p\neq 2$) for this kind of numerical method.

\item[(II)] \textbf{$A$ is globally continuous}. In \cite{FHM2017}, Feng, Hennings and Neilan 
introduced and analyzed a $C^{0}$-conforming DG method for 
(\ref{nondiv_pde_original}). The main assumptions are $A \in C^{0}(\overline{\Omega})$ 
and the domain $\Omega$ has $C^{1,1}$ boundary. They followed the approach in 
\cite[Chapter~$9$]{GT01} to prove that their numerical 
method achieves stability and optimal rate of convergence 
in discrete $W^{2,p}$-norm for any $1 < p < +\infty$. 
However, this numerical method 
and its variant - a DG method in \cite{FNS1} can not deal with discontinuous 
coefficients $A$, since they have the jump of $A\nabla u_{h}$ along mesh interfaces 
($u_{h}$ is the numerical solution). Therefore, the numerical methods in 
\cite{FHM2017,FNS1} can not be applied for the HJB equation (\ref{hjb_eqs_original}).

\item[(III)] \textbf{Residual minimization methods}. 
The main advantage of residual minimization methods \cite{QiuZhang2020,Tran2025,GallistlTran2025} 
is that they can easily inherit the a priori estimate of PDE operators of (\ref{nondiv_pde_original}) 
and (\ref{hjb_eqs_original}). However, the minimal requirements of the regularity of 
true solution $u$ for these methods are highly 
affected by the choice of the norm to measure the residual in their designs. 
It is not clear which norm to choose for the two cases mentioned above.

\end{itemize}

A natural question is whether it is possible to design one numerical method for both 
linear elliptic PDE in non-divergence form and the elliptic HJB equation on Lipschitz 
polyhedra in $\mathbb{R}^{d}$ ($d=2,3$) and achieve optimal convergence in discrete 
$W^{2,p}$-norm to the strong solution with a wide range of $p$. 
We give positive answer in this paper.

\subsection{Main contributions}

We introduce the finite element methods (\ref{nondiv_fem}) and (\ref{hjb_fem}) for 
the linear elliptic PDE in non-divergence form (\ref{nondiv_pde_original}) and 
the HJB equation (\ref{hjb_eqs_original}), respectively. It is easy to see that 
the FEM (\ref{hjb_fem}) becomes (\ref{nondiv_fem}), if the index set $\Lambda$ has 
a single element. Therefore, we can claim that essentially, there is only one 
FEM for (\ref{nondiv_pde_original}) and (\ref{hjb_eqs_original}).
We would like to point out that the discrete gradient operator $\overline{\nabla}_{h}$ 
(see (\ref{discrete_grad})) used in (\ref{nondiv_fem}) and (\ref{hjb_fem}) has been 
introduced in \cite{LakkisPryer2011}, and the finite element method (\ref{nondiv_fem}) 
is similar to the one in \cite{LakkisPryer2011}. \cite{Neilan2} provides optimal convergence 
in discrete $H^{2}$-norm for the method in \cite{LakkisPryer2011} for (\ref{nondiv_pde_original}) 
with $A\in C^{0}(\overline{\Omega})$ on domains with $C^{1,1}$ boundary. 
The technique in \cite{Neilan2} is similar to those in \cite{FHM2017,FNS1}.
It is not clear how to generalize the analysis in \cite{Neilan2} for (\ref{hjb_eqs_original}).

With respect to the $W^{2,p}$-norm for PDE analysis and the discrete $W^{2,p}$-norm (the $W_{h}^{2,p}$-norm introduced 
in (\ref{discrete_second_order_norm})) for numerical analysis, the ranges of $p$ considered in this paper depend on 
whether the the Lipschitz polyhedral domain $\Omega$ is convex. The ranges of $p$ are described below: 
\begin{subequations}
\label{p_ranges}
\begin{align}
\label{p_range1}
& \text{(PDE analysis) } p \in (1,2]  \text{ if } \Omega  \text{ is a convex polyhedra }(d=2,3), \\
\nonumber
& \qquad \text{ while } p \in (1, \frac{4}{3} + \epsilon_{2}) \text{ if } \Omega 
\text{ is a non-convex polygon in } \mathbb{R}^{2}; \\
\label{p_range2}
& \text{(Numerical analysis) } p \in (1,2]  \text{ if } \Omega  \text{ is a convex polyhedra }(d=2,3), \\
\nonumber 
& \qquad \text{ while } p \in (\frac{4}{3}-\epsilon_{1}, \frac{4}{3} + \epsilon_{2}) \text{ if } \Omega 
\text{ is a non-convex polygon in } \mathbb{R}^{2}.
\end{align}
\end{subequations}
Here $\epsilon_{1}$ and $\epsilon_{2}$ are two positive numbers which depend on the domain $\Omega$ 
(see Lemma~\ref{lemma_Poisson_regularity}). 

For the elliptic HJB equation (\ref{hjb_eqs_original}), we further assume that there is a subset 
$E_{0} \subset \Omega$ with zero $d$-dimensional Lebesgue measure such that 
\begin{align}
\label{hjb_coeffs_alternatives} 
& \text{there is a countable subset } \tilde{\Lambda}\subset \Lambda \text{ satisfying that}\\
\nonumber 
& \qquad \forall \boldx \in \Omega \setminus E_{0} \text{ and } \forall 
\alpha \in \Lambda \text{ and } \forall \epsilon > 0, \text{there is } 
\tilde{\alpha} \in \tilde{\Lambda} \text{ such that} \\
\nonumber 
& \qquad \vert A^{\alpha}(\boldx) - A^{\tilde{\alpha}}(\boldx)\vert 
+ \vert \boldsymbol{b}^{\alpha}(\boldx) - \boldsymbol{b}^{\tilde{\alpha}}
(\boldx) \vert + \vert c^{\alpha}(\boldx) - c^{\tilde{\alpha}}(\boldx) \vert 
+ \vert f^{\alpha}(\boldx) - f^{\tilde{\alpha}}(\boldx) \vert < \epsilon.
\end{align}
Here, we denote by $\vert B \vert= \left( \Sigma_{1\leq i, j \leq d} \vert b_{ij} \vert^{2} \right)^{\frac{1}{2}}$ 
for any matrix $B \in \mathbb{R}^{d\times d}$. 

\subsubsection{Contributions for the linear elliptic PDE in non-divergence form (\ref{nondiv_pde_original})}

For the linear elliptic PDE in non-divergence form (\ref{nondiv_pde_original}), 
the main theoretical results are Theorem~\ref{thm_A_continuous_complete}, 
Theorem~\ref{thm_pde_cordes} and Theorem~\ref{thm_global_estimate_cordes}. 

When the coefficient matrix $A \in C^{0}(\overline{\Omega})$, Theorem~\ref{thm_A_continuous_complete} 
provides the well-posedness of strong solution in $W^{2,p}(\Omega)$ and optimal convergence 
with respect to $W_{h}^{2,p}$-norm (see (\ref{discrete_second_order_norm})) of the numerical 
solution of the FEM (\ref{nondiv_fem}). The well-posedness of (\ref{nondiv_pde_original}) 
provided by \cite[Theorem~$9.15$]{GT01} and optimal convergence of the FEM in \cite{FHM2017}
require $\partial\Omega$ to be $C^{1,1}$. 
In contrast, Theorem~\ref{thm_A_continuous_complete} 
is valid on Lipschitz polyhedral domains (see (\ref{p_ranges}) for the detailed description). 
In fact, with the help of the well-posedness of strong solution stated in 
Theorem~\ref{thm_A_continuous_complete},  the $W_{h}^{2,p}$-norm stability of the FEM in \cite{FHM2017} 
can be immediately extended from domains with $C^{1,1}$ 
boundary to Lipschitz polyhedral domains with $p$ described in (\ref{p_range2}). 
Furthermore, in contrast to the proof of numerical stability presented in \cite{FHM2017}, which rests upon 
a specialized duality argument (see the proof of \cite[Lemma~$3.5$]{FHM2017}) contingent on 
the well-posedness of (\ref{nondiv_pde_original}),  the present work establishes numerical stability 
(Theorem~\ref{thm_global_estimate}) by invoking the discrete compactness property  (Lemma~\ref{lemma_conv_compactness}) 
and the uniqueness of the strong solution of (\ref{nondiv_pde_original}). Thus, the analysis of this paper is 
quite different from that of \cite{FHM2017,FNS1,Neilan2}.  

We explain the methodology to prove 
Theorem~\ref{thm_A_continuous_complete} by the following steps. 
\begin{itemize}

\item Step $1$. We prove Theorem~\ref{thm_global_A_continuous} which give the following global 
$W^{2,p}$ estimate (\ref{global_w2p_estimate_A_continuous}) of the linear elliptic PDE in 
non-divergence form (\ref{nondiv_pde_original}):
\begin{align*}
\Vert w\Vert_{W^{2,p}(\Omega)} \leq C \Vert A:D^{2}w + \boldsymbol{b}\cdot\nabla w + cw \Vert_{L^{p}(\Omega)}, 
\quad \forall w \in W^{2,p}(\Omega) \cap W_{0}^{1,p}(\Omega),
\end{align*} 
where the range of $p$ is described in (\ref{p_range1}).
The most difficult part of the proof of Theorem~\ref{thm_global_A_continuous} is to 
obtain the above estimate for $1 < p \leq \frac{3}{2}$ if 
$\Omega$ is a convex polyhedra in $\mathbb{R}^{3}$. It is because the Alexandroff–Bakelman–Pucci estimate 
can not be applied directly due to the fact $W^{2,p}(\Omega)$ can not be embedded into $C^{0}(\overline{\Omega})$.
To overcome this difficulty, we provide Lemma~\ref{lemma_pde_unique1} which shows the uniqueness of the strong solution 
of (\ref{nondiv_pde_original}) in $W^{2,p}(\Omega) \cap W_{0}^{1,p}(\Omega)$. Then Theorem~\ref{thm_global_A_continuous} 
basically follows \cite[Theorem~$1.1$]{QT2023} by a proof by contraction. 
Though it is inspired by that of \cite[Lemma~$9.16$]{GT01}, 
the proof of Lemma~\ref{lemma_pde_unique1} avoids using the common technique of ``flattening" 
$\partial\Omega$, which is critical in \cite[Lemma~$9.16$]{GT01}. Thus Lemma~\ref{lemma_pde_unique1} 
is valid on convex polyhedra in $\mathbb{R}^{3}$ (It is trivial to obtain the uniqueness on two dimensional 
Lipschitz polygon). 

\item Step $2$. We provide Lemma~\ref{lemma_gloabl_estimate_with_reaction} which give the inequality 
(\ref{gloabl_estimate_with_reaction}): 
\begin{align*}
\Vert w_{h} \Vert_{W_{h}^{2,p}(\Omega)} \leq C \left( \Vert \mathcal{L}_{\tilde{A},h}w_{h}\Vert_{L_{h}^{p}(\Omega)} 
+ \Vert w_{h}\Vert_{W^{1,p}(\Omega)} \right), \qquad \forall w_{h} \in V_{h}.
\end{align*}
Here, $V_{h}$ is the finite element space introduced in Section~\ref{sec_notation}, 
$\mathcal{L}_{\tilde{A},h}$ is the Riesz representation (see (\ref{L_h_operator})) of the FEM (\ref{nondiv_fem}) with 
$\boldsymbol{b}=\boldsymbol{0}$ and $c=0$.
 The proof of Lemma~\ref{lemma_gloabl_estimate_with_reaction} 
mimics that of \cite[Lemma~$3.4$]{FHM2017}. However, due to the construction of the discrete gradient 
operator (\ref{discrete_grad}), the proof of Lemma~\ref{lemma_gloabl_estimate_with_reaction} is much 
more technical.

\item Step $3$. To erase the term $\Vert w_{h}\Vert_{W^{1,p}(\Omega)}$ on the right hand side of 
(\ref{gloabl_estimate_with_reaction}), we prove by contradiction with the help of 
the discrete compactness - Lemma~\ref{lemma_conv_compactness} (can be considered as an analogue of 
\cite[Theorem~$5.2$]{BuffaOrtner2009} for the discrete $W^{2,p}$-norm) and the uniqueness of 
strong solution of (\ref{nondiv_pde_original}) by Theorem~\ref{thm_global_A_continuous}. Then we obtain 
Theorem~\ref{thm_global_estimate} which the stability of the FEM (\ref{nondiv_fem}):
\begin{align*}
\Vert w_{h} \Vert_{W_{h}^{2,p}(\Omega)} \leq C \Vert \tilde{A}: \overline{\nabla}_{h}(\nabla w_{h}) +\tilde{\boldsymbol{b}}\cdot\nabla w_{h} + \tilde{c} w_{h}\Vert_{L_{h}^{p}(\Omega)}, 
\qquad \forall w_{h} \in V_{h}.
\end{align*}
We would like to point out that to obtain the same result, the authors of \cite{FHM2017} used a specialized duality 
argument (see the proof of \cite[Lemma~$3.5$]{FHM2017}), which relies on the well-posedness of strong solution of 
(\ref{nondiv_pde_original}). Actually, our approach in this step is a discrete analogue of \cite[Lemma~$9.17$]{GT01}.
Furthermore, to prove Lemma~\ref{lemma_conv_compactness}, 
we provide Lemma~\ref{lemma_discrete_norm_control} which shows for any $p \in [1, + \infty]$, 
\begin{align*}
\Vert v_{h}\Vert_{W^{1,p}(\Omega)} \leq C \Vert v_{h}\Vert_{W_{h}^{2,p}(\Omega)}, 
\qquad \forall v_{h} \in V_{h}.
\end{align*}
Here the constant $C$ is independent of $p \in [1, + \infty]$.

\item Step $4$. Roughly speaking, we apply the discrete compactness - Lemma~\ref{lemma_conv_compactness} with 
Theorem~\ref{thm_global_A_continuous} to prove the limit of numerical solution of the FEM (\ref{nondiv_fem}) is 
a solution of (\ref{nondiv_pde_original}). Then by the global $W^{2,p}$ estimate (\ref{global_w2p_estimate_A_continuous}), 
the well-posedness of a strong solution to (\ref{nondiv_pde_original}) is proven for the range of $p$ 
described in (\ref{p_range2}). Finally by some density argument, we extend $p$ to the range described in (\ref{p_range1}).  
Optimal convergence of the numerical solution of the FEM (\ref{nondiv_fem}) is straightforward. 
\end{itemize}

When the coefficient matrix $A$ is discontinuous but $\gamma A$ is dominated by $I_{d}$ (the positive 
weight function $\gamma$ is introduced in (\ref{def_gamma})), Theorem~\ref{thm_pde_cordes} and 
Theorem~\ref{thm_global_estimate_cordes} provide well-posedness of strong solution in $W^{2,p}(\Omega)$ 
of (\ref{nondiv_pde_original}) and optimal convergence with respect to $W_{h}^{2,p}$-norm of the numerical 
solution of the FEM (\ref{nondiv_fem}). We would like to point out that Theorem~\ref{thm_pde_cordes} and 
Theorem~\ref{thm_global_estimate_cordes} are special cases of Theorem~\ref{thm_hjb_pde_wellposedness} 
and Theorem~\ref{thm_hjb_conv} for the elliptic HJB equation, respectively.

\subsubsection{Contributions for the elliptic HJB equation (\ref{hjb_eqs_original})}

For the elliptic HJB equation (\ref{hjb_eqs_original}), the main theoretical results are 
Theorem~\ref{thm_hjb_pde_wellposedness} and Theorem~\ref{thm_hjb_conv}.  
Theorem~\ref{thm_hjb_pde_wellposedness} provides the well-posedness of 
strong solution in $W^{2,p}(\Omega)$ of the elliptic HJB equation (\ref{hjb_eqs_original}),  
if the coefficients satisfy the condition (\ref{Cordes_coefficients_strong_hjb_pde_general}). 
In parallel, Theorem~\ref{thm_hjb_conv} provides optimal convergence of the FEM (\ref{hjb_fem}) 
in $W_{h}^{2,p}$-norm, if the coefficients satisfy the condition 
(\ref{Cordes_coefficients_strong_hjb_fem_general}). 
When the domain is convex and $p=2$, the condition (\ref{Cordes_coefficients_strong_hjb_pde_general}) 
is identical to the Cordes condition used in \cite{SmearsSuli2014} but 
the condition~(\ref{Cordes_coefficients_strong_hjb_fem_general}) is more restrictive.
Please see Remark~\ref{remark_hjb_pde_wellposedness} and Remark~\ref{remark_hjb_conv} for the special 
case $\boldsymbol{b}^{\alpha}=\boldsymbol{0}$ and $c^{\alpha} = 0$ for any $\alpha \in \Lambda$. 

We would like to emphasize that \cite{SmearsSuli2014} assumes $A^{\alpha} \in C^{0}(\overline{\Omega}\times 
\Lambda)$ and the index set $\Lambda$ is a compact metric space, while \cite[Proposition~$4.3$]{GallistlTran2025} 
assumes that there are $\beta \in (0,1]$ and $0< C < + \infty$ such that 
$\Vert A^{\alpha}\Vert_{C^{0,\beta}(\overline{\Omega})}  < C$  for any $\alpha \in \Lambda$ 
(\cite[Proposition~$4.3$]{GallistlTran2025} assumes $\boldsymbol{b}^{\alpha}=\boldsymbol{0}$, $c^{\alpha} = 0$, 
and $f^{\alpha}=f$). 
So, if the index set $\Lambda$ has a single element, the elliptic HJB equation (\ref{hjb_eqs_original}) 
will not become the linear elliptic PDE in non-divergence form (\ref{nondiv_pde_original}) with 
discontinuous matrix coefficient $A$. In contrast, the assumption~(\ref{hjb_coeffs_alternatives}) allows 
discontinuity of $A^{\alpha}$, $\boldsymbol{b}^{\alpha}$ and $c^{\alpha}$. By 
Lemma~\ref{lemma_equi_uniform_continuity_to_alternatives}, we can conclude that the assumptions used by 
\cite{SmearsSuli2014} and \cite[Proposition~$4.3$]{GallistlTran2025} imply that the 
assumption~(\ref{hjb_coeffs_alternatives}) holds.

We notice that for any $\alpha \in \Lambda$, $\boldsymbol{b}^{\alpha}$ and $c^{\alpha}$ may be discontinuous. 
To have one numerical method for the linear elliptic PDE in non-divergence form (\ref{nondiv_pde_original}) and 
the HJB equation (\ref{hjb_eqs_original}), the positive weight function $\gamma^{\alpha}$ in (\ref{def_gamma_alpha}) 
can not depend on $\boldsymbol{b}^{\alpha}$ and $c^{\alpha}$. Otherwise the analysis for the FEM (\ref{nondiv_fem}) 
for (\ref{nondiv_pde_original}) will not be valid for $A\in C^{0}(\Omega)$. Furthermore, we don't want to include 
the parameter $\lambda$ used in the similar weight function in \cite{SmearsSuli2014} which might not be known 
beforehand. These restrictions make the analysis 
of the FEM (\ref{hjb_fem}) for (\ref{hjb_eqs_original}) more difficult than that in \cite{SmearsSuli2014}.

\subsection{Structure of this paper}
The rest of this paper is organized as follows. In Section~$2$, we provide basic notations, the FEMs
for (\ref{nondiv_pde_original}) and (\ref{hjb_eqs_original}), and some preliminary results. In 
Section~$3$, we provide theoretical results for (\ref{nondiv_pde_original}). In Section~$4$, we provide 
theoretical results for (\ref{hjb_eqs_original}).

\section{Assumptions, notations, FEM and preliminary results}

In Section~\ref{sec_notation}, we provide the assumption of mesh, the discrete gradient operator (\ref{discrete_grad}) 
and several basic notations. In Section~\ref{sec_fem}, we introduce the FEM (\ref{nondiv_fem}) and the FEM  
(\ref{hjb_fem}) for (\ref{nondiv_pde_original}) and (\ref{hjb_eqs_original}), respectively. Actually, 
the FEM (\ref{hjb_fem}) becomes the FEM (\ref{nondiv_fem}) if the index set $\Lambda$ has a single element. 
In Section~\ref{sec_preliminary}, we provide some preliminary results. 

\subsection{Mesh and space notation}
\label{sec_notation}
We denote by $\mathcal{T}_{h}$ a quasi-uniform, simplicial, and conforming triangulation of the domain $\Omega$.
Let $\mathcal{F}_{h}^{I}$ be the collection of all $(d-1)$-dimensional interior mesh faces in $\mathcal{T}_{h}$, 
$\mathcal{F}_{h}^{B}$ the collection of $(d-1)$-dimensional boundary mesh faces in $\mathcal{T}_{h}$, and 
$\mathcal{F}_{h}:=\mathcal{F}_{h}^{I} \cup \mathcal{F}_{h}^{B}$ 
the collection of all $(d-1)$-dimensional mesh faces in $\mathcal{T}_{h}$.

Let $K^{+},K^{-}\in \mathcal{T}_{h}$ and $F = \partial K^{+} \cap \partial K^{-} \in \mathcal{F}_{h}^{I}$.
Without loss of generality, we assume that the global labelling number of $K^{+}$ is larger than that of $K^{-}$. 
We introduce the jump and average of a scalar or vector valued function $v$ as 
\begin{align*}
[v]|_{F}:= v^{+} - v^{-}, \qquad \{ v \}|_{F} = \frac{1}{2}\left(v^{+} + v^{-}\right),
\end{align*}
where $v^{\pm} = v|_{K^{\pm}}$.
On a boundary mesh face $F \in \mathcal{F}_{h}^{B}$ with $F = \partial K^{+} \cap \partial\Omega$, 
we define $[v]|_{F} = \{ v\}|_{F} = v^{+}$.

We define $V_{h}=H_{0}^{1}(\Omega) \cap P_{r}(\mathcal{T}_{h})$ and $\overline{V}_{h} = P_{r}(\mathcal{T}_{h})$ 
with integer $r \geq 2$. We also define the piecewise Sobolev space with respect to the mesh $\mathcal{T}_{h}$: 
\begin{align*}
& W^{s,p}(\mathcal{T}_{h}):= \boldsymbol{\Pi}_{K \in \mathcal{T}_{h}} W^{s,p}(K), \qquad 
W_{h}^{p}:= W^{2,p}(\mathcal{T}_{h}) \cap W_{0}^{1,p}(\Omega),  
\qquad W_{h}^{s,p}(D):= W^{s,p}(\mathcal{T}_{h})|_{D}.
\end{align*}
For a given subdomain $D \subset \Omega$. we define $V_{h}(D) \subset V_{h}$ and $W_{h}^{p}(D) \subset W_{h}^{p}(D)$ 
as the subspaces which vanish outside of $D$ by 
\begin{align*}
V_{h}(D):= \{ v \in V_{h}: v_{h}|_{\Omega \setminus D} = 0 \}, \qquad 
W_{h}^{p}(D):= \{ v \in W_{h}^{p}: v|_{\Omega \setminus D} = 0 \}.
\end{align*}

Associated with $D \subset \Omega$, we define a semi-norm on $W_{h}^{2,p}(D)$ for $p \in [1, +\infty]$: 
\begin{align}
\label{discrete_second_order_norm}
\Vert v\Vert_{W_{h}^{2,p}(D)} = \Vert D_{h}^{2}v\Vert_{L^{p}(D)} 
+ \left( \Sigma_{F \in \mathcal{F}_{h}^{I}} h_{F}^{1-p}\Vert [ \nabla v] \Vert_{L^{p}(F\cap 
\overline{D})}^{p} \right)^{\frac{1}{p}}.
\end{align}
Here $D_{h}^{2}v$ denotes the element-wise Hessian matrix of $v$. Obviously, 
\begin{align*}
& \Vert v\Vert_{W_{h}^{2,p}(\Omega)} = \Vert D_{h}^{2}v\Vert_{L^{p}(\Omega)} 
+ \left( \Sigma_{F \in \mathcal{F}_{h}^{I}} h_{F}^{1-p}\Vert [ \nabla v] 
\Vert_{L^{p}(F)}^{p} \right)^{\frac{1}{p}}, \\
& \Vert v\Vert_{W_{h}^{2,+\infty}(\Omega)} = \Vert D_{h}^{2}v\Vert_{L^{\infty}(\Omega)} 
+ \max_{F \in \mathcal{F}_{h}^{I}}h_{F}^{-1}\Vert [ \nabla v]\Vert_{L^{\infty}(F)}.
\end{align*}

For any subdomain $D\subset \Omega$ whose inner radius bigger than $2h$ and any $w \in L_{h}^{p}(D)$, we 
introduce the following mesh-dependent semi-norm 
\begin{align}
\label{discrete_zero_order_norm}
\Vert w\Vert_{L_{h}^{p}(D)}:= \sup_{v_{h} \in V_{h}(D)} \dfrac{(w, v_{h})_{D}}{\Vert v_{h}
\Vert_{L^{p^{\prime}}(D)}},
\end{align} 
where $\frac{1}{p} + \frac{1}{p^{\prime}} = 1$.

For any $F = \partial K^{+}\cap \partial K^{-} \in \mathcal{F}_{h}^{I}$, we denote by $\boldsymbol{n}_{F}$ 
the unit outward normal vector pointing in the direction of the element with the smaller global index.  
Therefore $\boldsymbol{n}_{F} = \boldsymbol{n}_{K^{+}} = - \boldsymbol{n}_{K^{-}}$. For $F \in \mathcal{F}_{h}^{B}$, 
we define $\boldsymbol{n}_{F}$ to be the unit outward normal vector to $\partial\Omega$ restricted to $F$.

\begin{definition}
\label{def_discrete_deri}
We denote by $W^{1,1}(\mathcal{T}_{h}):= \{\phi \in L^{1}(\Omega): \phi|_{K} \in W^{1,1}(K), \forall K \in\mathcal{T}_{h}\}$.  
We define the discrete partial derivatives $\overline{\partial}_{h, x_{i}}: W^{1,1}(\mathcal{T}_{h}) \rightarrow 
\overline{V}_{h}$ by 
\begin{equation}
	\label{discrete_deri}
	(\overline{\partial}_{h,x_{i}}v,\varphi_{h})_{\mathcal{T}_{h}}=\left\langle \{ v \} n^{(i)}, [\varphi_{h}] 
	\right\rangle_{\mathcal{F}_{h}}-(v,\partial_{x_{i}}\varphi_{h})_{\mathcal{T}_h}, \qquad 
	\forall \varphi_{h} \in \overline{V}_{h}.  
\end{equation} 
Here $v \in W^{1,1}(\mathcal{T}_{h})$, $i \in \{1,\cdots, d\}$ and $\boldsymbol{n}_{F} 
= ( n_{F}^{(1)},\cdots, n_{F}^{(d)})^{\top}$ for any $F \in \mathcal{F}_{h}$. 
It is easy to see that 
\begin{align}
\label{discrete_deri_equivalent}
	(\overline{\partial}_{h,x_{i}}v,\varphi_{h})_{\mathcal{T}_{h}}=- \left\langle [ v ] n^{(i)}, \{\varphi_{h}\} 
	\right\rangle_{\mathcal{F}_{h}^{I}} + (\partial_{x_{i}}v,\varphi_{h})_{\mathcal{T}_h}, \qquad 
	\forall \varphi_{h} \in \overline{V}_{h}.
\end{align}

The discrete gradient operator $\overline{\nabla}_{h}: W^{1,1}(\mathcal{T}_{h}) \rightarrow [\overline{V}_{h}]^{d}$ 
is defined as 
\begin{align}
\label{discrete_grad}
\left( \overline{\nabla}_{h} v\right)_{i} = \overline{\partial}_{h, x_{i}}v, \qquad \forall i \in \{1,\cdots, d\}.
\end{align}
\end{definition}
We would like to point out that the discrete gradient operator (\ref{discrete_grad}) 
has been introduced in \cite{LakkisPryer2011}. It has been used in \cite{BHW2021,DednerPryer2022,FLW2022} as well.

\subsection{Finite element method for linear elliptic PDE in non-divergence form and HJB equation}  
\label{sec_fem}

\subsubsection{FEM for linear elliptic PDE in non-divergence form} 

We define a function $\gamma$ on $\Omega$ to be 
\begin{align}
\label{def_gamma}
\gamma(\boldx) = \dfrac{\text{Tr}A(\boldx)}{\vert A(\boldx)\vert^{2}}
 = \dfrac{\Sigma_{1\leq i \leq d}a_{ii}(\boldx)}{\Sigma_{1\leq i,j\leq d}(a_{ij}(\boldx))^{2}}, 
\qquad \forall \boldx \in \Omega.
\end{align}

Obviously, the linear elliptic PDE in non-divergence form (\ref{nondiv_pde_original}) is equivalent to 
\begin{align}
\label{nondiv_pde}
\tilde{A}:D^{2}u + \tilde{\boldsymbol{b}}\cdot \nabla u 
+ \tilde{c}u = \tilde{f}  \text{ in } \Omega, \qquad
u = 0 \text{ on } \partial\Omega,
\end{align}
where $\tilde{A} = \gamma A$, $\tilde{\boldsymbol{b}} = \gamma
\boldsymbol{b}$, $\tilde{c} = \gamma c$ and $\tilde{f} = \gamma f$.

The finite element method for linear elliptic PDE in non-divergence form (\ref{nondiv_pde_original}) is to find 
$u_{h} \in V_{h}$ satisfying 
\begin{align}
\label{nondiv_fem}
& \int_{\Omega}(\tilde{A}: \overline{\nabla}_{h}(\nabla u_{h}) + \tilde{\boldsymbol{b}}\cdot \nabla u_{h} 
+ \tilde{c}u_{h} ) v_{h} d\boldx \\
\nonumber 
= & \int_{\Omega} \gamma (A: \overline{\nabla}_{h}(\nabla u_{h}) + \boldsymbol{b}\cdot \nabla u_{h} 
+ cu_{h} ) v_{h} d\boldx 
= \int_{\Omega} \gamma fv_{h} d\boldx = \int_{\Omega} \tilde{f}v_{h} d\boldx, \qquad \forall v_{h} \in V_{h}.
\end{align}

\subsubsection{FEM for the Hamilton-Jacobi-Bellman equation}

For any $\alpha \in \Lambda$, we define a function $\gamma^{\alpha}$ 
on $\Omega$ to be
\begin{align}
\label{def_gamma_alpha}
\gamma^{\alpha}(\boldx) = \dfrac{\text{Tr}A^{\alpha}(\boldx)}{\vert A^{\alpha}(\boldx)\vert}
=\dfrac{\Sigma_{1 \leq i \leq d}  a_{ii}^{\alpha}(\boldx) }
{\Sigma_{1\leq i,j \leq d} (a_{ij}^{\alpha}(\boldx))^{2}},\qquad \forall \boldx \in \Omega.
\end{align} 

The finite element method for the Hamilton-Jacobi-Bellman equation (\ref{hjb_eqs_original}) is 
to find $u_{h}\in V_{h}$ satisfying 
\begin{align}
\label{hjb_fem}
(\sup_{\alpha\in \Lambda}\gamma^{\alpha}[A^{\alpha}:\overline{\nabla}_{h}
(\nabla u_{h}) + \boldsymbol{b}^{\alpha}\cdot\nabla u_{h} 
+ c^{\alpha}u_{h} - f^{\alpha}], v_{h})_{\Omega} = 0, 
\qquad \forall v_{h} \in V_{h}.
\end{align}

\subsection{Preliminary results}
\label{sec_preliminary}

\begin{lemma}
\label{lemma_Poisson_regularity}
If the domain $\Omega$ is a convex polyhedra in $\mathbb{R}^{d}$ ($d=2,3$), 
let $p \in (1,2]$. If the domain is a Lipschitz polygon in $\mathbb{R}^{2}$, 
let $p \in (1,\frac{4}{3} + \epsilon_{2})$ where $\epsilon_{2}>0$ depends only on 
$\Omega$. Then for any $g \in L^{p}(\Omega)$, 
there is a unique solution $w \in W^{2,p}(\Omega) \cap W_{0}^{1,p}(\Omega)$ such that 
$\Delta w = g$ in $\Omega$ and 
\begin{align}
\label{Poisson_regularity_ineq1}
\Vert w\Vert_{W^{2,p}(\Omega)} \leq \underline{C}_{p} \Vert g\Vert_{L^{p}(\Omega)}.
\end{align}
Here the constant $\underline{C}_{p}$ may depend on $p$.

If the domain $\Omega$ is a convex polyhedra in $\mathbb{R}^{d}$ ($d=2,3$), 
let $p \in (1,+\infty)$. If the domain is a Lipschitz polygon in $\mathbb{R}^{2}$, 
let $p \in (\frac{4}{3} - \epsilon_{1}, 3 + \epsilon_{3})$ where $\epsilon_{1}, 
\epsilon_{3}>0$ depend only on $\Omega$. Then for any $g \in W^{-1,p}(\Omega)$, 
there is a unique solution $w \in W_{0}^{1,p}(\Omega)$ such that 
$\Delta w = g$ in $\Omega$ and 
\begin{align}
\label{Poisson_regularity_ineq2}
\Vert w\Vert_{W^{1,p}(\Omega)} \leq \underline{C}_{p}^{\prime} 
\Vert g\Vert_{W^{-1,p}(\Omega)}.
\end{align}
Here the constant $\underline{C}_{p}^{\prime}$ may depend on $p$.

In addition, if $\Omega$ is a convex polyhedra in $\mathbb{R}^{d}$ ($d=2,3$) 
and $g \in L^{2}(\Omega)$, we have 
\begin{align}
\label{Poisson_regularity_ineq3}
\Vert D^{2} w\Vert_{L^{2}(\Omega)} \leq \Vert g\Vert_{L^{2}(\Omega)}.
\end{align}
\end{lemma}

\begin{proof}
The proof of existence of a unique solution and the estimate 
(\ref{Poisson_regularity_ineq1}) are due to 
\cite[Proposition~$2.1$]{ChibaSaito2019} ($d=2$) and 
\cite[($1.6$)]{Fromm1992} ($d=3$). 
The proof of existence of a unique solution and the estimate 
(\ref{Poisson_regularity_ineq2}) are due to 
\cite[Theorem~$0.5$]{JerisonKenig1995} (non-convex and $d=2$) and 
\cite[($1.5$)]{Fromm1992} and \cite[Theorem~$2.3$]{Wood2006} with $\alpha = 0$ (convex). 
The estimate (\ref{Poisson_regularity_ineq3}) 
is due to \cite[Lemma~$1$]{SmearsSuli2013}.
\end{proof}

According to \cite[Theorem~$3$ and Theorem~$4$]{CrouzeixThomee1987} (see \cite[Lemma~$7.3$]{Wahlbin1991} as well), 
we have the following Lemma~\ref{lemma_L2_proj_props} since we assume the mesh $\mathcal{T}_{h}$ is quasi-uniform. 
We would like to point out that the proof of \cite[Theorem~$4$]{CrouzeixThomee1987} can be easily extended to obtain 
(\ref{L2_proj_prop2}) in three dimensional domains.

\begin{lemma}
\label{lemma_L2_proj_props}
We denote by $P_{h}$ the standard $L^{2}$-orthogonal projection onto $V_{h}$. Then 
there is a constant $\tilde{C}_{1}$ such that for any $q \in [1, +\infty)$,
\begin{subequations}
\label{L2_proj_props}
\begin{align}
\label{L2_proj_prop1}
& \Vert P_{h} v\Vert_{L^{q}(\Omega)} \leq \tilde{C}_{1} \Vert v\Vert_{L^{q}(\Omega)} 
\qquad \forall v \in L^{q}(\Omega), \\
\label{L2_proj_prop2}
& \Vert P_{h} v\Vert_{W^{1,q}(\Omega)} \leq \tilde{C}_{1} \Vert v\Vert_{W^{1,q}(\Omega)} 
\qquad \forall v \in W_{0}^{1,q}(\Omega).
\end{align}
\end{subequations}
\end{lemma}

According to \cite[Lemma~$4.3$]{ErnGuermond2017}, we have the following Lemma~\ref{lemma_H1_interpolation}.
\begin{lemma}
\label{lemma_H1_interpolation} 
For any $\overline{v}_{h} \in \overline{V}_{h}$, we denote by $\chi_{h} \in H^{1}(\Omega)\cap 
\overline{V}_{h}$ the standard $H^{1}$-conforming averaging of $\overline{v}_{h}$. Then 
there is a constant $C$ such that for any $q \in [1,+\infty]$, we have  
\begin{align*}
\Vert \chi_{h} - \overline{v}_{h} \Vert_{L^{q}(\Omega)} 
\leq C h \left( \Sigma_{F \in \mathcal{F}_{h}^{I}} h_{F}^{1-q}\Vert [ \overline{v}_{h}] \Vert_{L^{q}(F)}^{q}
 \right)^{\frac{1}{q}}.
\end{align*}
Here the domain $\Omega$ can be any Lipschitz polyhedra in $\mathbb{R}^{d}$ ($d=2,3$).
\end{lemma}

\begin{lemma}
\label{lemma_discrete_norm_control}
There is a constant $C$ such that for any $q\in [1,+\infty]$, we have 
\begin{align*}
\Vert \nabla v_{h} \Vert_{L^{q}(\Omega)} \leq C \Vert v_{h}\Vert_{W_{h}^{2,q}(\Omega)}, 
\qquad \forall v_{h} \in V_{h}.
\end{align*}
Here the domain $\Omega$ can be any Lipschitz polyhedra in $\mathbb{R}^{d}$ ($d=2,3$).
\end{lemma}

\begin{proof}
We choose $v_{h}\in V_{h}$ and $q\in [1,+\infty]$ arbitrarily. We denote by $C$ a constant independent 
of $v_{h}$ and $q \in [1,+\infty]$. According to Lemma~\ref{lemma_H1_interpolation}, there is 
$\boldsymbol{\psi}_{h}\in [H^{1}(\Omega)\cap P_{r}(\mathcal{T}_{h})]^{d}$ such that 
\begin{align}
\label{H1_interpolation_ineq1}
\Vert \boldsymbol{\psi}_{h} - \nabla v_{h} \Vert_{L^{q}(\Omega)} 
\leq C h \left( \Sigma_{F \in \mathcal{F}_{h}^{I}} h_{F}^{1-q}\Vert [ \nabla v_{h}] 
\Vert_{L^{q}(F)}^{q} \right)^{\frac{1}{q}}.
\end{align}  

Since the proof for the case $d=2$ is similar but simpler, we assume $d=3$ in the following. 
Since $v_{h} = 0$ on $\partial\Omega$, $\boldsymbol{n}\times (\nabla v_{h}) = \boldsymbol{0}$ on $\partial\Omega$. 
Then by (\ref{H1_interpolation_ineq1}), we have 
\begin{align}
\label{H1_interpolation_ineq2}
& \Vert \boldsymbol{n} \times \boldsymbol{\psi}_{h} \Vert_{L^{q}(\partial\Omega)} 
= \Vert \boldsymbol{n} \times (\boldsymbol{\psi}_{h}-\nabla v_{h} )\Vert_{L^{q}(\partial\Omega)} \\
\nonumber
\leq & C h^{1-\frac{1}{q}}\left( \Sigma_{F \in \mathcal{F}_{h}^{I}} h_{F}^{1-q}\Vert [ \nabla v_{h}] 
\Vert_{L^{q}(F)}^{q} \right)^{\frac{1}{q}} \leq C \Vert v_{h}\Vert_{W_{h}^{2,q}(\Omega)}.
\end{align}

Let $\{ \psi_{h}^{(i)} \}_{i=1}^{3}$ be the collection of all components of $\boldsymbol{\psi}_{h}$.
Since the domain $\Omega$ is a polyhedra, we can rotate the coordinate system such that there is a face 
$\tilde{F}$ on $\partial\Omega$ which is parallel to both $x_{1}$-axis and $x_{2}$-axis. 
Then by (\ref{H1_interpolation_ineq2}), we have 
\begin{align}
\label{H1_interpolation_ineq3}
\Vert \psi_{h}^{(1)} \Vert_{L^{q}(\tilde{F})} + \Vert \psi_{h}^{(2)} \Vert_{L^{q}(\tilde{F})} 
\leq C \Vert v_{h}\Vert_{W_{h}^{2,q}(\Omega)}.
\end{align}

In the following, we will prove
\begin{align}
\label{H1_interpolation_ineq4}
\Vert \psi_{h}^{(i)} \Vert_{L^{q}(\Omega)} \leq C \left( \Vert \psi_{h}^{(i)} \Vert_{L^{q}(\tilde{F})} 
+ \Vert \nabla \psi_{h}^{(i)} \Vert_{L^{q}(\Omega)} \right), \qquad \forall i\in \{1,2\}.
\end{align}
In fact, if (\ref{H1_interpolation_ineq4}) holds, then (\ref{H1_interpolation_ineq1},\ref{H1_interpolation_ineq3}) imply 
$\Vert \dfrac{\partial v_{h}}{\partial x_{1}} \Vert_{L^{q}(\Omega)} 
+ \Vert \dfrac{\partial v_{h}}{\partial x_{2}} \Vert_{L^{q}(\Omega)} 
\leq C \Vert v_{h}\Vert_{W_{h}^{2,q}(\Omega)}$.
Then by applying the same argument to another face on $\partial\Omega$ which is not parallel to $\tilde{F}$, we will have 
$\Vert \nabla v_{h} \Vert_{L^{q}(\Omega)} \leq C \Vert v_{h}\Vert_{W_{h}^{2,q}(\Omega)}$.

Therefore, we only need to prove (\ref{H1_interpolation_ineq4}). For the sake of simplicity, we will only prove 
(\ref{H1_interpolation_ineq4}) for $i=1$. We decompose the domain $\Omega$ into finitely many non-overlapping simplexes 
$\{ T_{i} \}_{i=1}^{N}$, such that $\overline{\Omega} = \cup_{i=1}^{N} \overline{T}_{i}$. We assume this decomposition 
of $\Omega$ is conforming such that it is like a conforming mesh of $\Omega$.
We would like to point out that this decomposition of $\Omega$ is not relevant to the mesh $\mathcal{T}_{h}$, such that 
$N$ is independent of $h$. If we take an element $K\in\mathcal{T}_{h}$ arbitrarily, $K$ may have non-trivial overlapping 
with two or more simplexes among $\{ T_{i} \}_{i=1}^{N}$.

Without losing of generality, we assume that a triangular face $F_{1}$ of $T_{1}$ is contained in $\tilde{F}$.
Let $T$ be an arbitrary simplex among $\{ T_{i} \}_{i=2}^{N}$. By relabelling $\{ T_{i} \}_{i=2}^{N}$, there is a 
positive integer $1< \tilde{N} \leq N$ such that $\partial T_{1} \cap \partial T_{2} = F_{2}$, 
$\partial T_{2} \cap \partial T_{3} = F_{3}$,$\cdots$,$\partial T_{\tilde{N}-1}\cap \partial T_{\tilde{N}} = F_{\tilde{N}}$ 
where $F_{2},\cdots,F_{\tilde{N}}$ are triangular faces.
We claim that for any $1\leq i \leq N$,
\begin{align}
\label{H1_interpolation_ineq5} 
\Vert \psi_{h}^{(1)} \Vert_{L^{q}(T_{i})} + \Vert \psi_{h}^{(1)} \Vert_{L^{q}(F_{i}^{\prime\prime})} 
\leq C \left( \Vert \psi_{h}^{(1)} \Vert_{L^{q}(F_{i}^{\prime})} 
+ \Vert \nabla \psi_{h}^{(1)} \Vert_{L^{q}(T_{i})} \right),
\end{align}
where $F_{i}^{\prime}$ and $F_{i}^{\prime\prime}$ can be any two triangular faces of $T_{i}$. 
If (\ref{H1_interpolation_ineq5}) holds, then (\ref{H1_interpolation_ineq4}) is true for $i=1$ 
because (\ref{H1_interpolation_ineq3}) holds and $N$ is independent of $h$.

Therefore, we only need to prove (\ref{H1_interpolation_ineq5}). Since $\psi_{h}^{(1)} \in H^{1}(\Omega)\cap 
P_{r}(\mathcal{T}_{h})$, $\psi_{h}^{(1)}$ is globally Lipschitz on $\overline{\Omega}$. So 
$\psi_{h}^{(1)}$ is globally Lipschitz on $\overline{T_{i}}$ for any $1\leq i \leq N$. 
We notice that the Lipschitz constant of $\psi_{h}^{(i)}$ on $\overline{T}_{i}$ is bounded by 
(independent of $h$) $\Vert \nabla \psi_{h}^{(i)} \Vert_{L^{\infty}(T_{i})}$ for any $1\leq i \leq N$. 
So (\ref{H1_interpolation_ineq5}) holds for $q = + \infty$. 

Now we are going to prove (\ref{H1_interpolation_ineq5}) for $q \in [1,+\infty)$. Via affine transformation, 
we can assume $T_{i}$ is the reference simplex determined by the four vertexes 
$(0,0,0)$, $(1,0,0)$, $(0,1,0)$ and $(0,0,1)$.
We further assume $F_{i}^{\prime}$ is the triangular face through $(0,0,0)$, $(1,0,0)$ and $(0,1,0)$, 
and $F_{i}^{\prime\prime}$ is the triangular face through $(1,0,0)$, $(0,1,0)$ and $(0,0,1)$. 
Since $\psi_{h}^{(i)}$ is globally Lipschitz on $\overline{T}_{i}$, we have that for any $(x_{1},x_{2},x_{3})
\in \overline{T}_{i}$, 
\begin{align}
\label{H1_interpolation_ineq6}
\psi_{h}^{(1)}(x_{1},x_{2},x_{3}) = \psi_{h}^{(1)}(x_{1},x_{2},0) 
+ \int_{0}^{x_{3}} \dfrac{\partial \psi_{h}^{(1)}}{\partial x_{3}}(x_{1},x_{2},s) ds.
\end{align}
Since $\psi_{h}^{(1)} \in H^{1}(\Omega)\cap P_{r}(\mathcal{T}_{h})$, we have that for any 
$(x_{1},x_{2},0) \in \overline{F}_{i}^{\prime}$, $\dfrac{\partial \psi_{h}^{(1)}}{\partial x_{3}}(x_{1},x_{2},s)$ 
on the right hand side of (\ref{H1_interpolation_ineq6}) is the same as the weak derivative of 
$\psi_{h}^{(i)}$ with respect to $x_{3}$ on the line segment $l(x_{1},x_{2}):=\{ (x_{1},x_{2},s): 0< s < 1 - x_{1} - x_{2} \}$ 
except at most finitely many points. In fact, these points are intersections between the line segment $l(x_{1},x_{2})$ 
and interior mesh interfaces in $\mathcal{F}_{h}^{I}$ which are not parallel to $l(x_{1},x_{2})$.  
Therefore, (\ref{H1_interpolation_ineq6}) implies that for any $q \in [1,+\infty)$, 
\begin{align}
\label{H1_interpolation_ineq7}
& \Vert \psi_{h}^{(1)} \Vert_{L^{q}(F_{i}^{\prime\prime})} \leq  
\left(\sqrt{3} \int_{F_{i}^{\prime}} \vert \psi_{h}^{(1)}(x_{1},x_{2},1-x_{1}-x_{2}) \vert^{q}dx_{1}dx_{2} 
\right)^{\frac{1}{q}} \\ 
\nonumber 
\leq & \sqrt{3} \left(\int_{F_{i}^{\prime}} \vert \psi_{h}^{(1)}(x_{1},x_{2},1-x_{1}-x_{2}) 
\vert^{q}dx_{1}dx_{2} \right)^{\frac{1}{q}} \\
\nonumber 
\leq & \sqrt{3} \left(\int_{F_{i}^{\prime}} \vert \psi_{h}^{(1)}(x_{1},x_{2},0) 
\vert^{q}dx_{1}dx_{2} \right)^{\frac{1}{q}} \\
\nonumber
& \qquad + \sqrt{3} \left( \int_{F_{i}^{\prime}} \left| \int_{0}^{1-x_{1}-x_{2}} 
\dfrac{\partial \psi_{h}^{(1)}}{\partial x_{3}}(x_{1},x_{2},s) ds \right|^{q} dx_{1}dx_{2} \right)^{\frac{1}{q}} \\
\nonumber 
\leq & \sqrt{3} \left(\int_{F_{i}^{\prime}} \vert \psi_{h}^{(1)}(x_{1},x_{2},0) 
\vert^{q}dx_{1}dx_{2} \right)^{\frac{1}{q}} \\
\nonumber
& \qquad + \sqrt{3} \left( \int_{F_{i}^{\prime}}  \int_{0}^{1-x_{1}-x_{2}} \left|
\dfrac{\partial \psi_{h}^{(1)}}{\partial x_{3}}(x_{1},x_{2},s) \right|^{q} ds  dx_{1}dx_{2} \right)^{\frac{1}{q}}, 
\end{align}
and 
\begin{align}
\label{H1_interpolation_ineq8}
& \Vert \psi_{h}^{(1)} \Vert_{L^{q}(T_{i})} = \left( \int_{F_{i}^{\prime}} \int_{0}^{1-x_{1}-x_{2}} 
\vert \psi_{h}^{(1)}(x_{1},x_{2},s) \vert^{q}ds dx_{1}dx_{2} \right)^{\frac{1}{q}} \\
\nonumber
\leq & \left( \int_{F_{i}^{\prime}} \int_{0}^{1-x_{1}-x_{2}} 
\vert \psi_{h}^{(1)}(x_{1},x_{2},0) \vert^{q}ds dx_{1}dx_{2} \right)^{\frac{1}{q}} \\
\nonumber 
& \qquad + \left( \int_{F_{i}^{\prime}} \int_{0}^{1-x_{1}-x_{2}} 
\left| \int_{0}^{s} \dfrac{\partial \psi_{h}^{(1)}}{\partial x_{3}}(x_{1},x_{2},t) dt \right|^{q} 
ds dx_{1}dx_{2} \right)^{\frac{1}{q}} \\
\nonumber
\leq & \left( \int_{F_{i}^{\prime}} \vert \psi_{h}^{(1)}(x_{1},x_{2},0) \vert^{q}
dx_{1}dx_{2} \right)^{\frac{1}{q}} \\
\nonumber 
& \qquad + \left( \int_{F_{i}^{\prime}} \int_{0}^{1-x_{1}-x_{2}} 
 \left| \dfrac{\partial \psi_{h}^{(1)}}{\partial x_{3}}(x_{1},x_{2},s)\right|^{q}  
ds dx_{1}dx_{2} \right)^{\frac{1}{q}}.
\end{align}

We notice that (\ref{H1_interpolation_ineq7},\ref{H1_interpolation_ineq8}) are valid 
on the reference simplex via an affine transformation from $T_{i}$. 
We pull back (\ref{H1_interpolation_ineq7},\ref{H1_interpolation_ineq8}) to $T_{i}$ 
via the inverse of this affine transformation. Then we have 
\begin{align*}
& \Vert \psi_{h}^{(1)} \Vert_{L^{q}(F_{i}^{\prime\prime})} 
\leq C \left( \Vert \psi_{h}^{(1)} \Vert_{L^{q}(F_{i}^{\prime})} + \Vert \nabla \psi_{h}^{(1)}\Vert_{L^{q}(T_{i})}\right), \\
& \Vert \psi_{h}^{(1)} \Vert_{L^{q}(T_{i})} \leq C \left( \Vert \psi_{h}^{(1)} \Vert_{L^{q}(F_{i}^{\prime})} 
+ \Vert \nabla \psi_{h}^{(1)}\Vert_{L^{q}(T_{i})}\right).
\end{align*}
Here the constant $C$ may depend on $T_{i}$.

Therefore, we have that (\ref{H1_interpolation_ineq5}) for $q \in [1,+\infty]$. 
We can conclude that the proof is complete.
\end{proof}

\section{Linear elliptic PDE in non-divergence form} 
\label{sec_nondiv}

For the linear elliptic PDE in non-divergence form (\ref{nondiv_pde_original}), 
the main theoretical results are Theorem~\ref{thm_A_continuous_complete}, 
Theorem~\ref{thm_pde_cordes} and Theorem~\ref{thm_global_estimate_cordes}. 

Theorem~\ref{thm_A_continuous_complete} 
provides the well-posedness of strong solution in $W^{2,p}(\Omega)$ and optimal convergence 
with respect to $W_{h}^{2,p}$-norm (see (\ref{discrete_second_order_norm})) of the numerical 
solution of the FEM (\ref{nondiv_fem}), when the coefficient matrix 
$A \in C^{0}(\overline{\Omega})$. 

Theorem~\ref{thm_pde_cordes} and Theorem~\ref{thm_global_estimate_cordes} provides 
well-posedness of strong solution in $W^{2,p}(\Omega)$ of (\ref{nondiv_pde_original}) and 
optimal convergence with respect to $W_{h}^{2,p}$-norm of the numerical 
solution of the FEM (\ref{nondiv_fem}), in the setting where the coefficient matrix $A$
is discontinuous but $\gamma A$ is dominated by $I_{d}$, where the positive weight function 
$\gamma$ is introduced in (\ref{def_gamma}).
Since they are special examples of Theorem~\ref{thm_hjb_pde_wellposedness} and Theorem~\ref{thm_hjb_conv} 
respectively, the proofs of Theorem~\ref{thm_pde_cordes} and Theorem~\ref{thm_global_estimate_cordes} are skipped.
Please refer to those of Theorem~\ref{thm_hjb_pde_wellposedness} and Theorem~\ref{thm_hjb_conv}.

\subsection{Global $W^{2,p}$ estimate of linear elliptic PDE in non-divergence form 
with uniformly continuous $A$}  

\begin{lemma}
\label{lemma_pde_unique1}
(Uniqueness of linear elliptic PDE in non-divergence form)
Let $A \in [C^{0}(\overline{\Omega})]^{3\times 3}$ uniformly elliptic, $\boldsymbol{b}\in 
[L^{\infty}(\Omega)]^{3}$, $c \in L^{\infty}(\Omega)$ with $c \leq 0$, 
and $\Omega$ is an open bounded convex domain in $\mathbb{R}^{3}$. 
Let $1 < p \leq \frac{3}{2}$. 
If $\phi \in W^{2,p}(\Omega) \cap W_{0}^{1,p}(\Omega)$ such that 
\begin{align*}
A:D^{2} \phi + \boldsymbol{b}\cdot \nabla \phi + c \phi = 0 \text{ in } \Omega,
\end{align*}
then $\phi = 0$ in $\Omega$.
\end{lemma}

\begin{proof}
We choose $\boldx_{0} \in \overline{\Omega}$ arbitrarily. For any $R>0$, we denote by $\eta$ the cut-off function to be 
$1$ inside of $\mathcal{B}_{R}(\boldx_{0}):= \{ \boldx \in \mathbb{R}^{3}: \vert \boldx - \boldx_{0}\vert < R \}$ and 
to be $0$ outside of $\mathcal{B}_{2R}(\boldx_{0}):= \{ \boldx \in \mathbb{R}^{3}: \vert \boldx - \boldx_{0}\vert < 2R \}$.

We define $v = \eta \phi$. We notice that in $\Omega$, 
\begin{align*}
& A:D^{2}v = A:D^{2}(\eta \phi) \\ 
= & \eta A:D^{2}\phi + 2 A:(\nabla \eta \otimes \nabla \phi) + \phi A:D^{2} \eta \\ 
= & - \eta \left( \boldsymbol{b}\cdot \nabla \phi + c \phi \right)
+ 2 A:(\nabla \eta \otimes \nabla \phi) + \phi A:D^{2} \eta,
\end{align*}
since $A:D^{2}\phi + \boldsymbol{b}\cdot \nabla \phi + c \phi = 0$ in $\Omega$. 
Therefore, by Sobolev's embedding inequality, we have that 
\begin{align}
\label{improved_regularity1}
A:D^{2}v \in L^{q}(\Omega) \text{ with } q := \min (\dfrac{3p}{3-p}, 2) > \frac{3}{2}.
\end{align}
Obviously, $1 < p < q$.

We define $A_{0} = A(\boldx_{0}) \in \mathbb{R}^{3 \times 3}$. 
We define the linear mapping $M: L^{\tilde{q}}(\Omega) \rightarrow W^{2,\tilde{q}}(\Omega) 
\cap W_{0}^{1,\tilde{q}}(\Omega)$ by 
\begin{align*}
\nabla \cdot \left( A_{0} \nabla M(f) \right) = f,\qquad \forall f \in L^{\tilde{q}}(\Omega), 
\forall 1 < \tilde{q} \leq q.
\end{align*} 
By Lemma~\ref{lemma_Poisson_regularity}, 
the mapping $M$ is not only well-defined but also continuous 
for any $1 < \tilde{q} \leq q$.

It is easy to see that in $\Omega$
\begin{align*}
A_{0}:D^{2} v = (A_{0} - A):D^{2}v + A:D^{2}v.
\end{align*}
Since $1 < p < q$, we can 
apply the mapping $M$ on both sides of the above equation, to obtain that  
\begin{align*}
v = M \left( (A_{0} - A):D^{2}v \right) + M\left( A:D^{2} v \right) \text{ in } \Omega.
\end{align*}

We define $B_{2R}(\boldx_{0}):=\{ \boldx\in \Omega: \vert \boldx - \boldx_{0} \vert < 2R\} 
= \mathcal{B}_{2R}(\boldx_{0}) \cap \Omega$. We define a linear mapping 
$T : W^{2,\tilde{q}}(B_{2R}(\boldx_{0})) \rightarrow W^{2,\tilde{q}}(B_{2R}(\boldx_{0}))$ by 
\begin{align}
\label{map_T}
Tw = \left( M \tilde{w}\right)|_{B_{2R}(\boldx_{0})}
\end{align}
where $w \in W^{2,\tilde{q}}(B_{2R}(\boldx_{0}))$ and 
\begin{align}
\label{tilde_w_T}
\tilde{w} = (A_{0} - A):D^{2} w \text{ in } B_{2R}(\boldx_{0}),\qquad \tilde{w} = 0 \text{ in } 
\Omega \setminus \overline{B_{2R}(\boldx_{0})}.
\end{align}
Here $1 < \tilde{q} \leq q$. 
Therefore, for any $ 1 < \tilde{q} \leq q$, we have that 
\begin{align*}
& \Vert Tw \Vert_{W^{2,\tilde{q}}(B_{2R}(\boldx_{0}))} 
\leq C_{\tilde{q}} \Vert \tilde{w}\Vert_{L^{\tilde{q}}(\Omega)} \\
= & C_{\tilde{q}} \Vert (A_{0} - A):D^{2} w \Vert_{L^{\tilde{q}}(B_{2R}(\boldx_{0}))} \\
\leq & C_{\tilde{q}} \Vert A_{0} - A \Vert_{L^{\infty}(B_{2R}(\boldx_{0}))} 
\Vert w\Vert_{W^{2,\tilde{q}}(B_{2R}(\boldx_{0}))}, \qquad \forall w \in W^{2,\tilde{q}}(B_{2R}(\boldx_{0})).
\end{align*}
The constant $C_{\tilde{q}}$ is due to Lemma~\ref{lemma_Poisson_regularity}.

From now on, we only consider $\tilde{q} = p$ and $\tilde{q}= q = \min (\dfrac{3p}{3-p}, 2)$.
Since $A \in [C^{0}(\overline{\Omega})]^{3 \times 3}$, 
there is a $R_{0}=R_{0}(\boldx_{0}) > 0$ such that $\Vert T\Vert \leq \frac{1}{2}$ from 
$W^{2,\tilde{q}}(B_{2R_{0}}(\boldx_{0}))$ to $W^{2,\tilde{q}}(B_{2R_{0}}(\boldx_{0}))$ for  
both $\tilde{q} = p$ and $\tilde{q}= q$. We would like to point out 
that $R_{0}$ may depend on $\boldx_{0}$. 
By (\ref{improved_regularity1}) and Lemma~\ref{lemma_Poisson_regularity}, 
$M\left( A:D^{2} v \right) \in L^{\tilde{q}}(\Omega)$ 
for any $1 < \tilde{q} \leq q$. 
Therefore for both $\tilde{q} = p$ and $\tilde{q}= q$, there is 
a unique $w \in W^{2,\tilde{q}}(B_{2R_{0}}(\boldx_{0}))$ such that 
\begin{align}
\label{w_T_unique}
w = T w + M\left( A:D^{2} v \right)|_{B_{2R_{0}(\boldx_{0})}} \quad \text{ in }B_{2R_{0}(\boldx_{0})}.
\end{align}

On the other hand, since $v = 0$ in $\Omega \setminus \overline{B_{2R}(\boldx_{0})}$ 
and $v \in W^{2,p}(\Omega)$ ($1< p < q$), (\ref{map_T}) and (\ref{tilde_w_T}) imply that 
\begin{align*}
T (v|_{B_{2R_{0}}(\boldx_{0})}) = M\left( (A_{0} - A):D^{2} v \right)|_{B_{2R_{0}}(\boldx_{0})}.
\end{align*}
Then we have 
\begin{align}
\label{map_T_contraction}
& v|_{B_{2R_{0}}(\boldx_{0})} = M(A_{0}:D^{2}v)|_{B_{2R_{0}}(\boldx_{0})} \\ 
\nonumber 
= & T (v|_{B_{2R_{0}}(\boldx_{0})}) + M\left( A:D^{2} v \right)|_{B_{2R_{0}}(\boldx_{0})} 
\quad \text{ in } B_{2R_{0}}(\boldx_{0}).
\end{align} 
(\ref{w_T_unique}) and (\ref{map_T_contraction}) imply that $\phi |_{B_{R_{0}}(\boldx_{0})} 
\in W^{2,q}(B_{R_{0}}(\boldx_{0}))$. 

Since $\boldx_{0} \in \overline{\Omega}$ is chosen arbitrarily and $\overline{\Omega}$ is compact,  
$\Omega = \cup_{\boldx_{0}\in \mathcal{Q}}B_{R_{0}(\boldx_{0})}(\boldx_{0})$ where $\mathcal{Q}$ is 
a collection of finitely many points in $\overline{\Omega}$.
Therefore, we can conclude that $\phi \in W^{2,q}(\Omega)$ 
where $q = \min (\dfrac{3p}{3 - p} ,2) > \frac{3}{2}$. So $\phi \in C^{0}(\overline{\Omega})$. Furthermore, by the interior estimate in 
\cite[Lemma~$9.16$]{GT01}, we have $\phi \in W_{\text{loc}}^{2,3}(\Omega)$ since 
$A:D^{2} \phi + \boldsymbol{b}\cdot \nabla \phi + c \phi = 0$ in $\Omega$.
Therefore, by the ABP estimate \cite[Theorem~$9.1$]{GT01}, we can conclude 
that $\phi = 0$ \text{ in } $\Omega$.
\end{proof}

\begin{theorem}
\label{thm_global_A_continuous}
(Global $W^{2,p}$ estimate of linear elliptic PDE in non-divergence form)
Let $A \in [C^{0}(\overline{\Omega})]^{d\times d}$, $\boldsymbol{b}\in 
[L^{\infty}(\Omega)]^{d}$, $c \in L^{\infty}(\Omega)$ with $c \leq 0$, 
 and $\Omega$ is a bounded open Lipschitz polyhedral domain 
in $\mathbb{R}^{d}$ ($d=2,3$).  We further assume $A$ is uniformly elliptic in $\Omega$.
 There is a constant $C_{p}>0$ such that 
\begin{align}
\label{global_w2p_estimate_A_continuous}
\Vert w\Vert_{W^{2,p}(\Omega)} \leq C_{p} \Vert A:D^{2}w  + \boldsymbol{b}\cdot 
\nabla w + c w\Vert_{L^{p}(\Omega)}, 
\qquad \forall w \in W^{2,p}(\Omega) \cap W_{0}^{1,p}(\Omega).
\end{align}
Here $p$ is valid in the range described in (\ref{p_range1}).
\end{theorem}

\begin{proof}
When $d=2$, $W^{2,p}(\Omega) \Subset C^{0}(\overline{\Omega})$ due to 
Sobolev's embedding theorem. Furthermore,  by the interior estimate in 
\cite[Lemma~$9.16$]{GT01}, 
we have $u \in W_{\text{loc}}^{2,2}(\Omega)$. Therefore, 
Theorem~\ref{thm_global_A_continuous} can be proven by mimicking the proofs 
of \cite[Theorem~$1.1$]{QT2023} and \cite[Corollary~$1.2$]{QT2023} with the global 
$W^{2,p}$ estimate of Poisson equation (see \cite[Proposition~$2.1$]{ChibaSaito2019}) 
and the Alexandroff–Bakelman–Pucci estimate estimate \cite[Theorem~$9.1$]{GT01}.

Now we focus on the case $d=3$. Due to \cite[Corollary~$1.2$]{QT2023}, Theorem~\ref{thm_global_A_continuous} 
holds if we further assume $\Omega$ is a convex polyhedra in $\mathbb{R}^{3}$ and $\frac{3}{2} < p \leq 2$. 
In fact, if we replace \cite[Corollary~$3.12$]{Dauge1992} by
(\ref{Poisson_regularity_ineq1}) of Lemma~\ref{lemma_Poisson_regularity} 
in the proof of \cite[Theorem~$1.1$]{QT2023} (right above $(3.13)$ of 
the proof of \cite[Proposition~$3.2$]{QT2023}), it is easy to have that  
for any $1 < p \leq 2$ and for any $w \in W^{2,p}(\Omega) 
\cap W_{0}^{1,p}(\Omega)$,
\begin{align}
\label{global_estimate1_A_continuous}
\Vert w\Vert_{W^{2,p}(\Omega)} \leq C_{p}^{\prime} \left( \Vert A:D^{2}w  
+ \boldsymbol{b}\cdot\nabla w + cw\Vert_{L^{p}(\Omega)} 
+ \Vert w \Vert_{L^{p}(\Omega)} \right).
\end{align}  
Here $C_{p}^{\prime}$ depends on $1 < p \leq 2$ and $\Omega$ is an open bounded convex domain in $\mathbb{R}^{3}$.

As the proof of \cite[Corollary~$1.2$]{QT2023}, Theorem~\ref{thm_global_A_continuous} will be proven 
via (\ref{global_estimate1_A_continuous}) and 
the proof of contraction with the uniqueness of solution in $W^{2,p}(\Omega) \cap W_{0}^{1,p}(\Omega)$: 
\begin{align*}
A:D^{2}w + \boldsymbol{b}\cdot \nabla w + cw = 0 \text{ in } \Omega 
\text{ with } w \in W^{2,p}(\Omega) \cap W_{0}^{1,p}(\Omega)   
\Longrightarrow u = 0 \text{ in } \Omega.
\end{align*}

When $\frac{3}{2} < p \leq 2$, it is straightforward by the Alexandroff–Bakelman–Pucci estimate estimate with the fact 
$W^{2,p}(\Omega) \Subset C^{0}(\overline{\Omega})$ (see the proof of \cite[Corollary~$1.2$]{QT2023}). 
Therefore, in order to complete the proof of Theorem~\ref{thm_global_A_continuous}, we only need 
the uniqueness for $1<p \leq \frac{3}{2}$, which is provided by Lemma~\ref{lemma_pde_unique1}.
Therefore, we can conclude that the proof of Theorem~\ref{thm_global_A_continuous} is complete.
\end{proof}

\subsection{Stability analysis of FEM for linear elliptic PDE in non-divergence form with uniformly continuous $A$}

\subsubsection{Stability estimates for the linear elliptic PDE in non-divergence form 
with constant $A$ and $\boldsymbol{b}=\boldsymbol{0}, c = 0$}
\label{sec_stability_constant_A} 

All results in Section~\ref{sec_stability_constant_A} can be found in \cite{FHM2017}. 
But the results in \cite{FHM2017} are only valid on domains with $C^{1,1}$ boundary. 

We denote by $\tilde{A}_{0} \in \mathbb{R}^{d\times d}$ a symmetric and positive definite matrix.

According to Definition~\ref{def_discrete_deri} and the fact $V_{h} \subset \overline{V}_{h}$, we have
\begin{align}
\label{nondiv_fem_constant_A}
\int_{\Omega}(\tilde{A}_{0}: \overline{\nabla}_{h}(\nabla w_{h}))v_{h} d\boldx 
= -(\tilde{A}_{0}\nabla w_{h}, \nabla v_{h})_{\Omega}, \qquad \forall w_{h},v_{h}\in V_{h}.
\end{align}

We define a linear operator $\mathcal{L}_{\tilde{A}_{0},h}: W_{h}^{1} \rightarrow V_{h}$ as 
\begin{align}
\label{def_L_A_constant_numerical}
(\mathcal{L}_{\tilde{A}_{0},h}w, v_{h})_{\Omega} = \int_{\Omega}(\tilde{A}_{0}: \overline{\nabla}_{h}(\nabla w))v_{h} d\boldx,  
\qquad \forall w \in W_{h}^{1}, v_{h} \in V_{h}.
\end{align}
Here, we recall that the piecewise Sobolev space $W_{h}^{1} =W^{2,1}(\mathcal{T}_{h}) \cap W_{0}^{1,1}(\Omega)$.

\begin{lemma}
\label{lem_w1p_estimate}
There is a positive constant $C$ such that for any $h \in (0, h_{0})$,
\begin{align}
\label{gloabl_w1p_estiamte_constant_A}
\Vert w_{h} - w\Vert_{W^{1,p}(\Omega)}\leq C 
\inf_{\chi_{h} \in V_{h}}\Vert \chi_{h} - w\Vert_{W^{1,p}(\Omega)}, 
\qquad \forall w \in W_{0}^{1,p}(\Omega).
\end{align}
Here $w_{h} \in V_{h}$ satisfying 
\begin{align*}
(\tilde{A}_{0} \nabla w_{h}, v_{h})_{\Omega} = 
(\tilde{A}_{0} \nabla w, v_{h})_{\Omega}, 
\qquad \forall v_{h} \in V_{h}.
\end{align*}
If the domain $\Omega$ is a convex polyhedra in $\mathbb{R}^{d}$ ($d=2,3$), the estimate (\ref{gloabl_w1p_estiamte_constant_A}) 
holds for any $p \in (1,+\infty]$. If the domain is a Lipschitz polygon in 
$\mathbb{R}^{2}$, the estimate 
(\ref{gloabl_w1p_estiamte_constant_A}) holds for any $p \in (\frac{4}{3} - \epsilon_{1}, +\infty]$ where $\epsilon_{1}$ 
is the same as that in Lemma~\ref{lemma_Poisson_regularity}. 
The constants $C$ and $h_{0}$ may depend on $\tilde{A}_{0}$.
\end{lemma}

\begin{proof}
Obviously, we have 
\begin{align}
\label{H1_estimate}
\Vert \nabla (w_{h} - w)\Vert_{L^{2}(\Omega)} \leq C \Vert \nabla(\chi_{h} - w) 
\Vert_{L^{2}(\Omega)}, \qquad \forall \chi_{h} \in V_{h}.
\end{align}
By \cite[Theorem~$2$]{GLRS2009} ($d=3$), \cite[Theorem~$2$]{Schatz1980} 
($d=2$), 
\begin{align}
\label{W1infinity_estimate}
\Vert \nabla (w_{h} - w)\Vert_{L^{\infty}(\Omega)} \leq C \Vert \nabla(\chi_{h} - w) 
\Vert_{L^{\infty}(\Omega)}, \qquad \forall \chi_{h} \in V_{h},
\end{align}
if $h <h_{0}$. Here the constant $C$ in (\ref{H1_estimate},\ref{W1infinity_estimate}) may depend on $\tilde{A}_{0}$.
Therefore, according to (\ref{H1_estimate},\ref{W1infinity_estimate}), 
there is a constant $C$ such that for any $2 \leq p \leq +\infty$, if $h<h_{0}$,
\begin{align}
\label{W1p_estimate_1}
\Vert w_{h} - w \Vert_{W^{1,p}(\Omega)} \leq C \Vert \chi_{h} - w \Vert_{W^{1,p}(\Omega)}, 
\qquad \forall \chi_{h} \in V_{h}.
\end{align}

Now, for any $p \in (1,2)$, we denote by $p^{\prime} \geq 2$ such that $\frac{1}{p} + \frac{1}{p^{\prime}}$. We choose $\chi_{h} \in V_{h}$ arbitrarily. 
Then by Lemma~\ref{lemma_Poisson_regularity}, 
we have for any $p \in (1,2)$ if $\Omega$ is a convex polyhedra in $\mathbb{R}^{d}$ 
($d=2,3$) and for any $p \in (\frac{4}{3}-\epsilon_{1},2]$ if $\Omega$ is a 
Lipschitz polygon in $\mathbb{R}^{2}$,
\begin{align*}
& \Vert w_{h} - \chi_{h}\Vert_{W^{1,p}(\Omega)} 
\leq C \Vert \nabla \cdot ( \tilde{A}_{0}\nabla (w_{h} 
- \chi_{h}) ) \Vert_{W^{-1,p}(\Omega)} \\
= & C \sup_{\tilde{v} \in W_{0}^{1,p^{\prime}}(\Omega)} \dfrac{-( \nabla \cdot ( \tilde{A}_{0}\nabla (w_{h} - \chi_{h}), \tilde{v}))_{\Omega}}
{\Vert \tilde{v}\Vert_{W^{1,p^{\prime}}(\Omega)}}
= C \sup_{\tilde{v} \in W_{0}^{1,p^{\prime}}(\Omega)} \dfrac{( \tilde{A}_{0}\nabla (w_{h} - \chi_{h}), \nabla \tilde{v})_{\Omega}}
{\Vert \tilde{v}\Vert_{W^{1,p^{\prime}}(\Omega)}} \\
= & C \sup_{\tilde{v} \in W_{0}^{1,p^{\prime}}(\Omega)} \dfrac{( \tilde{A}_{0}\nabla (w_{h} - \chi_{h}), \nabla \tilde{v}_{h})_{\Omega}}
{\Vert \tilde{v}\Vert_{W^{1,p^{\prime}}(\Omega)}} 
= C \sup_{\tilde{v} \in W_{0}^{1,p^{\prime}}(\Omega)} \dfrac{( \tilde{A}_{0}\nabla 
(w - \chi_{h}), \nabla \tilde{v}_{h})_{\Omega}}
{\Vert \tilde{v}\Vert_{W^{1,p^{\prime}}(\Omega)}},
\end{align*}
where $\tilde{v}_{h} \in V_{h}$ satisfying 
\begin{align*}
(\tilde{A}_{0}\nabla \tilde{v}_{h}, \nabla \tilde{w}_{h})_{\Omega} = (\tilde{A}_{0}\nabla \tilde{v}, \nabla \tilde{w}_{h}), 
\qquad \forall \tilde{w}_{h} \in V_{h}.
\end{align*}
By (\ref{W1p_estimate_1}), 
\begin{align*}
\Vert \tilde{v}_{h}\Vert_{W^{1,p^{\prime}}(\Omega)} \leq C \Vert \tilde{v}\Vert_{W^{1,p^{\prime}}(\Omega)}.
\end{align*}
So we have 
\begin{align*}
\Vert w_{h}  - \chi_{h}\Vert_{W^{1,p}(\Omega)} \leq C 
\Vert w - \chi_{h}\Vert_{W^{1,p}(\Omega)}.
\end{align*}
Therefore, for any $p \in (1,2)$, we also have 
that if $h<h_{0}$,
\begin{align}
\label{W1p_estimate_2}
\Vert w_{h} - w \Vert_{W^{1,p}(\Omega)} \leq C \Vert \chi_{h} - w \Vert_{W^{1,p}(\Omega)}, 
\qquad \forall \chi_{h} \in V_{h}.
\end{align}

(\ref{W1p_estimate_1}, \ref{W1p_estimate_2}) imply that for any $p\in (1,+\infty]$, 
\begin{align*}
\Vert w_{h} - w \Vert_{W^{1,p}(\Omega)} \leq C \inf_{\chi_{h}\in V_{h}} 
\Vert \chi_{h} - w \Vert_{W^{1,p}(\Omega)}.
\end{align*}
Therefore, the proof of the estimates (\ref{gloabl_w1p_estiamte_constant_A}) is complete.
\end{proof}

\begin{lemma}
\label{lemma_gloabl_w2p_estiamte_constant_A}
Let $\mathcal{L}_{\tilde{A}_{0},h}$ be the linear operator defined 
in (\ref{def_L_A_constant_numerical}). 
There is a positive constant $C$ such that for all $h \in (0, h_{0})$,
\begin{align}
\label{gloabl_w2p_estiamte_constant_A}
\Vert w_{h}\Vert_{W_{h}^{2,p}(\Omega)} \leq C \Vert \mathcal{L}_{\tilde{A}_{0},h} w_{h}\Vert_{L^{p}(\Omega)}, 
\qquad \forall w_{h} \in V_{h}.
\end{align}
Here $p$ is valid in the range described in (\ref{p_range2}).
The constants $C$ and $h_{0}$ may depend $\tilde{A}_{0}$.
\end{lemma}

\begin{proof}
We choose $w_{h} \in V_{h}$ arbitrarily. We notice that $\mathcal{L}_{\tilde{A}_{0},h}w_{h} \in V_{h}$.
We denote by $w\in H_{0}^{1}(\Omega)$ satisfying 
\begin{align*}
\tilde{A}_{0}: D^{2} w = \nabla \cdot \left( \tilde{A}_{0} \nabla w\right) = \mathcal{L}_{\tilde{A}_{0},h}w_{h} 
\quad \text{ in } \Omega.
\end{align*}

According to (\ref{nondiv_fem_constant_A}), 
\begin{align}
\label{fem_projection}
(\tilde{A}_{0}\nabla w_{h}, \nabla v_{h})_{\Omega} 
= (\tilde{A}_{0}\nabla w, \nabla v_{h})_{\Omega}, 
\qquad \forall v_{h} \in V_{h}.
\end{align}

Then by (\ref{fem_projection}),  the construction of $w$ 
and Lemma~\ref{lem_w1p_estimate}, we have that
for any $p\in (1,+\infty]$ and $0< h < h_{0}$, 
\begin{align}
\label{W1p_estimate_constant_A_complete}
\Vert w_{h} - w \Vert_{W^{1,p}(\Omega)} \leq C \inf_{\chi_{h}\in V_{h}} 
\Vert \chi_{h} - w \Vert_{W^{1,p}(\Omega)}\leq C h \Vert w\Vert_{W^{2,p}(\Omega)}.
\end{align}
(\ref{W1p_estimate_constant_A_complete}) and Lemma~\ref{lemma_Poisson_regularity} 
imply that for any $p \in (1, 2]$ if $\Omega$ is a convex polyhedra in 
$\mathbb{R}^{d}$ ($d=2,3$) and for any $p \in (\frac{4}{3}-\epsilon_{1}, 
\frac{4}{3}) + \epsilon_{2}$, if $h < h_{0}$,
\begin{align}
\label{W1p_estimate_constant_A} 
\Vert w_{h} - w \Vert_{W^{1,p}(\Omega)} 
\leq C h \Vert w\Vert_{W^{2,p}(\Omega)}
\leq CC^{\prime} h \Vert \mathcal{L}_{\tilde{A}_{0},h}w_{h}\Vert_{L^{p}(\Omega)}.
\end{align}
Here the constant $C^{\prime}$ may also depend on $\tilde{A}_{0}$.

We denote by $\boldsymbol{\phi}_{h}$ the standard $L^{2}$-orthogonal projection 
of $\nabla w$ into $[P_{0}(\mathcal{T}_{h})]^{d}$. 
Then by the discrete trace inequality, we have 
\begin{align}
\label{w1p_estiamte_constant_A_projection}
& \Vert \boldsymbol{\phi}_{h} - \nabla w \Vert_{L^{p}(\Omega)} 
\leq C h \Vert w\Vert_{W^{2,p}(\Omega)} 
\leq C C^{\prime} h \Vert \mathcal{L}_{\tilde{A}_{0},h}w_{h}\Vert_{L^{p}(\Omega)},\\
\nonumber
& \left( \Sigma_{F \in \mathcal{F}_{h}^{I}} h_{F}^{1-p}\Vert [\boldsymbol{\phi}_{h}] \Vert_{L^{p}(F\cap 
\overline{D})}^{p} \right)^{\frac{1}{p}}
\leq C \Vert w\Vert_{W^{2,p}(\Omega)} 
\leq C C^{\prime} \Vert \mathcal{L}_{\tilde{A}_{0},h}w_{h}\Vert_{L^{p}(\Omega)}. 
\end{align}

Then by (\ref{W1p_estimate_constant_A},\ref{w1p_estiamte_constant_A_projection}) and 
the discrete inverse inequality, we can conclude that the proof is complete.
\end{proof}

\begin{lemma}
\label{lemma_local_estimate_constant_A}
For $\boldx_{0} \in \overline{\Omega}$ and $R > 0$, define 
\begin{align*}
B_{R}(\boldx_{0}) := \{ \boldx \in \Omega: \vert \boldx - \boldx_{0} \vert < R \} \subset \Omega.
\end{align*}
Let $R^{\prime} = R + d$ with $d \geq 2h$. 
Then there is a positive constant $C$ such that 
\begin{align}
\label{local_estimate_constant_A}
\Vert w_{h}\Vert_{W_{h}^{2,p}(B_{R}(\boldx_{0}))} \leq C \Vert \mathcal{L}_{\tilde{A}_{0},h}w_{h}
\Vert_{L_{h}^{p}(B_{R^{\prime}}(\boldx_{0}))},\qquad \forall w_{h} \in V_{h}(B_{R}(\boldx_{0})),
\end{align}
for any $h \in (0, h_{0})$ where $h_{0}$ is introduced in Lemma~\ref{lemma_gloabl_w2p_estiamte_constant_A}.

Here $p$ is valid in the range described in (\ref{p_range2}).
The constant $C$ is independent of the choice of $R$ and $R^{\prime}$. 
But the constant $C$ may depend on $\tilde{A}_{0}$.
\end{lemma}

\begin{proof}
By (\ref{nondiv_fem_constant_A}), $\mathcal{L}_{\tilde{A}_{0},h}w_{h}$ is the same as 
$\mathcal{L}_{0,h}$ in \cite[Lemma~$2.7$]{FHM2017}. Therefore, by replacing \cite[Lemma~$2.6$]{FHM2017}
by Lemma~\ref{lemma_gloabl_w2p_estiamte_constant_A}, the proof in \cite[Lemma~$2.7$]{FHM2017} 
is also valid for this Lemma. It is easy to verify that the proof of \cite[Lemma~$2.7$]{FHM2017} 
is independent of the choice of $\boldx_{0}$, $R$ and $R^{\prime}$.
\end{proof}

\subsubsection{Stability estimates for linear elliptic PDE in non-divergence form with uniformly continuous $A$}

Analogues of Lemma~\ref{lemma_diff_A_variable_constant}, Lemma~\ref{lemma_local_estimate_variable_A} 
and Lemma~\ref{lemma_gloabl_estimate_with_reaction} can be found in \cite{FHM2017}.

We define $\mathcal{L}_{\tilde{A},h}: W_{h}^{1} \rightarrow V_{h}$ as 
\begin{align}
\label{L_h_operator}
(\mathcal{L}_{\tilde{A},h}w, v_{h})_{\Omega} = \int_{\Omega}(\tilde{A}: \overline{\nabla}_{h}(\nabla u_{h}))v_{h} d\boldx, 
\qquad \forall v_{h} \in V_{h}.
\end{align}
Here we recall that $W_{h}^{1} = W_{0}^{1,1}(\Omega)\cap W^{2,1}(\mathcal{T}_{h})$.

For $\boldx_{0} \in \overline{\Omega}$ and $R > 0$, we define 
\begin{align*}
B_{R}(\boldx_{0}) := \{ \boldx \in \Omega: \vert \boldx - \boldx_{0} \vert < R \} \subset \Omega.
\end{align*}

\begin{lemma}
\label{lemma_diff_A_variable_constant}
We assume $\tilde{A} \in [C^{0}(\overline{\Omega})]^{d \times d}$.
Let $p \in (1,+\infty)$. For any $\delta > 0$, there exists $R_{\delta}>0$ and $h_{\delta} > 0$ such that for 
any $\boldx_{0}\in \overline{\Omega}$ with $\tilde{A}_{0} = \tilde{A}(\boldx_{0})$, 
\begin{align*}
\Vert (\mathcal{L}_{\tilde{A},h} - \mathcal{L}_{\tilde{A}_{0},h})w\Vert_{L_{h}^{p}(B_{R_{\delta}}(\boldx_{0}))} 
\leq \delta \Vert w\Vert_{W_{h}^{2,p}(B_{R_{\delta}}(\boldx_{0}))}, \qquad \forall 
w \in W_{h}^{(p)}, h\leq h_{\delta}. 
\end{align*}
\end{lemma}

\begin{proof}
We denote by $B_{t}:= \{ \boldx \in \Omega: \vert \boldx - \boldx_{0} \vert < t \} \subset \Omega$ 
for any $t>0$. For any $v_{h} \in V_{h}(B_{t})$, we have 
\begin{align*}
((\mathcal{L}_{\tilde{A},h} - \mathcal{L}_{\tilde{A}_{0},h})w, v_{h})_{\Omega} 
= \int_{\Omega}\left((\tilde{A}_{0} - \tilde{A}):\overline{\nabla}_{h}(\nabla w)\right)v_{h} d\boldx.
\end{align*}
We recall that $V_{h}(B_{t})=\{ v \in V_{h}: v|_{\Omega \setminus B_{t}} = 0\}$. 
We define $\mathcal{T}_{t,h} = \{ K \in \mathcal{T}_{h}: \mathcal{H}^{d}(K \setminus B_{t}) = 0 \}$, 
where $\mathcal{H}^{d}(\cdot)$ is the $d$-dimensional Hausdorff measure.

We define $p^{\prime}\in (1, + \infty)$ such that $\frac{1}{p} + \frac{1}{p^{\prime}} = 1$.  So we have 
\begin{align*}
& ((\mathcal{L}_{\tilde{A},h} - \mathcal{L}_{\tilde{A}_{0},h})w, v_{h})_{\Omega} 
= \Sigma_{K \in \mathcal{T}_{t,h}}\int_{K}\left((\tilde{A}_{0} - \tilde{A}):\overline{\nabla}_{h}(\nabla w)\right)
v_{h} d\boldx \\
\leq & \Vert \tilde{A}_{0} - \tilde{A} \Vert_{L^{\infty}(B_{t})} 
\left( \Sigma_{K \in \mathcal{T}_{t,h}} \Vert \overline{\nabla}_{h}(\nabla w)\Vert_{L^{p}(K)}^{p} \right)^{\frac{1}{p}} 
\Vert v_{h}\Vert_{L^{p^{\prime}}(\Omega)}.
\end{align*}

We denote by $\mathcal{F}_{t,h}$ the collection of all mesh interface of $\mathcal{T}_{t,h}$.  
By the construction of $\mathcal{T}_{t,h}$, we have  
\begin{align*}
\mathcal{H}^{d-1}(F \cap \overline{B_{t}}) = \mathcal{H}^{d-1}(F), 
\qquad \forall F \in \mathcal{F}_{t,h}.
\end{align*}
Then with Definition~\ref{def_discrete_deri} and (\ref{discrete_second_order_norm}), we have 
\begin{align*}
& \left(\Sigma_{K \in \mathcal{T}_{t,h}} \Vert \overline{\nabla}_{h}(\nabla w)\Vert_{L^{p}(K)}^{p}\right)^{\frac{1}{p}} \\ 
\leq & C_{*} \left(\Sigma_{K \in \mathcal{T}_{t,h}} \Vert D_{h}^{2} w\Vert_{L^{p}(K)}^{p}
+ \Sigma_{F \in (\mathcal{F}_{t,h} \cap \mathcal{F}_{h}^{I}) }h_{F}^{1-p}\Vert [\nabla w]\Vert_{L^{p}(F)}^{p} 
\right)^{\frac{1}{p}} 
\leq C_{*} \Vert w\Vert_{W_{h}^{2,p}(B_{t})},
\end{align*} 
where $\mathcal{F}_{h}^{I}$ is the collection of all interior mesh faces of $\mathcal{T}_{h}$.
The constant $C_{*}$ is independent of $\boldx_{0}$ and $t$. 

We choose $t_{0}>0$ small enough, which is independent of $h$, such that 
\begin{align*}
\Vert \tilde{A}_{0} - \tilde{A} \Vert_{L^{\infty}(B_{t_{0}})} \leq \dfrac{\delta}{C_{*} + 1}.
\end{align*}
Since $\tilde{A}\in [C(\overline{\Omega})]^{d \times d}$, $t_{0}$ can be chosen to be independent 
of $\boldx_{0} \in \overline{\Omega}$.

Then we have 
\begin{align*}
\vert ((\mathcal{L}_{\tilde{A},h} - \mathcal{L}_{\tilde{A}_{0},h})w, v_{h})_{\Omega} \vert
\leq \delta \Vert w\Vert_{W_{h}^{2,p}(B_{t_{0}})} \Vert v_{h}\Vert_{L^{p^{\prime}}(\Omega)}, 
\qquad \forall v_{h} \in V_{h}(B_{t_{0}}).
\end{align*}
We choose $h_{\delta} = \frac{t_{0}}{4}$ such that for any $0<h\leq h_{\delta}$, $V_{h}(B_{t_{0}})$ is not empty. 
Then we can conclude that the proof is complete.
\end{proof}

\begin{lemma}
\label{lemma_local_estimate_variable_A}
We assume $\tilde{A} \in [C^{0}(\overline{\Omega})]^{d \times d}$ is uniformly elliptic.
There exists $R_{1} > 0$ and $h_{1} > 0$ such that for any $\boldx_{0} \in \overline{\Omega}$, 
\begin{align}
\label{local_estimate_variable_A}
\lambda_{0}\Vert w_{h}\Vert_{W_{h}^{2,p}(B_{R_{1}}(\boldx_{0}))} 
\leq \Vert \mathcal{L}_{\tilde{A},h}w_{h} \Vert_{L_{h}^{p}(B_{2R_{1}}(\boldx_{0}))}, 
\qquad \forall w_{h} \in V_{h}(B_{R_{1}}(\boldx_{0})), h < h_{1}.
\end{align}

Here $p$ is valid in the range described in (\ref{p_range2}).
The constant $\lambda_{0}$ may depend on $\tilde{A}(\boldx_{0})$.
\end{lemma}

\begin{proof}
The proof is an immediate consequence of Lemma~\ref{lemma_local_estimate_constant_A} 
and Lemma~\ref{lemma_diff_A_variable_constant}.
\end{proof}

\begin{lemma}
\label{lemma_gloabl_estimate_with_reaction}
If $\tilde{A} \in [C^{0}(\overline{\Omega})]^{d\times d}$ is uniformly elliptic, then 
there is $h_{2} > 0$ which may depend on $p$, such that for all $h \in (0, h_{2})$,
\begin{align}
\label{gloabl_estimate_with_reaction}
\Vert w_{h} \Vert_{W_{h}^{2,p}(\Omega)} \leq C \left( \Vert \mathcal{L}_{\tilde{A},h}w_{h}\Vert_{L_{h}^{p}(\Omega)} 
+ \Vert w_{h}\Vert_{W^{1,p}(\Omega)} \right), \qquad \forall w_{h} \in V_{h}.
\end{align}
Here $p$ is valid in the range described in (\ref{p_range2}).
\end{lemma}

\begin{proof}
We divided the proof into two steps.

Step $1$. For any $\boldx_{0} \in \overline{\Omega}$, we denote by $B_{R}:=\{ \boldx \in \Omega: 
\vert \boldx - \boldx_{0} \vert < R \}$ for any $R > 0$. Let $\eta \in C^{3}(\mathbb{R}^{d})$ be 
a cut-off function satisfying 
\begin{align}
\label{cut_off_func}
0\leq \eta \leq 1, \quad \eta|_{B_{\frac{4}{3}R}} = 1, \quad \eta|_{\mathbb{R}^{d} \setminus B_{\frac{5}{3}R}} = 0, 
\quad \Vert \eta \Vert_{W^{m,\infty}(\mathbb{R}^{d})} \leq C R^{-m} (m=0,1,2).
\end{align} 

From now on, we assume $h < \frac{R}{24}$. We define $\tilde{A}_{0} = \tilde{A}(\boldx_{0})\in \mathbb{R}^{d \times d}$. 

We choose $w_{h} \in V_{h}$ arbitrarily. We notice that $\eta w_{h} \in W_{h}^{(p)}(B_{\frac{5}{3}R})$ 
and $I_{h}(\eta w_{h}) \in V_{h}(B_{\frac{11}{6}R})=\{ v_{h} \in V_{h} : v_{h}|_{\Omega \setminus B_{\frac{11}{6} R}} = 0 \}$. 
Here $I_{h}: C^{0}(\overline{\Omega})\rightarrow V_{h}$ is the standard nodal interpolant onto $V_{h}$.

We notice that $I_{h}(\eta v_{h}) = \eta v_{h} = 0$ in $\Omega \setminus B_{\frac{11}{6}R}$ for any $v_{h} \in V_{h}$. 
Thus \cite[Lemma~$2.4$]{FHM2017} implies that for any $q \in (1, +\infty)$, there is a constant $C_{q}$ such that 
\begin{align}
\label{nodal_interpolation_error}
& \Vert I_{h}(\eta v_{h}) - \eta v_{h} \Vert_{L^{q}(\Omega)} \leq C_{q} h R^{-1} \Vert v_{h}\Vert_{L^{q}(B_{2R})},\\
\nonumber
& \Vert \nabla (I_{h}(\eta v_{h}) -\eta v_{h} ) \Vert_{L^{q}(\Omega)} 
\leq C_{q}R^{-1} \Vert v_{h} \Vert_{L^{q}( B_{2R} )}, \\
\nonumber
& \Vert I_{h}(\eta v_{h}) - \eta v_{h}\Vert_{W_{h}^{2,q}(\Omega)} 
\leq C_{q} R^{-2}\Vert v_{h} \Vert_{W^{1,q}(B_{2R})},
\quad \forall v_{h} \in V_{h}.
\end{align}

By Lemma~\ref{lemma_local_estimate_variable_A}, there are two constants $\lambda_{0}>0$ and $R > 0$ 
(we choose $R$ to be $R_{1}$ in Lemma~\ref{lemma_local_estimate_variable_A}) such that 
\begin{align}
\label{gloabl_estimate_with_reaction_ineq1}
\lambda_{0}\Vert w_{h} \Vert_{W_{h}^{p}(B_{R})}  
= \lambda_{0}\Vert I_{h}(\eta w_{h}) \Vert_{W_{h}^{p}(B_{R})} 
\leq \Vert \mathcal{L}_{\tilde{A},h}(I_{h}(\eta w_{h})) \Vert_{L_{h}^{p}(B_{2R})}. 
\end{align}
Here the constant $\lambda_{0}$ may depend on $\tilde{A}(\boldx_{0})$. 

By (\ref{gloabl_estimate_with_reaction_ineq1},\ref{discrete_zero_order_norm}) 
and Definition~\ref{def_discrete_deri}, we have 
\begin{align*}
& \lambda_{0}\Vert w_{h} \Vert_{W_{h}^{p}(B_{R})}  
= \lambda_{0}\Vert I_{h}(\eta w_{h}) \Vert_{W_{h}^{p}(B_{R})} \\
\leq & \Vert \mathcal{L}_{\tilde{A},h}(\eta w_{h}) \Vert_{L_{h}^{p}(B_{2R})} 
+ \Vert \mathcal{L}_{\tilde{A},h}(I_{h}(\eta w_{h})-\eta w_{h}) \Vert_{L_{h}^{p}(B_{2R})} \\
\leq & \Vert \mathcal{L}_{\tilde{A},h}(\eta w_{h}) \Vert_{L_{h}^{p}(B_{2R})} 
+ \Vert \mathcal{L}_{\tilde{A},h}(I_{h}(\eta w_{h})-\eta w_{h}) \Vert_{L^{p}(\Omega)} \\
\leq & \Vert \mathcal{L}_{\tilde{A},h}(\eta w_{h}) \Vert_{L_{h}^{p}(B_{2R})} 
+ \Vert \tilde{A} \Vert_{L^{\infty}(\Omega)}\Vert \overline{\nabla}_{h}\nabla 
(I_{h}(\eta w_{h})-\eta w_{h})\Vert_{L^{p}(\Omega)} \\
\leq & \Vert \mathcal{L}_{\tilde{A},h}(\eta w_{h}) \Vert_{L_{h}^{p}(B_{2R})} 
+ C\Vert \tilde{A} \Vert_{L^{\infty}(\Omega)}\Vert I_{h}(\eta w_{h})-\eta w_{h}\Vert_{W_{h}^{2,p}(\Omega)}
\end{align*}
Then by (\ref{nodal_interpolation_error}), we have 
\begin{align}
\label{gloabl_estimate_with_reaction_ineq2}
\lambda_{0}\Vert w_{h} \Vert_{W_{h}^{p}(B_{R})} 
\leq \Vert \mathcal{L}_{\tilde{A},h}(\eta w_{h}) \Vert_{L_{h}^{p}(B_{2R})} 
+ C R^{-2}\Vert w_{h}\Vert_{W^{1,p}(B_{2R})}.
\end{align}

We define $p^{\prime} \in (1, + \infty)$ such that $\frac{1}{p} + \frac{1}{p^{\prime}}$. 
By (\ref{discrete_zero_order_norm}), 
\begin{align*}
\Vert \mathcal{L}_{\tilde{A},h}(\eta w_{h}) \Vert_{L_{h}^{p}(B_{2R})} 
= \sup_{v_{h} \in V_{h}(B_{2R})} \dfrac{(\mathcal{L}_{\tilde{A},h}(\eta w_{h}), v_{h})_{\Omega}}
{\Vert v_{h}\Vert_{L^{p^{\prime}}(B_{2R})}},
\end{align*}
where $V_{h}(B_{2R}) = \{ v_{h} \in V_{h} : v_{h}|_{\Omega \setminus B_{2R}} = 0 \}$. 
We choose $v_{h} \in V_{h}(B_{2R})$ such that 
\begin{align}
\label{local_norm_variable_A}
\Vert \mathcal{L}_{\tilde{A},h}(\eta w_{h}) \Vert_{L_{h}^{p}(B_{2R})} 
= (\mathcal{L}_{\tilde{A},h}(\eta w_{h}), v_{h})_{\Omega}, \quad 
\Vert v_{h}\Vert_{L^{p^{\prime}}(B_{2R})} = \Vert v_{h}\Vert_{L^{p^{\prime}}(\Omega)} = 1.
\end{align}

By (\ref{local_norm_variable_A},\ref{L_h_operator}) and Definition~\ref{def_discrete_deri}, we have 
\begin{align}
\label{gloabl_estimate_with_reaction_ineq3}
& \Vert \mathcal{L}_{\tilde{A},h}(\eta w_{h}) \Vert_{L_{h}^{p}(B_{2R})} 
= (\tilde{A}:\overline{\nabla}_{h}(\nabla (\eta w_{h})), v_{h})_{\Omega} \\ 
\nonumber 
= & (\tilde{A}:\overline{\nabla}_{h}(\eta \nabla  w_{h}), v_{h})_{\Omega} 
+ (\tilde{A}:\overline{\nabla}_{h}(w_{h}\nabla \eta ), v_{h})_{\Omega}\\ 
\nonumber 
\leq & (\tilde{A}:\overline{\nabla}_{h}(\eta \nabla  w_{h}), v_{h})_{\Omega} 
+ \Vert \tilde{A}\Vert_{L^{\infty}(\Omega)} 
\Vert \overline{\nabla}_{h}(w_{h}\nabla \eta )\Vert_{L^{p}(\Omega)} \\ 
\nonumber
\leq & (\tilde{A}:\overline{\nabla}_{h}(\eta \nabla  w_{h}), v_{h})_{\Omega} 
+ C \Vert w_{h}\nabla \eta\Vert_{W_{h}^{2,p}(\Omega)} \\ 
\nonumber 
\leq & (\tilde{A}:\overline{\nabla}_{h}(\eta \nabla  w_{h}), v_{h})_{B_{2R}} 
+ C R^{-2}\Vert w_{h}\Vert_{W^{1,p}(\Omega)}.
\end{align}

Therefore, (\ref{gloabl_estimate_with_reaction_ineq2},\ref{gloabl_estimate_with_reaction_ineq3}) imply that 
\begin{align}
\label{gloabl_estimate_with_reaction_ineq4}
\lambda_{0}\Vert w_{h} \Vert_{W_{h}^{p}(B_{R})} 
\leq & (\tilde{A}:\overline{\nabla}_{h}(\eta \nabla  w_{h}), v_{h})_{\Omega} 
+ C R^{-2}\Vert w_{h}\Vert_{W^{1,p}(B_{2R})}.
\end{align}

We denote by $\overline{\eta}_{h} \in P_{0}(\mathcal{T}_{h})$ the element-wise average of $\eta$.
Then we have 
\begin{align}
\label{gloabl_estimate_with_reaction_ineq5}
& (\tilde{A}:\overline{\nabla}_{h}(\eta \nabla  w_{h}), v_{h})_{\Omega} \\
\nonumber
= & (\tilde{A}:\overline{\nabla}_{h}( (\eta - \overline{\eta}_{h} ) \nabla  w_{h}), v_{h})_{\Omega} 
+ (\tilde{A}:\left(\overline{\nabla}_{h}( \overline{\eta}_{h}  \nabla  w_{h}) 
- \overline{\eta}_{h}\overline{\nabla}_{h}( \nabla  w_{h})\right), v_{h})_{\Omega} \\ 
\nonumber
& \qquad + (\tilde{A}:\overline{\nabla}_{h}( \nabla  w_{h}), \overline{\eta}_{h} v_{h} - I_{h}(\eta v_{h}))_{\Omega}
+ (\tilde{A}:\overline{\nabla}_{h}( \nabla  w_{h}), I_{h}(\eta v_{h}))_{\Omega}
\end{align} 
 
We give estimates of all four terms on the right hand side of (\ref{gloabl_estimate_with_reaction_ineq5}) in the following. 

By (\ref{local_norm_variable_A}), we have 
\begin{align*}
(\tilde{A}:\overline{\nabla}_{h}( (\eta - \overline{\eta}_{h} ) \nabla  w_{h}), v_{h})_{\Omega} 
\leq \Vert \tilde{A} \Vert_{L^{\infty}(\Omega)} \Vert \overline{\nabla}_{h}( (\eta - \overline{\eta}_{h} ) 
\nabla  w_{h}) \Vert_{L^{p}(\Omega)}
\end{align*}
Since $\overline{\nabla}_{h}( (\eta - \overline{\eta}_{h} ) \nabla  w_{h}) \in [\overline{V}_{h}]^{d \times d}$, 
we denote by $\underline{\varphi}_{h} \in [\overline{V}_{h}]^{d \times d}$ such that 
\begin{align*}
\Vert \overline{\nabla}_{h}( (\eta - \overline{\eta}_{h} ) \nabla  w_{h}) \Vert_{L^{p}(\Omega)} 
= \left( \overline{\nabla}_{h}( (\eta - \overline{\eta}_{h} ) \nabla  w_{h}),  \underline{\varphi}_{h}\right)_{\Omega}, 
\qquad \Vert \underline{\varphi}_{h} \Vert_{L^{p^{\prime}}(\Omega)} = 1.
\end{align*} 
By Definition~\ref{def_discrete_deri} and the fact
$\eta = \overline{\eta}_{h} = 0$ in $\Omega \setminus B_{\frac{11}{6}R}$, we have 
\begin{align*}
&\left( \overline{\nabla}_{h}( (\eta - \overline{\eta}_{h} ) \nabla  w_{h}),  \underline{\varphi}_{h}\right)_{\Omega} 
= - \langle [(\eta - \overline{\eta}_{h} ) \nabla  w_{h}]\otimes \boldsymbol{n}, \{ \underline{\varphi}_{h} \} 
\rangle_{\mathcal{F}_{h}^{I}} 
+ (\nabla ((\eta - \overline{\eta}_{h} ) \nabla  w_{h}), \underline{\varphi}_{h})_{\mathcal{T}_{h}} \\
= & - \langle [\eta - \overline{\eta}_{h}]\{ \nabla  w_{h}\} \otimes \boldsymbol{n}, \{ \underline{\varphi}_{h} \} 
\rangle_{\mathcal{F}_{h}^{I}} 
- \langle \{\eta - \overline{\eta}_{h}\}  [\nabla  w_{h}]\otimes \boldsymbol{n}, \{ \underline{\varphi}_{h} \} 
\rangle_{\mathcal{F}_{h}^{I}}
+ (\nabla ((\eta - \overline{\eta}_{h} ) \nabla  w_{h}), \underline{\varphi}_{h})_{\mathcal{T}_{h}} \\
\leq & C \left( R^{-1} \Vert w_{h} \Vert_{W^{1,p}(B_{\frac{11}{6}R})} 
+ hR^{-1}\Vert w_{h}\Vert_{W_{h}^{2,p}(B_{\frac{11}{6} R})}\right) \Vert \underline{\varphi}_{h} 
\Vert_{L^{p^{\prime}}(\Omega)}.
\end{align*}
By \cite[Lemma~$2.2$]{FHM2017}, we have 
\begin{align*}
\Vert w_{h}\Vert_{W_{h}^{2,p}(B_{\frac{11}{6} R})} \leq Ch^{-1} \Vert w_{h}\Vert_{W^{1,p}(B_{2R})}.
\end{align*} 
Therefore, we have 
\begin{align}
\label{gloabl_estimate_with_reaction_ineq6}
(\tilde{A}:\overline{\nabla}_{h}( (\eta - \overline{\eta}_{h} ) \nabla  w_{h}), v_{h})_{\Omega}
\leq C R^{-1} \Vert w_{h} \Vert_{W^{1,p}(B_{2R})}.
\end{align}

By (\ref{local_norm_variable_A}), we have 
\begin{align*}
(\tilde{A}:\left(\overline{\nabla}_{h}( \overline{\eta}_{h}  \nabla  w_{h}) 
- \overline{\eta}_{h}\overline{\nabla}_{h}( \nabla  w_{h})\right), v_{h})_{\Omega} 
\leq \Vert \tilde{A}\Vert_{L^{\infty}(\Omega)} 
\Vert \overline{\nabla}_{h}( \overline{\eta}_{h}  \nabla  w_{h}) 
- \overline{\eta}_{h}\overline{\nabla}_{h}( \nabla  w_{h})\Vert_{L^{p}(\Omega)}.
\end{align*}
We notice $\overline{\nabla}_{h}( \overline{\eta}_{h}  \nabla  w_{h}) 
- \overline{\eta}_{h}\overline{\nabla}_{h}( \nabla  w_{h}) \in [\overline{V}_{h}]^{d\times d}$. 
we denote by $\underline{\varphi}_{h} \in [\overline{V}_{h}]^{d \times d}$ such that 
\begin{align*}
\Vert \overline{\nabla}_{h}( \overline{\eta}_{h}  \nabla  w_{h}) 
- \overline{\eta}_{h}\overline{\nabla}_{h}( \nabla  w_{h})\Vert_{L^{p}(\Omega)} 
= \left(\overline{\nabla}_{h}( \overline{\eta}_{h}  \nabla  w_{h}) 
- \overline{\eta}_{h}\overline{\nabla}_{h}( \nabla  w_{h}), \underline{\varphi}_{h} \right)_{\Omega}, 
\qquad \Vert \underline{\varphi}_{h} \Vert_{L^{p^{\prime}}(\Omega)} = 1.
\end{align*}
Since $\overline{\eta}_{h}\underline{\varphi}_{h} \in [\overline{V}_{h}]^{d \times d}$, 
Definition~\ref{def_discrete_deri} implies that  
\begin{align*}
& \left(\overline{\nabla}_{h}( \overline{\eta}_{h}  \nabla  w_{h}) 
- \overline{\eta}_{h}\overline{\nabla}_{h}( \nabla  w_{h}), \underline{\varphi}_{h} \right)_{\Omega} \\
= & - \langle [\overline{\eta}_{h}  \nabla  w_{h}] \otimes \boldsymbol{n}, 
\{ \underline{\varphi}_{h} \} \rangle_{\mathcal{F}_{h}^{I}} 
+ \langle [\nabla  w_{h}] \otimes \boldsymbol{n}, 
\{ \overline{\eta}_{h} \underline{\varphi}_{h} \} \rangle_{\mathcal{F}_{h}^{I}} \\
= & - \langle [\overline{\eta}_{h}] \{ \nabla  w_{h}\} \otimes \boldsymbol{n}, 
\{ \underline{\varphi}_{h} \} \rangle_{\mathcal{F}_{h}^{I}} 
+ \frac{1}{4}\langle [\nabla  w_{h}] \otimes \boldsymbol{n}, 
[ \overline{\eta}_{h}] [\underline{\varphi}_{h} ] \rangle_{\mathcal{F}_{h}^{I}}.
\end{align*}
Then by discrete trace inequality and the fact
$\eta = \overline{\eta}_{h} = 0$ in $\Omega \setminus B_{\frac{11}{6}R}$, we have 
\begin{align}
\label{gloabl_estimate_with_reaction_ineq7}
(\tilde{A}:\left(\overline{\nabla}_{h}( \overline{\eta}_{h}  \nabla  w_{h}) 
- \overline{\eta}_{h}\overline{\nabla}_{h}( \nabla  w_{h})\right), v_{h})_{\Omega} 
\leq C R^{-1} \Vert w_{h}\Vert_{W^{1,p}(B_{2R})}.
\end{align}

By the fact $\eta = \overline{\eta}_{h} = 0$ in $\Omega \setminus B_{\frac{11}{6}R}$ and Definition~\ref{def_discrete_deri},
\begin{align*}
& (\tilde{A}:\overline{\nabla}_{h}( \nabla  w_{h}), \overline{\eta}_{h} v_{h} - I_{h}(\eta v_{h}))_{\Omega} \\
\leq & \Vert \tilde{A} \Vert_{L^{\infty}(\Omega)} \Vert \overline{\nabla}_{h}( \nabla  w_{h}) \Vert_{L^{p}(B_{\frac{11}{6}R})} 
\Vert \overline{\eta}_{h} v_{h} - I_{h}(\eta v_{h}) \Vert_{L^{p^{\prime}}(\Omega)} \\
\leq & C \Vert w_{h}\Vert_{W_{h}^{2,p}(B_{\frac{23}{12}}R)}
\Vert \overline{\eta}_{h} v_{h} - I_{h}(\eta v_{h}) \Vert_{L^{p^{\prime}}(\Omega)}.
\end{align*}
By \cite[Lemma~$2.2$]{FHM2017}, we have 
\begin{align*}
\Vert w_{h} \Vert_{W_{h}^{2,p}(B_{\frac{23}{12} R})} \leq Ch^{-1} \Vert w_{h}\Vert_{W^{1,p}(B_{2R})}.
\end{align*}
So we have 
\begin{align*}
(\tilde{A}:\overline{\nabla}_{h}( \nabla  w_{h}), \overline{\eta}_{h} v_{h} - I_{h}(\eta v_{h}))_{\Omega} 
\leq C h^{-1}\Vert w_{h}\Vert_{W^{1,p}(B_{2R})} 
\Vert \overline{\eta}_{h} v_{h} - I_{h}(\eta v_{h}) \Vert_{L^{p^{\prime}}(\Omega)}.
\end{align*}
By (\ref{nodal_interpolation_error},\ref{local_norm_variable_A}), we have 
\begin{align}
\label{gloabl_estimate_with_reaction_ineq8} 
& (\tilde{A}:\overline{\nabla}_{h}( \nabla  w_{h}), \overline{\eta}_{h} v_{h} - I_{h}(\eta v_{h}))_{\Omega} \\
\nonumber
\leq & C R^{-1} \Vert w_{h}\Vert_{W^{1,p}(B_{2R})} \Vert v_{h} \Vert_{L^{p^{\prime}}(\Omega)} 
= C R^{-1} \Vert w_{h}\Vert_{W^{1,p}(B_{2R})}.
\end{align}

We notice $I_{h}(\eta v_{h}) \in V_{h}(B_{\frac{11}{6}R})$. Then by 
(\ref{L_h_operator},\ref{discrete_zero_order_norm},\ref{nodal_interpolation_error},\ref{local_norm_variable_A}), 
we have 
\begin{align}
\label{gloabl_estimate_with_reaction_ineq9}
& (\tilde{A}:\overline{\nabla}_{h}( \nabla  w_{h}), I_{h}(\eta v_{h}))_{\Omega} 
= (\mathcal{L}_{\tilde{A},h}w_{h}, I_{h}(\eta v_{h}))_{\Omega} \\
\nonumber
\leq & \Vert \mathcal{L}_{\tilde{A},h}w_{h}\Vert_{L_{h}^{p}(B_{2R})} 
\Vert I_{h}(\eta v_{h})\Vert_{L^{p^{\prime}}(\Omega)}
\leq C \Vert \mathcal{L}_{\tilde{A},h}w_{h}\Vert_{L_{h}^{p}(B_{2R})} \Vert v_{h} \Vert_{L^{p^{\prime}}(B_{2R})}
= C \Vert \mathcal{L}_{\tilde{A},h}w_{h}\Vert_{L_{h}^{p}(B_{2R})}.
\end{align}

According to (\ref{gloabl_estimate_with_reaction_ineq4},\ref{gloabl_estimate_with_reaction_ineq5},
\ref{gloabl_estimate_with_reaction_ineq6},\ref{gloabl_estimate_with_reaction_ineq7},\ref{gloabl_estimate_with_reaction_ineq8},
\ref{gloabl_estimate_with_reaction_ineq9}), we have 
\begin{align}
\label{gloabl_estimate_with_reaction_ineq10} 
\lambda_{0}\Vert w_{h}\Vert_{W_{h}^{2,p}(B_{R})} 
\leq C \left( \Vert \mathcal{L}_{\tilde{A},h}w_{h}\Vert_{L_{h}^{p}(B_{2R})} 
+ R^{-2} \Vert w_{h} \Vert_{W^{1,p}(B_{2R})} \right), \qquad \forall w_{h} \in V_{h}.
\end{align}
We recall that $R$ is chosen to be $R_{1}$ in Lemma~\ref{lemma_local_estimate_variable_A}. 
The constant $\lambda_{0}$ may depend on $\tilde{A}(\boldx_{0})$.  
 
Step $2$. Let $\{\boldx_{j}\}_{j=1}^{N} \subset \overline{\Omega}$ such that 
$\overline{\Omega} = \cup_{j=1}^{N}\overline{B}_{R_{j}^{\prime}}(\boldx_{j})$ 
where $R_{j}^{\prime}$ is the constant $R$ for $\boldx_{j}$ in (\ref{gloabl_estimate_with_reaction_ineq10}), 
and $B_{R_{j}^{\prime}}:= \{ \boldx \in \Omega : \vert \boldx - \boldx_{j} \vert 
< R_{j}^{\prime}\}$. Since $\overline{\Omega}$ is compact, $N$ is finite and independent of $h$.

We denote by $S_{j} = B_{R_{j}^{\prime}}$ and $\tilde{S}_{j} = B_{2R_{j}^{\prime}}
:= \{ \boldx \in \Omega : \vert \boldx - \boldx_{j} \vert < 2R_{j}^{\prime}\}$ for any $1\leq j \leq N$.

Since $R_{j}$ is independent of $h$ for any $1 \leq j \leq N$, (\ref{gloabl_estimate_with_reaction_ineq10}) 
implies that 
\begin{align*}
\Vert w_{h}\Vert_{W_{h}^{2,p}(\Omega)}^{p}  \leq \Sigma_{j=1}^{N} \Vert w_{h} \Vert_{W_{h}^{2,p}(S_{j})}^{p} 
\leq C \left( \Sigma_{j=1}^{N} \Vert \mathcal{L}_{\tilde{A},h} w_{h} \Vert_{L_{h}^{p}(\tilde{S}_{j})}^{p} 
+ \Vert w_{h}\Vert_{W^{1,p}(\Omega)}^{p} \right).
\end{align*} 

By (\ref{discrete_zero_order_norm}), 
\begin{align*}
& \Sigma_{j=1}^{N} \Vert \mathcal{L}_{\tilde{A},h} w_{h} \Vert_{L_{h}^{p}(\tilde{S}_{j})}^{p}
= \Sigma_{j=1}^{N} \left| \sup_{v_{h} \in V_{h}(\tilde{S}_{j})}\dfrac{(\mathcal{L}_{\tilde{A},h} 
w_{h}, v_{h})_{\Omega}}{\Vert v_{h}\Vert_{L^{p^{\prime}}(\tilde{S}_{j})}} \right|^{p}\\
\leq & N \left| \sup_{v_{h} \in V_{h}}\dfrac{(\mathcal{L}_{\tilde{A},h} 
w_{h}, v_{h})_{\Omega}}{\Vert v_{h}\Vert_{L^{p^{\prime}}(\Omega)}} \right|^{p} 
= N \Vert \mathcal{L}_{\tilde{A}, h}w_{h} \Vert_{L_{h}^{p}(\Omega)}.
\end{align*}
 
Therefore we can conclude that the proof is complete.
\end{proof}

\begin{lemma}
\label{lemma_conv_compactness} 
(Discrete compactness)
Let $1 < p < + \infty$ and $\Omega$ be a Lipschitz polyhedra in $\mathbb{R}^{d}$ 
($d=2,3$). We assume that there are $0 < M_{1} < +\infty$ and 
$\{w_{h} \in V_{h}\}_{h > 0}$ satisfying 
\begin{align*}
\Vert  w_{h} \Vert_{W_{h}^{2,p}(\Omega)} 
\leq M_{1}, \qquad \forall h > 0.
\end{align*} 
Then there is $w \in W^{2,p}(\Omega) \cap W_{0}^{1,p}(\Omega)$ and $ 0<M_{2}< + \infty$ 
such that 
\begin{align*}
\Vert w_{h} - w\Vert_{W^{1,p}(\Omega)} \rightarrow 0, \qquad 
\overline{\nabla}_{h}(\nabla w_{h}) \rightharpoonup D^{2}w \text{ in } 
[L^{p}(\Omega)]^{d \times d},
\end{align*}
for a subsequence of $h \rightarrow 0$; and 
\begin{align*}
\quad \Vert \overline{\nabla}_{h}(\nabla w) \Vert_{L^{p}(\Omega)} 
+ \Vert \overline{\nabla}_{h} (\nabla w_{h})\Vert_{L^{p}(\Omega)} 
\leq M_{2}, \qquad \forall h> 0.
\end{align*}
\end{lemma}

\begin{proof}
Since $\Vert  w_{h} \Vert_{W_{h}^{2,p}(\Omega)} \leq M_{1}$ for any $h>0$, 
Lemma~\ref{lemma_discrete_norm_control} implies 
\begin{align*}
\Vert w_{h}\Vert_{W^{1,p}(\Omega)} \leq C M_{1}, \qquad \forall h>0.
\end{align*}
Thus there is $w \in W_{0}^{1,p}(\Omega)$ such that 
\begin{align}
\label{conv_w}
\Vert w_{h} - w\Vert_{L^{p}(\Omega)} \rightarrow 0, \qquad 
w_{h} \rightharpoonup w \text{ in } W_{0}^{1,p}(\Omega),
\end{align}
for a subsequence of $h \rightarrow 0$.
We denote by $\boldsymbol{\phi}_{h} \in [H^{1}(\Omega)\cap \overline{V}_{h}]^{d}$ the standard $H^{1}$-conforming 
averaging of $\nabla w_{h}$. By Lemma~\ref{lemma_H1_interpolation}, we have  
\begin{align}
\label{interpolation_gradient_w}
\Vert \boldsymbol{\phi}_{h} - \nabla w_{h} \Vert_{L^{p}(\Omega)} 
\leq C h \left( \Sigma_{F \in \mathcal{F}_{h}^{I}} h_{F}^{1-p}\Vert [\nabla w_{h}] \Vert_{L^{p}(F)}^{p} \right)^{\frac{1}{p}}.
\end{align}
(\ref{interpolation_gradient_w}) and the fact $ \Vert w_{h}\Vert_{W^{1,p}(\Omega)} + \Vert w_{h} \Vert_{W_{h}^{2,p}(\Omega)} 
\leq (1+C)M_{1}$ imply that there is $\boldsymbol{\phi} \in [W^{1,p}(\Omega)]^{d}$ such that 
\begin{align}
\label{conv_phi}
\Vert \boldsymbol{\phi}_{h} - \boldsymbol{\phi}\Vert_{L^{p}(\Omega)} 
+ \Vert \nabla w_{h} - \boldsymbol{\phi} \Vert_{L^{p}(\Omega)} \rightarrow 0, 
\quad \text{ and } \quad \boldsymbol{\phi}_{h} \rightharpoonup \boldsymbol{\phi} \text{ in } [W^{1,p}(\Omega)]^{d},
\end{align}
for a subsequence of $h \rightarrow 0$.
(\ref{conv_w},\ref{conv_phi}) imply that $\nabla w = \boldsymbol{\phi}$ almost everywhere in $\Omega$. Therefore, we have that 
$w \in W^{2,p}(\Omega) \cap W_{0}^{1,p}(\Omega)$ and 
\begin{align}
\label{conv_w_second_order}
\Vert w_{h} - w \Vert_{W^{1,p}(\Omega)} + \Vert \boldsymbol{\phi}_{h} - \nabla w \Vert_{L^{p}(\Omega)} \rightarrow 0, 
\quad \text{ and } \boldsymbol{\phi}_{h} \rightharpoonup \nabla w \text{ in } [W^{1,p}(\Omega)]^{d},
\end{align}
for a subsequence of $h \rightarrow 0$.

According to Definition~\ref{def_discrete_deri} with the fact $\Vert w_{h} \Vert_{W_{h}^{2,p}(\Omega)} \leq M_{1}
< + \infty$ and $w\in W^{2,p}(\Omega)$, there is a constant $M_{2}$ such that 
\begin{align}
\label{discrete_second_deri_bound}
\Vert \overline{\nabla}_{h}(\nabla w) \Vert_{L^{p}(\Omega)} 
+ \Vert \overline{\nabla}_{h} (\nabla w_{h})\Vert_{L^{p}(\Omega)} \leq M_{2} < + \infty.
\end{align}
We denote by $P_{h}$ the standard $L^{2}$-orthogonal projection onto $V_{h}$.
By Lemma~\ref{lemma_L2_proj_props}, there is a constant $\tilde{C}_{1}$ independent of $h$ 
such that for any $v \in W_{0}^{1,p^{\prime}}(\Omega)$, 
\begin{align}
\label{L2_projection_bound1}
\Vert P_{h} v\Vert_{L^{p^{\prime}}(\Omega)} \leq \tilde{C}_{1} \Vert v\Vert_{L^{p^{\prime}}(\Omega)}, 
\qquad \Vert P_{h} v\Vert_{W^{1,p^{\prime}}(\Omega)} \leq \tilde{C}_{1} \Vert v\Vert_{W^{1,p^{\prime}}(\Omega)}.
\end{align}

By (\ref{discrete_second_deri_bound}), there is $\underline{\psi} \in [L^{p}(\Omega)]^{d \times d}$ such that 
\begin{align*}
\overline{\nabla}_{h}(\nabla w_{h}) \rightharpoonup \underline{\psi} \in [L^{p}(\Omega)]^{d \times d},
\end{align*} 
for a subsequence of $h \rightarrow 0$.
On the other hand, we have that for any $\underline{\zeta} \in [C_{0}^{\infty}(\Omega)]^{d\times d}$,
\begin{align*}
(\overline{\nabla}_{h}(\nabla w_{h}), \underline{\zeta})_{\Omega} 
= (\overline{\nabla}_{h}(\nabla w_{h}), P_{h}(\underline{\zeta}))_{\Omega} 
+ (\overline{\nabla}_{h}(\nabla w_{h}), P_{h}(\underline{\zeta}) - \underline{\zeta})_{\Omega}.
\end{align*}
Since $P_{h}(\underline{\zeta}) \in [V_{h}]^{d \times d}$, then by Definition~\ref{def_discrete_deri}  
\begin{align*}
& (\overline{\nabla}_{h}(\nabla w_{h}), \underline{\zeta})_{\Omega} 
= -(\nabla w_{h}, \nabla \cdot P_{h}(\underline{\zeta}))_{\Omega} 
+ (\overline{\nabla}_{h}(\nabla w_{h}), P_{h}(\underline{\zeta}) - \underline{\zeta})_{\Omega} \\
= &  -(\nabla w_{h}, \nabla \cdot \underline{\zeta})_{\Omega} 
-  (\nabla w_{h}, \nabla \cdot (P_{h}(\underline{\zeta}) -  \underline{\zeta}))_{\Omega} 
+ (\overline{\nabla}_{h}(\nabla w_{h}), P_{h}(\underline{\zeta}) - \underline{\zeta})_{\Omega}
\end{align*} 
In the following of the proof, $\lim_{h\rightarrow 0}$ means the limit 
when a subsequence mentioned above of $h \rightarrow 0$.
By (\ref{conv_w_second_order}), we have 
\begin{align*}
\lim_{h \rightarrow 0} -(\nabla w_{h}, \nabla \cdot \underline{\zeta})_{\Omega} 
= -(\nabla w, \nabla\cdot \underline{\zeta})_{\Omega} 
= (D^{2}w, \underline{\zeta})_{\Omega}.
\end{align*}
By (\ref{discrete_second_deri_bound}) and the fact $\Vert w_{h} \Vert_{W_{h}^{2,p}(\Omega)} \leq M_{1}
< + \infty$, (\ref{L2_projection_bound1}) implies that 
\begin{align*}
\lim_{h\rightarrow 0} (\nabla w_{h}, \nabla \cdot (P_{h}(\underline{\zeta}) -  \underline{\zeta}))_{\Omega}  
= \lim_{h \rightarrow 0} (\overline{\nabla}_{h}(\nabla w_{h}), P_{h}(\underline{\zeta}) - \underline{\zeta})_{\Omega} = 0.
\end{align*}
Therefore, we have 
\begin{align*}
\lim_{h \rightarrow 0} (\overline{\nabla}_{h}(\nabla w_{h}), \underline{\zeta})_{\Omega} = 
(D^{2}w, \underline{\zeta})_{\Omega}, \qquad \forall \underline{\zeta} \in [C_{0}^{\infty}(\Omega)]^{d \times d}. 
\end{align*}
Therefore, we have $\underline{\psi} = D^{2}w$ almost everywhere in $\Omega$. So we have 
\begin{align}
\label{conv_w_second_order_weak} 
\overline{\nabla}_{h}(\nabla w_{h}) \rightharpoonup D^{2}w \in [L^{p}(\Omega)]^{d \times d},
\end{align}
for a subsequence of $h \rightarrow 0$.

According to (\ref{conv_w_second_order}, \ref{discrete_second_deri_bound}, 
\ref{conv_w_second_order_weak}), We can conclude the proof is complete.
\end{proof}

\begin{theorem}
\label{thm_global_estimate}
(Stability of the FEM (\ref{nondiv_fem})) 
Let $A \in [C^{0}(\overline{\Omega})]^{d\times d}$ uniformly elliptic, $\boldsymbol{b}\in 
[L^{\infty}(\Omega)]^{d}$, $c \in L^{\infty}(\Omega)$ with $c \leq 0$, 
 and $\Omega$ is a bounded open Lipschitz polyhedral domain in $\mathbb{R}^{d}$ ($d=2,3$).  
There is $h_{2} > 0$ which may depend on $p$, such that for all $h \in (0, h_{2})$,
\begin{align}
\label{global_estimate}
\Vert w_{h}\Vert_{W^{1,p}(\Omega)} + \Vert w_{h} \Vert_{W_{h}^{2,p}(\Omega)} 
\leq C \Vert \tilde{A}: \overline{\nabla}_{h}(\nabla w_{h}) +\tilde{\boldsymbol{b}}\cdot\nabla w_{h} 
+ \tilde{c} w_{h}\Vert_{L_{h}^{p}(\Omega)}, \quad \forall w_{h} \in V_{h}.
\end{align}
Here $p$ is valid in the range described in (\ref{p_range2}). 
$\tilde{A}$, $\tilde{\boldsymbol{b}}$ and $\tilde{c}$ are defined in (\ref{nondiv_pde}). 
The norm $\Vert \cdot \Vert_{L_{h}^{p}(\Omega)}$ is defined in (\ref{discrete_zero_order_norm}).
\end{theorem}

\begin{proof}
We notice that by (\ref{def_gamma}), $A \in [C(\overline{\Omega})]^{d\times d}$ is equivalent 
to $\tilde{A} \in [C(\overline{\Omega})]^{d\times d}$, and $c \leq 0$ is 
equivalent to $\tilde{c} \leq 0$. 

For the sake of simplicity, we only give detailed proof of the case $d=3$. The proof of the case $d=2$ ($\Omega$ 
can be non-convex) can be proven in the same manner.

According to Lemma~\ref{lemma_gloabl_estimate_with_reaction}, it is sufficient to prove that for any $p \in (1, 2]$,
\begin{align*}
\Vert w_{h}\Vert_{W^{1,p}(\Omega)} \leq C \Vert \tilde{A}: \overline{\nabla}_{h}(\nabla w_{h})
+\tilde{\boldsymbol{b}}\cdot\nabla w_{h} + \tilde{c} w_{h}  \Vert_{L_{h}^{p}(\Omega)}, 
\qquad \forall w_{h} \in V_{h}.
\end{align*}
Here the constant $C$ may depend on $p$. 

We prove by contradiction. Without losing generality, we assume that there is one $p \in (1, 2]$ and 
a sequence of $\{w_{h} \in V_{h}\}_{h>0}$, such that 
\begin{align}
\label{assumption_contradition1}
\Vert w_{h}\Vert_{W^{1,p}(\Omega)} = 1, \quad \text{and } \quad  
\lim_{h\rightarrow 0}\Vert \tilde{A}: \overline{\nabla}_{h}(\nabla w_{h})
+\tilde{\boldsymbol{b}}\cdot\nabla w_{h} 
+ \tilde{c}w_{h}\Vert_{L_{h}^{p}(\Omega)} = 0.
\end{align} 
So there is a positive constant $C$ such that 
\begin{align*}
& \Vert \mathcal{L}_{\tilde{A},h}w_{h} \Vert_{L_{h}^{p}(\Omega)} 
= \Vert \tilde{A}: \overline{\nabla}_{h}(\nabla w_{h})\Vert_{L_{h}^{p}(\Omega)} \\
\leq & C \left( \Vert \tilde{A}: \overline{\nabla}_{h}(\nabla w_{h})
+\tilde{\boldsymbol{b}}\cdot\nabla w_{h} + \tilde{c}w_{h}\Vert_{L_{h}^{p}(\Omega)}
+  \Vert w_{h} \Vert_{W^{1,p}(\Omega)} \right).
\end{align*}
The equality above is due to (\ref{L_h_operator}). 
According to Lemma~\ref{lemma_gloabl_estimate_with_reaction}, there is $M_{1} < +\infty$ such that 
\begin{align*}
\Vert w_{h} \Vert_{W_{h}^{2,p}(\Omega)} \leq M_{1}, \qquad \forall h>0.
\end{align*}
Therefore, Lemma~\ref{lemma_conv_compactness} implies that 
there is $w \in W^{2} \cap W_{0}^{1,p}(\Omega)$ and $ 0<M_{2}< + \infty$ 
such that 
\begin{align}
\label{conv_w_second_order_proof}
& \Vert w_{h} - w\Vert_{W^{1,p}(\Omega)} \rightarrow 0, \\
\label{conv_w_second_order_weak_proof} 
& \overline{\nabla}_{h}(\nabla w_{h}) \rightharpoonup D^{2}w \text{ in } 
L^{p}(\Omega),
\end{align}
for a subsequence of $h \rightarrow 0$; and 
\begin{align}
\label{discrete_second_deri_bound_proof}
\quad \Vert \overline{\nabla}_{h}(\nabla w) \Vert_{L^{p}(\Omega)} 
+ \Vert \overline{\nabla}_{h} (\nabla w_{h})\Vert_{L^{p}(\Omega)} 
\leq M_{2}, \qquad \forall h> 0.
\end{align}

In the following of the proof, we don't distinguish $h \rightarrow 0$ 
and a subsequence of $h \rightarrow 0$.
Then (\ref{assumption_contradition1}) and (\ref{conv_w_second_order_proof}) imply that $\Vert w\Vert_{W^{1,p}(\Omega)} 
\geq \frac{1}{2}$. By Theorem~\ref{thm_global_A_continuous}, in order to achieve contradiction with 
$\Vert w\Vert_{W^{1,p}(\Omega)} \geq \frac{1}{2}$, we only need to prove that 
\begin{align*}
\tilde{A}: D^{2} w + \tilde{\boldsymbol{b}}\cdot \nabla w 
+ \tilde{c} w = 0 \text{ almost everywhere in } \Omega. 
\end{align*}
Since $w \in W^{2,p}(\Omega)$, 
it is sufficient to prove that 
\begin{align}
\label{statement_contradiction1}
 (\tilde{A}:D^{2}w + \tilde{\boldsymbol{b}}\cdot \nabla w + \tilde{c} w, 
 v)_{\Omega}  = 0, \qquad \forall v \in C_{0}^{\infty}(\Omega).
\end{align}

We choose $v \in C_{0}^{\infty}(\Omega)$ arbitrarily.
It is easy to see 
\begin{align*}
& (\tilde{A}:D^{2}w + \tilde{\boldsymbol{b}}\cdot \nabla w + \tilde{c} 
w, v)_{\Omega} \\ 
= & (\tilde{A}:(D^{2}w -\overline{\nabla}_{h}(\nabla w_{h})), v )_{\Omega} 
+ (\tilde{A}:\overline{\nabla}_{h}(\nabla w_{h}), v )_{\Omega} 
+ (\tilde{\boldsymbol{b}}\cdot \nabla w + \tilde{c} w, v)_{\Omega} \\
= & (\tilde{A}:(D^{2}w -\overline{\nabla}_{h}(\nabla w_{h})), v )_{\Omega} 
+ (\tilde{A}:\overline{\nabla}_{h}(\nabla w_{h}) 
+ \tilde{\boldsymbol{b}}\cdot \nabla w_{h} + \tilde{c} w_{h}, P_{h} v )_{\Omega} \\
& \qquad + (\tilde{A}:\overline{\nabla}_{h}(\nabla w_{h})
+ \tilde{\boldsymbol{b}}\cdot \nabla w_{h} + \tilde{c} w_{h}, v - P_{h}v )_{\Omega} 
+ (\tilde{\boldsymbol{b}}\cdot \nabla (w-w_{h}) + \tilde{c}(w-w_{h}), v)_{\Omega}.
\end{align*}
Here we denote by $P_{h}$ the standard $L^{2}$-orthogonal projection onto $V_{h}$.
We define $p^{\prime} \in (1, + \infty)$ by  $\frac{1}{p} + \frac{1}{p^{\prime}} = 1$. 
By (\ref{conv_w_second_order_weak_proof}) and the fact that $\tilde{A}v\in [L^{p^{\prime}}(\Omega)]^{d \times d}$, 
\begin{align*}
\lim_{h \rightarrow 0} (\tilde{A}:(D^{2}w -\overline{\nabla}_{h}(\nabla w_{h})), v )_{\Omega} = 0.
\end{align*}
By Lemma~\ref{lemma_L2_proj_props}, there is a constant $\tilde{C}_{1}$ independent of $h$, 
\begin{align}
\label{L2_projection_bound2}
\Vert P_{h} v\Vert_{L^{p^{\prime}}(\Omega)} \leq \tilde{C}_{1} \Vert v\Vert_{L^{p^{\prime}}(\Omega)}, 
\qquad \Vert P_{h} v\Vert_{W^{1,p^{\prime}}(\Omega)} \leq \tilde{C}_{1} \Vert v\Vert_{W^{1,p^{\prime}}(\Omega)}.
\end{align}
Therefore, by (\ref{assumption_contradition1}) and (\ref{L2_projection_bound2}), we have 
\begin{align*}
& \lim_{h\rightarrow 0} \vert (\tilde{A}:\overline{\nabla}_{h}(\nabla w_{h}
+ \tilde{\boldsymbol{b}}\cdot\nabla w_{h}+\tilde{c}w_{h}), P_{h}v)_{\Omega}\vert\\
\leq & \lim_{h \rightarrow 0} \Vert \tilde{A}: \overline{\nabla}_{h}(\nabla w_{h}) 
+ \tilde{\boldsymbol{b}}\cdot \nabla w_{h} + \tilde{c} w_{h}\Vert_{L_{h}^{p}(\Omega)} 
\Vert P_{h} v\Vert_{L^{p^{\prime}}(\Omega)} \\
\leq & \tilde{C}_{1} \lim_{h \rightarrow 0} \Vert \tilde{A}: \overline{\nabla}_{h}(\nabla w_{h}) 
+ \tilde{\boldsymbol{b}}\cdot \nabla w_{h} + \tilde{c} w_{h}\Vert_{L_{h}^{p}(\Omega)} 
\Vert v\Vert_{L^{p^{\prime}}(\Omega)} = 0.
\end{align*}
By (\ref{assumption_contradition1},\ref{discrete_second_deri_bound_proof},\ref{L2_projection_bound2}), we have 
\begin{align*}
\lim_{h\rightarrow 0} (\tilde{A}:\overline{\nabla}_{h}(\nabla w_{h})
+\tilde{\boldsymbol{b}}\cdot\nabla w_{h}+\tilde{c}w_{h},v-P_{h}v)_{\Omega} = 0.
\end{align*}
By (\ref{conv_w_second_order_proof}), we have 
\begin{align*}
\lim_{h\rightarrow 0} (\tilde{\boldsymbol{b}}\cdot \nabla (w-w_{h}) + \tilde{c}(w-w_{h}), v)_{\Omega} = 0.
\end{align*}
Therefore, we have 
\begin{align}
\label{verification_contradiction1}
(\tilde{A}:D^{2}w + \tilde{\boldsymbol{b}}\cdot \nabla w 
+ \tilde{c}w, v)_{\Omega} = 0, \qquad \forall v \in C_{0}^{\infty}(\Omega).
\end{align}

(\ref{verification_contradiction1}) implies that the proof is complete.
\end{proof}

\subsection{Well-posedness of linear elliptic PDE in non-divergence form with uniformly continuous $A$ and optimal 
convergence of the numerical solution of FEM}

\begin{theorem}
\label{thm_A_continuous_complete}
Let $A \in [C^{0}(\overline{\Omega})]^{d\times d}$ uniformly elliptic, 
$\boldsymbol{b}\in [L^{\infty}(\Omega)]^{d}$, $c \in L^{\infty}(\Omega)$ with $c \leq 0$, 
 and $\Omega$ is a bounded open Lipschitz polyhedral domain in $\mathbb{R}^{d}$ ($d=2,3$).  
For any $f \in L^{p}(\Omega)$, 
there is a unique $u \in W^{2,p}(\Omega) \cap W_{0}^{1,p}(\Omega)$ such that 
\begin{subequations}
\label{A_continuous_wellposedness_props}
\begin{align}
\label{A_continuous_wellposedness_prop1}
& A:D^{2} u + \boldsymbol{b}\cdot \nabla u + cu = f \text{ in } \Omega, \\
\label{A_continuous_wellposedness_prop2}
& \Vert u\Vert_{W^{2,p}(\Omega)} \leq C_{p} \Vert f \Vert_{L^{p}(\Omega)}.
\end{align}
\end{subequations}
Here the constant $C_{p}$ is the same as the one in Theorem~\ref{thm_global_A_continuous}. 
$p$ is valid in the range described in (\ref{p_range1}).

Let $u_{h} \in V_{h}$ be the numerical solution of the finite element method 
(\ref{nondiv_fem}). Then we have 
\begin{align}
\label{A_continuous_conv}
\Vert u_{h} -  u \Vert_{W^{1,p}(\Omega)} \rightarrow 0, \qquad 
\overline{\nabla}_{h}(\nabla u_{h}) \rightharpoonup D^{2}u \text{ in } 
[L^{p}(\Omega)]^{d \times d},
\end{align}
as $h \rightarrow 0$. 
And if $h \in (0,h_{2})$ where $h_{2}$ is introduced in Theorem~\ref{thm_global_estimate}, then for any 
$\chi_{h} \in V_{h}$,
\begin{align}
\label{A_continuous_conv_rate}
& \Vert u_{h} - u \Vert_{W^{1,p}(\Omega)}
+ \Vert u_{h} - u \Vert_{W_{h}^{2,p}(\Omega)} \\
\nonumber
\leq C & \big( \Vert u  - \chi_{h} \Vert_{W^{1,p}(\Omega)} + \Vert u - \chi_{h}\Vert_{W_{h}^{2,p}(\Omega)} 
+ \Vert \overline{P}_{h}(D^{2}u) - D^{2}u \Vert_{L^{p}(\Omega)} \big).
\end{align}
Here $p$ is valid in the range described in (\ref{p_range2}). 
$\overline{P}_{h}$ denotes the standard $L^{2}$-orthogonal projection onto $[\overline{V}_{h}]^{d\times d}$.
\end{theorem}

\begin{proof}
We divide the proof into two steps. In step $1$,
We will first prove the estimates (\ref{A_continuous_wellposedness_props}) 
and (\ref{A_continuous_conv},\ref{A_continuous_conv_rate}) for 
 $1< p \leq 2$ if $\Omega$ is convex polyhedra in $\mathbb{R}^{d}$ ($d=2,3$), and 
 $ \frac{4}{3} - \epsilon_{1} <p < \frac{4}{3} 
+ \epsilon_{2}$ if $\Omega$ is a polygon (maybe non-convex) in $\mathbb{R}^{2}$. 
In step $2$, we will prove the estimate (\ref{A_continuous_wellposedness_props}) 
for $1< p \leq \frac{4}{3} - \epsilon_{1}$ and $\Omega$ is a 
polygon (maybe non-convex) in $\mathbb{R}^{2}$. 

Step $1$. We consider $1< p \leq 2$ if $\Omega$ is convex polyhedra in $\mathbb{R}^{d}$ ($d=2,3$), and $ \frac{4}{3} - \epsilon_{1} <p < \frac{4}{3} 
+ \epsilon_{2}$ if $\Omega$ is a polygon (maybe non-convex) in $\mathbb{R}^{2}$.

By the design of the finite element method (\ref{nondiv_fem}) and 
Theorem~\ref{thm_global_estimate}, there is a unique $u_{h} \in V_{h}$ 
to be the numerical solution of (\ref{nondiv_fem}) and 
\begin{align*}
\Vert u_{h} \Vert_{W_{h}^{2,p}(\Omega)} \leq C \Vert f\Vert_{L^{p}(\Omega)},
\end{align*}
if $h$ is small enough. Then by Lemma~\ref{lemma_conv_compactness}, there is 
$u \in W^{2,p}(\Omega) \cap W_{0}^{1,p}(\Omega)$ such that 
\begin{align*}
\Vert u_{h} - u\Vert_{W^{1,p}(\Omega)} \rightarrow 0, \qquad 
\overline{\nabla}_{h}(\nabla u_{h}) \rightharpoonup D^{2}u 
\text{ in } [L^{p}(\Omega)]^{d \times d},
\end{align*}
for a sequence of $h \rightarrow 0$; and there is $0 < M_{2} < +\infty$ such that 
$\Vert \overline{\nabla}(\nabla u_{h})\Vert_{L^{p}(\Omega)} \leq M_{2}$ for any 
$h > 0$.

Now we need to prove (\ref{A_continuous_wellposedness_prop1}). 
We define $p^{\prime} \in (1, + \infty)$ by  $\frac{1}{p} + \frac{1}{p^{\prime}} = 1$. 
We denote by $P_{h}$ the standard $L^{2}$-orthogonal projection onto $V_{h}$.
It is well-known that there is a constant $\tilde{C}_{1}$ independent of $h$, 
\begin{align}
\label{L2_projection_bound3}
\Vert P_{h} v\Vert_{L^{p^{\prime}}(\Omega)} \leq \tilde{C}_{1} \Vert v\Vert_{L^{p^{\prime}}(\Omega)}, 
\qquad \Vert P_{h} v\Vert_{W^{1,p^{\prime}}(\Omega)} \leq \tilde{C}_{1} \Vert v\Vert_{W^{1,p^{\prime}}(\Omega)}.
\end{align}

We choose $v \in C_{0}^{\infty}(\Omega)$ arbitrarily. 
By the design of the finite element method (\ref{nondiv_fem}), We have 
\begin{align*}
& (\tilde{A}:D^{2}u + \tilde{\boldsymbol{b}}\cdot\nabla u 
+ \tilde{c}u, v)_{\Omega} \\
= & (\tilde{A}: (D^{2}u - \overline{\nabla}_{h}(\nabla u_{h})), v)_{\Omega} 
+ (\tilde{A}:\overline{\nabla}_{h}(\nabla u_{h}) 
+ \tilde{\boldsymbol{b}}\cdot\nabla u + \tilde{c}u, v)_{\Omega} \\
= & (\tilde{A}: (D^{2}u - \overline{\nabla}_{h}(\nabla u_{h})), v)_{\Omega} 
+ (\tilde{A}:\overline{\nabla}_{h}(\nabla u_{h}) 
+ \tilde{\boldsymbol{b}}\cdot\nabla u + \tilde{c}u, v - P_{h}v)_{\Omega} \\
& \qquad + (\tilde{A}:\overline{\nabla}_{h}(\nabla u_{h}) 
+ \tilde{\boldsymbol{b}}\cdot\nabla u + \tilde{c}u, P_{h}v)_{\Omega} \\
= & (\tilde{A}: (D^{2}u - \overline{\nabla}_{h}(\nabla u_{h})), v)_{\Omega} 
+ (\tilde{A}:\overline{\nabla}_{h}(\nabla u_{h}) 
+ \tilde{\boldsymbol{b}}\cdot\nabla u + \tilde{c}u, v - P_{h}v)_{\Omega} \\
& \qquad + (\tilde{A}:\overline{\nabla}_{h}(\nabla u_{h}) 
+ \tilde{\boldsymbol{b}}\cdot\nabla u_{h} + \tilde{c}u_{h}, P_{h}v)_{\Omega} 
+ (\tilde{\boldsymbol{b}}\cdot \nabla(u - u_{h}) 
+ \tilde{c}(u - u_{h}), P_{h}v)_{\Omega} \\
= & (\tilde{A}: (D^{2}u - \overline{\nabla}_{h}(\nabla u_{h})), v)_{\Omega} 
+ (\tilde{A}:\overline{\nabla}_{h}(\nabla u_{h}) 
+ \tilde{\boldsymbol{b}}\cdot\nabla u + \tilde{c}u, v - P_{h}v)_{\Omega} \\
& \qquad + (\tilde{f}, v)_{\Omega} + (\tilde{f}, P_{h}v - v)_{\Omega} 
+ (\tilde{\boldsymbol{b}}\cdot \nabla(u - u_{h}) 
+ \tilde{c}(u - u_{h}), P_{h}v)_{\Omega}.
\end{align*}

Since $\overline{\nabla}_{h}(\nabla u_{h}) \rightharpoonup D^{2}u 
\text{ in } [L^{p}(\Omega)]^{d \times d}$, we have 
\begin{align*}
(\tilde{A}:(D^{2}u - \overline{\nabla}_{h}(\nabla u_{h})), v)_{\Omega} 
+ (\tilde{\boldsymbol{b}}\cdot \nabla(u - u_{h}) 
+ \tilde{c}(u - u_{h}), P_{h}v)_{\Omega}\rightarrow 0,
\end{align*}
for a sequence of $h \rightarrow 0$.

By (\ref{L2_projection_bound3}) and the fact that $\Vert \overline{\nabla}(\nabla u_{h})\Vert_{L^{p}(\Omega)} 
\leq M_{2}$, we have 
\begin{align*}
& \vert (\tilde{A}:\overline{\nabla}_{h}(\nabla u_{h})
+ \tilde{\boldsymbol{b}}\cdot \nabla u + \tilde{c}u, v - P_{h}v)_{\Omega} \vert \\
\leq & C \left(\Vert \overline{\nabla}_{h}(\nabla u_{h})\Vert_{L^{p}(\Omega)} 
+ \Vert u\Vert_{W^{1,p}(\Omega)}\right)\Vert v - P_{h} 
v\Vert_{L^{p^{\prime}}(\Omega)} \rightarrow 0,
\end{align*}
as $h \rightarrow 0$. 

Similarly, we have 
\begin{align*}
\vert (\tilde{f}, P_{h}v - v)_{\Omega} \vert \rightarrow 0,
\end{align*}
as $h \rightarrow 0$.

So we have 
\begin{align*}
(\tilde{A}:D^{2}u + \tilde{\boldsymbol{b}}\cdot\nabla u 
+ \tilde{c}u, v)_{\Omega} = (\tilde{f}, v)_{\Omega}. 
\end{align*}

Since $v \in C_{0}^{\infty}(\Omega)$ is chosen arbitrarily, 
$\tilde{A}:D^{2}+\tilde{\boldsymbol{b}}\cdot\nabla u 
+ \tilde{c}u = \tilde{f}$ in $\Omega$ almost everywhere. 
By the definition of the function $\gamma$ (see (\ref{def_gamma})) and 
the fact that $A$ is uniformly elliptic and uniformly bounded in $\Omega$,  
$\gamma$ is uniformly bounded from below and above by positive constants. 
So we have $A:D^{2}u +\boldsymbol{b}\cdot\nabla u + cu = f$ in $\Omega$ 
almost everywhere. 
So (\ref{A_continuous_wellposedness_prop1}) 
is proven. Theorem~\ref{thm_global_A_continuous} implies the uniqueness of 
$u \in W^{2,p}(\Omega) \cap W_{0}^{1,p}(\Omega)$, which results in 
\begin{align*}
\Vert u_{h} - u\Vert_{W^{1,p}(\Omega)} \rightarrow 0, \qquad 
\overline{\nabla}_{h}(\nabla u_{h}) \rightharpoonup D^{2}u 
\text{ in } [L^{p}(\Omega)]^{d \times d},
\end{align*}
as $h \rightarrow 0$. 

The proof of (\ref{A_continuous_conv_rate}) is straightforward due to 
Theorem~\ref{thm_global_estimate} and the construction of the finite element method (\ref{nondiv_fem}).
Thus the proof in Step $1$ is complete.

Step $2$. We consider $1< p \leq \frac{4}{3} - \epsilon_{1}$ if $\Omega$ is a 
polygon (maybe non-convex) in $\mathbb{R}^{2}$. 

Let $\{{f_{i}}\}_{i \geq 1} \subset L^{\frac{4}{3}}(\Omega)$ such that 
$\Vert f_{i} -f \Vert_{L^{p}(\Omega)} \rightarrow 0$ as $i \rightarrow +\infty$. 
For any $i \geq 1$, we denote by $u_{i} \in W^{2,\frac{4}{3}}(\Omega) \cap 
W_{0}^{1, \frac{4}{3}}(\Omega)$ satisfying 
\begin{align*}
A:D^{2} u_{i} + \boldsymbol{b}\cdot \nabla u_{i} + c u_{i} = f_{i} 
\text{ in } \Omega \text{ almost everywhere}.
\end{align*}

Since $\Vert f_{i} -f \Vert_{L^{p}(\Omega)} \rightarrow 0$ as $i \rightarrow +\infty$, 
$\{ \Vert f_{i} \Vert_{L^{p}(\Omega)} \}_{i \geq 1}$ have a uniform upper bound. 
So $\{ \Vert u_{i}\Vert_{W^{2,p}(\Omega)} \}_{i \geq 1}$ have a uniform upper bound 
as well. So there is $u \in W^{2,p}(\Omega) \cap W_{0}^{1,p}(\Omega)$ such that 
\begin{align*}
u_{i} \rightharpoonup u \text{ in } W^{2,p}(\Omega),
\end{align*}
for a subsequence of $i \rightarrow + \infty$.

We choose $v \in C_{0}^{\infty}(\Omega)$ arbitrarily. Then 
\begin{align*}
& (A:D^{2} u + \boldsymbol{b}\cdot \nabla u + cu, v)_{\Omega} \\
= & (A:D^{2}(u - u_{i}) + \boldsymbol{b}\cdot \nabla (u - u_{i}) 
+ c(u - u_{i}), v)_{\Omega} 
+ (A:D^{2}u_{i} + \boldsymbol{b}\cdot\nabla u_{i} + c u_{i}, v)_{\Omega} \\
= & (A:D^{2}(u - u_{i}) + \boldsymbol{b}\cdot \nabla (u - u_{i}) 
+ c(u - u_{i}), v)_{\Omega} + (f_{i} - f, v)_{\Omega}
+ (f, v)_{\Omega}. 
\end{align*}

Since $u_{i} \rightharpoonup u \text{ in } W^{2,p}(\Omega)$ for 
a subsequence of $i \rightarrow + \infty$ and 
$\Vert f_{i} - f \Vert_{L^{p}(\Omega)} \rightarrow 0$ as 
$i \rightarrow + \infty$, we have 
\begin{align*}
(A:D^{2} u + \boldsymbol{b}\cdot \nabla u 
+ c u, v)_{\Omega} = (f, v)_{\Omega}.
\end{align*}

Since $v \in C_{0}^{\infty}(\Omega)$ is chosen arbitrarily, 
$A:D^{2}u +  \boldsymbol{b}\cdot \nabla u + c u= f$ in $\Omega$ almost everywhere. 
By Theorem~\ref{thm_global_A_continuous}, 
$u \in W^{2,p}(\Omega)\cap W_{0}^{1,p}(\Omega)$ is the unique solution of 
(\ref{nondiv_pde_original}). Thus the proof of Step $2$ is complete. 
\end{proof}

\subsection{Linear elliptic PDE in non-divergence form with discontinuous $A$ but $\gamma A$ dominated by $I_{d}$}

We skip the proof of Theorem~\ref{thm_pde_cordes} and Theorem~\ref{thm_global_estimate_cordes}, since they are 
special examples of Theorem~\ref{thm_hjb_pde_wellposedness} and Theorem~\ref{thm_hjb_conv}  
with the index set $\Lambda$ containing a single element. 

\begin{theorem}
\label{thm_pde_cordes}
Let $A \in [L^{\infty}(\Omega)]^{d\times d}$ uniformly elliptic, $\boldsymbol{b}\in [L^{\infty}(\Omega)]^{d}$, 
$c \in L^{\infty}(\Omega)$ with $c \leq 0$, and $\Omega$ be a Lipschitz polyhedra in $\mathbb{R}^{d}$ ($d=2,3$).  
There is a constant $0\leq \underline{\kappa}<1$ 
which may depend on $p$, such that if 
\begin{align}
\label{Cordes_coefficients_strong_nondiv_pde_general}
& \dfrac{ \vert A(\boldx)\vert^{2} 
+ (2\lambda)^{-1}\vert \boldsymbol{b}(\boldx)\vert^{2} 
+ \lambda^{-2}\vert c(\boldx)\vert^{2}}{\left( 
\text{Tr}A(\boldx) + \lambda^{-1}
\vert c(\boldx)\vert\right)^{2}} \\
\nonumber 
\leq & \dfrac{1}{d+ \epsilon}, \qquad
\forall \boldx \in \Omega \text{ almost everywhere},
\end{align}
for some $\epsilon \in (\underline{\kappa},1)$ and $\lambda > 0$,
then there is a unique solution $u \in W^{2,p}(\Omega)
\cap W_{0}^{1,p}(\Omega)$ satisfying (\ref{nondiv_pde_original}) 
for any $f \in L^{p}(\Omega)$. 
Furthermore, there is a positive constant $C$ which may depend on $p$, such that
\begin{align*}
\Vert u\Vert_{W^{2,p}(\Omega)} \leq C \Vert f\Vert_{L^{p}(\Omega)}.
\end{align*}
Here $p$ is valid in the range described in (\ref{p_range1}). 
In addition, $\underline{\kappa} = 0$ if $p=2$ and 
the domain is a convex polyhedra in $\mathbb{R}^{d}$ ($d=2,3$).
\end{theorem}

\begin{theorem}
\label{thm_global_estimate_cordes}
Let $A \in [L^{\infty}(\Omega)]^{d\times d}$ uniformly elliptic, $\boldsymbol{b}\in [L^{\infty}(\Omega)]^{d}$, 
$c \in L^{\infty}(\Omega)$ with $c \leq 0$, and 
$\Omega$ be a Lipschitz polyhedra in $\mathbb{R}^{d}$ ($d=2,3$).  
There is a constant $0 \leq \kappa<1$ 
which may depend on $p$, such that if 
\begin{align}
\label{Cordes_coefficients_strong_nondiv_fem_general}
& \dfrac{\vert A(\boldx)\vert^{2}(\text{Tr} A
(\boldx))^{-2}}{1 + 2\lambda^{-1}\vert c(\boldx)\vert
(\text{Tr}A(\boldx))^{-1}
- \vert A(\boldx)\vert^{-2}
\big( \lambda^{-2} \vert c(\boldx)\vert^{2} 
+ (2\lambda)^{-1}\vert \boldsymbol{b}(\boldx)\vert^{2}\big)} \\
\nonumber 
\leq & \dfrac{1}{d + \epsilon}, \quad
\forall \boldx \in \Omega \text{ almost everywhere},
\end{align}
for some constants $\epsilon \in (\kappa,1)$ and $\lambda > 0$, then 
there is a unique numerical 
solution $u_{h}\in V_{h}$ of the finite element method (\ref{nondiv_fem}) for any 
$0< h < h_{3}$.
In addition, there is a positive constant $C$ which may depend on $p$, such that 
\begin{align*}
\Vert u_{h}\Vert_{W^{1,p}(\Omega)}+\Vert u_{h}\Vert_{W^{2,p}_{h}(\Omega)} \leq C \Vert f \Vert_{L^{p}(\Omega)}.
\end{align*}
Furthermore, if there are 
$\epsilon \in [0,1)$ and $\lambda>0$ such that (\ref{Cordes_coefficients_strong_nondiv_fem_general}) holds, then 
(\ref{Cordes_coefficients_strong_nondiv_pde_general}) in Theorem~\ref{thm_pde_cordes} holds as well for the same 
$\epsilon$ and $\lambda$. We also have the convergent result that 
for any $\chi_{h}\in V_{h}$, 
\begin{align}
\label{nondiv_fem_con_ineq}
& \Vert u_{h} - u\Vert_{W^{1,p}(\Omega)}
+ \Vert u_{h} - u\Vert_{W_{h}^{2,p}(\Omega)} \\
\nonumber
\leq & C \big(\Vert u - \chi_{h}\Vert_{W^{1,p}(\Omega)}
+\Vert u - \chi_{h}\Vert_{W_{h}^{2,p}(\Omega)} 
+ \Vert \overline{P}_{h}(D^{2}u) - D^{2}u\Vert_{L^{p}(\Omega)} \big),
\end{align}
if $0 < h < h_{3}$. Here $\overline{P}_{h}$ is the standard $L^{2}$-orthogonal projection 
onto $[\overline{V}_{h}]^{d\times d}$. 
$p$ is valid in the range described in (\ref{p_range2}).
\end{theorem}

\section{Hamilton-Jacobi-Bellman equation}
\label{sec_hjb}

For the elliptic HJB equation (\ref{hjb_eqs_original}), the main theoretical results are 
Theorem~\ref{thm_hjb_pde_wellposedness} and Theorem~\ref{thm_hjb_conv}.  
Theorem~\ref{thm_hjb_pde_wellposedness} provides the well-posedness of 
strong solution in $W^{2,p}(\Omega)$ of the elliptic HJB equation (\ref{hjb_eqs_original}) if the coefficients 
satisfy the condition (\ref{Cordes_coefficients_strong_hjb_pde_general}), 
while Theorem~\ref{thm_hjb_conv} provides optimal convergence of the FEM (\ref{hjb_fem}) 
if the coefficients satisfy the condition (\ref{Cordes_coefficients_strong_hjb_fem_general}).

\subsection{Auxiliary results for HJB equation}

The identity (\ref{hjb_uncountable_countable_equivalent}) in the following Lemma~\ref{lemma_hjb_measurable} is 
essentially important in analysis in Section~\ref{sec_hjb}, since it implies the index set $\Lambda$ can be 
replaced by a countably subset $\tilde{\Lambda}$.

\begin{lemma}
\label{lemma_hjb_measurable}
We assume that the assumption (\ref{hjb_coeffs_alternatives}) is satisfied. 
Then for any $B \in [L^{1}(\Omega)]^{d\times d}, \boldsymbol{c}\in 
[L^{1}(\Omega)]^{d}, w \in L^{1}(\Omega)$, 
\begin{align*}
\sup_{\alpha \in \Lambda} [A^{\alpha}:B 
+ \boldsymbol{b}^{\alpha}\cdot \boldsymbol{c} + c^{\alpha}w - f^{\alpha}]
\end{align*}
is a measurable function on $\Omega$, and 
\begin{align}
\label{hjb_uncountable_countable_equivalent}
& \sup_{\alpha \in \Lambda} [A^{\alpha}:B
+ \boldsymbol{b}^{\alpha}\cdot \boldsymbol{c} + c^{\alpha}w - f^{\alpha}] \\
\nonumber
= & \sup_{\tilde{\alpha} \in \tilde{\Lambda}} [A^{\tilde{\alpha}}:B 
+ \boldsymbol{b}^{\tilde{\alpha}}\cdot \boldsymbol{c} 
+ c^{\tilde{\alpha}}w - f^{\tilde{\alpha}}]
\qquad \text{ in } \Omega \text{ almost everywhere}.
\end{align}

In addition, if (\ref{hjb_uniform_ellipticity},\ref{hjb_coeffs_bounds}) 
hold and  $B\in [L^{p}(\Omega)]^{d\times d}, \boldsymbol{c}\in [L^{p}(\Omega)]^{d}, 
w \in L^{p}(\Omega)$ and $\sup_{\tilde{\alpha} 
\in \tilde{\Lambda}}\vert f^{\tilde{\alpha}}\vert \in L^{p}(\Omega)$ 
for some $p \in [1, +\infty)$, then we have that 
$\sup_{\alpha \in \Lambda} [A^{\alpha}:B
+ \boldsymbol{b}^{\alpha}\cdot \boldsymbol{c} 
+ c^{\alpha}w - f^{\alpha}] \in L^{p}(\Omega)$ and a positive constant $C$ satisfying
\begin{align}
\label{hjb_operator_bound1}
& \Vert \sup_{\alpha \in \Lambda} [A^{\alpha}:B 
+ \boldsymbol{b}^{\alpha}\cdot \boldsymbol{c} 
+ c^{\alpha}w - f^{\alpha}]\Vert_{L^{p}(\Omega)} \\
\nonumber
\leq & C \big( (\sup_{\tilde{\alpha}\in \tilde{\Lambda}}\Vert A^{\tilde{\alpha}}
\Vert_{L^{\infty}(\Omega)}\Vert B\Vert_{L^{p}(\Omega)}
 +\sup_{\tilde{\alpha}\in \tilde{\Lambda}}\Vert 
\boldsymbol{b}^{\tilde{\alpha}}\Vert_{L^{\infty}(\Omega)}
\Vert \boldsymbol{c}\Vert_{L^{p}(\Omega)} 
+ \sup_{\tilde{\alpha}\in \tilde{\Lambda}}\Vert c^{\tilde{\alpha}}
\Vert_{L^{\infty}(\Omega)}\Vert w\Vert_{L^{p}(\Omega)} \\
\nonumber
& \qquad + \Vert \sup_{\tilde{\alpha} 
\in \tilde{\Lambda}}\vert f^{\tilde{\alpha}}\vert \Vert_{L^{p}(\Omega)} \big).
\end{align}
Here $\tilde{\Lambda} \subset \Lambda$ is introduced in the assumption 
(\ref{hjb_coeffs_alternatives}).
\end{lemma}

\begin{proof}
According to the assumption (\ref{hjb_coeffs_alternatives}), it is easy to see 
that for any $\boldx \in \Omega \setminus E_{0}$ and for any $\alpha\in\Lambda$, 
$a_{ij}^{\alpha}(\boldx),b_{i}^{\alpha}(\boldx),c^{\alpha}(\boldx), 
f^{\alpha}(\boldx)$ are all finite numbers for any $1\leq i,j \leq d$. 
Since $B \in [L^{1}(\Omega)]^{d\times d}, \boldsymbol{c}\in 
[L^{1}(\Omega)]^{d}, w \in L^{1}(\Omega)$, there is a set 
$E_{1} \subset \Omega$ with zero $d$-dimensional Lebesgue measure, 
such that $b_{ij}, c_{i}, w$ are all finite numbers in $\Omega\setminus 
E_{1}$ for any $1\leq i,j \leq d$.

We choose $t \in \mathbb{R}$ arbitrarily. It is easy to see
\begin{align*}
& \{\boldx \in \Omega: \sup_{\alpha \in \Lambda} [A^{\alpha}:B 
+ \boldsymbol{b}^{\alpha}\cdot \boldsymbol{c} + c^{\alpha}w - f^{\alpha}]
(\boldx) > t \} \\
= & \cup_{\alpha \in \Lambda} \{ \boldx \in \Omega : 
[A^{\alpha}:B + \boldsymbol{b}^{\alpha}\cdot \boldsymbol{c} 
+ c^{\alpha}w - f^{\alpha}](\boldx) > t \}.
\end{align*}
We denote by $S_{t} := \{\boldx \in \Omega: \sup_{\alpha \in \Lambda} 
[A^{\alpha}:B + \boldsymbol{b}^{\alpha}\cdot \boldsymbol{c} 
+ c^{\alpha}w - f^{\alpha}](\boldx) > t \}$.

Therefore, for any $\boldx_{0} \in S_{t} \setminus (E_{0}\cup E_{1})$, there 
is $\alpha_{0} \in \Lambda$ such that 
\begin{align*}
[A^{\alpha_{0}}:B
+ \boldsymbol{b}^{\alpha_{0}}\cdot \boldsymbol{c}  
+ c^{\alpha_{0}}w - f^{\alpha_{0}}](\boldx_{0}) > t
\end{align*}

We define $\epsilon_{0} = [A^{\alpha_{0}}:B
+ \boldsymbol{b}^{\alpha_{0}}\cdot \boldsymbol{c}  
+ c^{\alpha_{0}}w - f^{\alpha_{0}}](\boldx_{0}) - t > 0$. 
According to the assumption~(\ref{hjb_coeffs_alternatives}), there is 
$\tilde{\alpha}_{0} \in \tilde{\Lambda}$ such that 
\begin{align*}
& \vert (A^{\tilde{\alpha}_{0}} - A^{\alpha_{0}}):B(\boldx_{0})\vert 
+ \vert (\boldsymbol{b}^{\tilde{\alpha}_{0}} - \boldsymbol{b}^{\alpha_{0}})
\cdot \boldsymbol{c} (\boldx_{0}) \vert \\ 
& \qquad + \vert (c^{\tilde{\alpha}_{0}}- c^{\alpha_{0}})
w(\boldx_{0}) \vert + \vert 
(f^{\tilde{\alpha}_{0}} - f^{\alpha_{0}})(\boldx_{0}) 
\vert <  \frac{1}{2}\epsilon_{0}.
\end{align*}
So we have 
\begin{align*}
[A^{\tilde{\alpha}_{0}}:B + \boldsymbol{b}^{\tilde{\alpha}_{0}}\cdot 
\boldsymbol{c} + c^{\tilde{\alpha}_{0}}w - f^{\tilde{\alpha}_{0}}](\boldx_{0}) - t 
> \epsilon_{0} - \frac{1}{2} \epsilon_{0} > 0.
\end{align*}
Therefore, we have 
\begin{align*}
\boldx_{0} \in \{ \boldx \in\Omega \setminus(E_{0}\cup E_{1}):\sup_{\tilde{\alpha}
\in \tilde{\Lambda}}
[A^{\tilde{\alpha}}:B + \boldsymbol{b}^{\tilde{\alpha}}\cdot\boldsymbol{c} 
+ c^{\tilde{\alpha}}w - f^{\tilde{\alpha}}](\boldx) > t \}.
\end{align*}

Since $\boldx_{0} \in S_{t} \setminus (E_{0}\cup E_{1})$ is chosen 
arbitrarily, we have 
\begin{align*}
& S_{t}\setminus (E_{0} \cup E_{1}) = 
\{ \boldx \in \Omega \setminus (E_{0}\cup E_{1}): \sup_{\alpha \in \Lambda} 
[A^{\alpha}:B + \boldsymbol{b}^{\alpha}\cdot \boldsymbol{c} 
+ c^{\alpha}w - f^{\alpha}](\boldx) > t \} \\
= & \{ \boldx \in \Omega \setminus (E_{0}\cup E_{1}): \sup_{\tilde{\alpha} \in 
\tilde{\Lambda}} [A^{\tilde{\alpha}}:B 
+ \boldsymbol{b}^{\tilde{\alpha}}\cdot \boldsymbol{c} 
+ c^{\tilde{\alpha}}w - f^{\tilde{\alpha}}]
(\boldx) > t \} \\
= & \cup_{\tilde{\alpha}\in \tilde{\Lambda}} 
\{\boldsymbol{x}\in \Omega \setminus (E_{0}\cup E_{1}): 
[A^{\tilde{\alpha}}:B 
+ \boldsymbol{b}^{\tilde{\alpha}}\cdot \boldsymbol{c} 
+ c^{\tilde{\alpha}}w - f^{\tilde{\alpha}}]
(\boldx) > t \}.
\end{align*}

Since $\tilde{\Lambda}$ has at most countably many elements and 
$E_{0},E_{1}$ have zero $d$-dimensional Lebesgue measure, we have that 
$S_{t}$ is a measurable subset of $\Omega$. Since $t\in \mathbb{R}$ 
is chosen arbitrarily, we can conclude 
\begin{align*}
& \sup_{\alpha \in \Lambda} [A^{\alpha}:B 
+ \boldsymbol{b}^{\alpha}\cdot \boldsymbol{c} + c^{\alpha}w - f^{\alpha}] \\
\end{align*}
is a measurable function on $\Omega$ and 
(\ref{hjb_uncountable_countable_equivalent}) holds.

According to (\ref{hjb_uncountable_countable_equivalent}), (\ref{hjb_uniform_ellipticity},\ref{hjb_coeffs_bounds}) 
and the fact $\sup_{\tilde{\alpha} 
\in \tilde{\Lambda}}\vert f^{\tilde{\alpha}}\vert \in L^{p}(\Omega)$
imply that $\sup_{\alpha \in \Lambda} [A^{\alpha}:B 
+ \boldsymbol{b}^{\alpha}\cdot \boldsymbol{c} 
+ c^{\alpha}w - f^{\alpha}]\in L^{p}(\Omega)$ 
and (\ref{hjb_operator_bound1}) holds.
\end{proof}

\begin{lemma}
\label{lemma_hjb_equivalent}
We assume that the assumption (\ref{hjb_coeffs_alternatives}) is satisfied. 
For any $\alpha \in \Lambda$, let $\varrho^{\alpha} \in L^{\infty}(\Omega)$. 
We further assume that there are two positive constants $\nu_{0}\leq 
\nu_{1}$ and $E_{1} \subset \Omega$ with zero $d$-dimensional Lebesgue 
measure, such that 
\begin{subequations}
\label{weight_func_assumptions}
\begin{align}
\label{weight_func_assumption1}
& \nu_{0} \leq \varrho^{\alpha}(\boldx) \leq \nu_{1}, 
\qquad \forall \boldx \in \Omega \setminus E_{1}, 
\forall \alpha \in \Lambda; \\
\label{weight_func_assumption2}
& \forall \boldx\in \Omega \setminus E_{1} \text{ and } 
\forall \alpha \in \Lambda \text{ and } \forall \epsilon >0, 
\text{ there is } \tilde{\alpha} \in \tilde{\Lambda} \text{ satisfying } \\
\nonumber
& \qquad \vert \varrho^{\alpha}(\boldx) - \varrho^{\tilde{\alpha}}(\boldx) \vert 
< \epsilon.  
\end{align}
\end{subequations}
Here $\tilde{\Lambda} \subset \Lambda$ is introduced in in the assumption~(\ref{hjb_coeffs_alternatives}), 
and $\tilde{\alpha}\in \tilde{\Lambda}$ is the same as that in the 
assumption~(\ref{hjb_coeffs_alternatives}) if $\boldx \in \Omega \setminus (E_{0}\cup E_{1})$.

We choose $w \in W^{2,1}(\Omega)$ arbitrarily. Then 
$\sup_{\alpha\in \Lambda}\varrho^{\alpha}[A^{\alpha}:D^{2}w 
+ \boldsymbol{b}^{\alpha}\cdot \nabla w + c^{\alpha}w - f^{\alpha}]$ 
is a measurable function on $\Omega$.
If $\sup_{\alpha\in \Lambda}[A^{\alpha}:D^{2}w 
+ \boldsymbol{b}^{\alpha}\cdot \nabla w + c^{\alpha}w - f^{\alpha}] = 0$ 
in $\Omega$, then 
\begin{align*}
\sup_{\alpha\in \Lambda}\varrho^{\alpha}[A^{\alpha}:D^{2}w 
+ \boldsymbol{b}^{\alpha}\cdot \nabla w + c^{\alpha}w - f^{\alpha}] = 0 
\qquad \text{ in } \Omega.
\end{align*}
\end{lemma}

\begin{proof}
According to the assumption (\ref{hjb_coeffs_alternatives}), it is easy to see 
that for any $\boldx \in \Omega \setminus E_{0}$ and for any $\alpha\in\Lambda$, 
$a_{ij}^{\alpha}(\boldx),b_{i}^{\alpha}(\boldx),c^{\alpha}(\boldx), 
f^{\alpha}(\boldx)$ are all finite numbers for any $1\leq i,j \leq d$.
Then the assumption~(\ref{hjb_coeffs_alternatives}) is 
satisfied by $\left(\varrho^{\alpha}A^{\alpha}, \varrho^{\alpha}
\boldsymbol{b}^{\alpha}, \varrho^{\alpha}c^{\alpha}, 
\varrho^{\alpha}f^{\alpha}\right)$ with $\tilde{\Lambda}\subset \Lambda$ 
and the set $E_{0}\cup E_{1}$. 
So Lemma~\ref{lemma_hjb_measurable} implies that 
$\sup_{\alpha\in \Lambda}\varrho^{\alpha}[A^{\alpha}:D^{2}w 
+ \boldsymbol{b}^{\alpha}\cdot \nabla w + c^{\alpha}u - f^{\alpha}]$ 
is a measurable function on $\Omega$.

Since $w \in W^{2,1}(\Omega)$, there is a set $E_{1}^{\prime}\subset 
\Omega$ with zero $d$-dimensional Lebesgue measure such that 
$\frac{\partial^{2}w}{\partial x_{i}\partial x_{j}}(\boldx), 
\frac{\partial w}{\partial x_{i}}(\boldx), w(\boldx)$ are all finite numbers 
for any $\boldx \in \Omega \setminus E_{1}^{\prime}$ and 
for any $1\leq i,j \leq d$.
We choose $\boldx_{0} \in \Omega \setminus (E_{0}\cup E_{1} \cup 
E_{1}^{\prime})$ arbitrarily. We claim that if 
$\sup_{\alpha\in \Lambda}[A^{\alpha}:D^{2}w + \boldsymbol{b}^{\alpha}\cdot \nabla w 
+ c^{\alpha}w - f^{\alpha}](\boldx_{0}) = 0$, then 
\begin{align*}
\sup_{\alpha\in \Lambda} \varrho^{\alpha}[A^{\alpha}:D^{2}w 
+ \boldsymbol{b}^{\alpha}\cdot \nabla w 
+ c^{\alpha}w - f^{\alpha}](\boldx_{0}) = 0.
\end{align*}
The proof will be complete if this claim is true.

We prove this claim in the following. 
The fact $\sup_{\alpha\in \Lambda}[A^{\alpha}:D^{2}w 
+ \boldsymbol{b}^{\alpha}\cdot \nabla w 
+ c^{\alpha}w - f^{\alpha}](\boldx_{0}) = 0$ is equivalent to  
\begin{subequations}
\label{point_sup_props}
\begin{align}
\label{point_sup_prop1} 
& [A^{\alpha}:D^{2}w + \boldsymbol{b}^{\alpha}\cdot \nabla w 
+ c^{\alpha}w - f^{\alpha}](\boldx_{0}) \leq 0, \qquad 
\forall \alpha \in \Lambda; \\
\label{point_sup_prop2} 
& \text{there is a sequence } \{ \alpha_{i} \}_{i=1}^{+\infty} \subset \Lambda 
\text{ such that } \\
\nonumber 
& \qquad \lim_{i\rightarrow +\infty}[A^{\alpha_{i}}:D^{2}w 
+ \boldsymbol{b}^{\alpha_{i}}
\cdot \nabla w + c^{\alpha_{i}}w - f^{\alpha_{i}}](\boldx_{0}) = 0.
\end{align}
\end{subequations}
By (\ref{weight_func_assumption1},\ref{point_sup_prop1}), we have 
\begin{align*}
\varrho^{\alpha}[A^{\alpha}:D^{2}w + \boldsymbol{b}^{\alpha}\cdot \nabla w 
+ c^{\alpha}w - f^{\alpha}](\boldx_{0}) \leq 0, \qquad 
\forall \alpha \in \Lambda.
\end{align*}
On the other hand, (\ref{weight_func_assumption1},\ref{point_sup_prop1},
\ref{point_sup_prop2}) implies 
\begin{align*}
\lim_{i\rightarrow +\infty} \varrho^{\alpha}[A^{\alpha_{i}}:D^{2}w 
+ \boldsymbol{b}^{\alpha_{i}}
\cdot \nabla w + c^{\alpha_{i}}w - f^{\alpha_{i}}](\boldx_{0}) = 0. 
\end{align*}
So the claim has been proven. Therefore, the proof is complete.
\end{proof}

The following Lemma~\ref{lemma_equi_uniform_continuity_to_alternatives} shows that 
the assumptions on coefficients and source term of (\ref{hjb_eqs_original}) used in 
\cite{SmearsSuli2014} and (\cite[Proposition~$4.3$]{GallistlTran2025} imply the assumption
(\ref{hjb_coeffs_alternatives}).
\begin{lemma}
\label{lemma_equi_uniform_continuity_to_alternatives}
We assume that for any $\epsilon >0$, there is $\delta > 0$ such that 
\begin{align}
\label{hjb_coeffs_equi_uniform_continuity} 
\vert A^{\alpha}(\boldx) - A^{\alpha}(\boldy) \vert + \vert \boldsymbol{b}^{\alpha}(\boldx) 
- \boldsymbol{b}^{\alpha}(\boldy)  \vert + \vert c^{\alpha}(\boldx) - c^{\alpha}(\boldy) \vert 
+ \vert f^{\alpha}(\boldx) - f^{\alpha}(\boldy) \vert < \epsilon, 
\end{align} 
for any $\alpha \in \Lambda$ and for any $\boldx,\boldy \in \Omega$ satisfying 
$\vert \boldx - \boldy \vert < \delta$.
Then the assumption (\ref{hjb_coeffs_alternatives}) is satisfied.
\end{lemma}

\begin{proof}
For any $\alpha \in \Lambda$ and any $\boldx \in \Omega$, we define $\boldsymbol{q}^{\alpha}(\boldx)\in \mathbb{R}^{d^{2}+d+2}$ 
to be the vector consisting of all components in $A^{\alpha}(\boldx)$, $\boldsymbol{b}^{\alpha}(\boldx)$, 
$c^{\alpha}(\boldx)$ and $f^{\alpha}(\boldx)$.  For any $1\leq i,j \leq d$, $q_{(i(d-1)+j)}^{\alpha}(\boldx) 
= a_{ij}^{\alpha}(\boldx)$, $q_{d^{2}+i}^{\alpha}(\boldx) = b_{i}^{\alpha}(\boldx)$, $q_{d^{2}+d+1}^{\alpha}(\boldx) 
= c^{\alpha}( \boldx )$ and $q_{d^{2}+d+2}^{\alpha}(\boldx) = f^{\alpha}(\boldx)$. 

Let $n$ be an arbitrary positive integer. We define $\Omega_{n} = \Omega \cap 
\{(l_{1}2^{-n},\cdots, l_{d}2^{-n}): (l_{1},\cdots,l_{d})\in \mathbb{Z}^{d}\}$.
We define a countable subset $\Lambda_{n}$ of $\Lambda$. 
For any $\boldy \in \Omega_{n}$ and any $l^{(i)} \in \mathbb{Z}$ with $1 \leq i \leq d^{2}+d+2$, if there 
is $\alpha \in \Lambda$ 
such that $l^{(i)}2^{-n}\leq q_{i}^{\alpha}(\boldy) < (l^{(i)}+1)2^{-n}$ for all $1 \leq i \leq d^{2}+d+2$, then 
we put this $\alpha$ into $\Lambda_{n}$.
Notice that though there may be more than one $\alpha \in \Lambda$ qualified, we only choose 
at most one $\alpha$ in the above procedure. Therefore, $\Lambda_{n}$ has at most countably many 
elements. We define $\tilde{\Lambda} = \cup_{n=1}^{\infty}\Lambda_{n} \subset \Lambda$. 

We choose $\boldx \in \Omega$ and $\alpha \in \Lambda$ and $\epsilon >0$ arbitrarily. 
Then there is a sequence $\{\boldy_{n}\}_{n=1}^{\infty} \subset \tilde{\Lambda}$ such that 
$ \lim_{n \rightarrow +\infty}\vert \boldy_{n} - \boldx \vert = 0$ and $\boldy_{n}\in \Lambda_{n}$ 
for any positive integer $n$. Then by (\ref{hjb_coeffs_equi_uniform_continuity}), there is a positive 
integer $N_{0}$ such that 
\begin{align*}
\Sigma_{1 \leq i \leq d^{2}+d+2} \vert q_{i}^{\beta}(\boldy_{n}) - q_{i}^{\beta}(\boldx) \vert < \frac{\epsilon}{3}, 
\qquad 2^{-n} < \frac{\epsilon}{3 (d^{2} + d + 2) },
\end{align*}
for any $n \geq N_{0}$ and any $\beta \in \Lambda$. 

There are integers $\{ l_{0}^{(i)} \}_{i=1}^{d^{2}+d+2}$ such that 
$l_{0}^{(i)}2^{-N_{0}} \leq q_{i}^{\alpha}(\boldy_{N_{0}})< (l_{0}^{(i)}+1)2^{-N_{0}}$ 
for all $1 \leq i \leq d^{2}+d+2$. 
By the construction of $\Lambda_{n} \subset \tilde{\Lambda}$, there is 
$\tilde{\alpha} \in \Lambda_{n}$ such that for all $ 1 \leq i \leq d^{2} + d + 2 $,
\begin{align*}
l_{0}^{(i)}2^{-N_{0}} \leq q_{i}^{\tilde{\alpha}}(\boldy_{N_{0}}) < (l_{0}^{(i)} +1)2^{-N_{0}}.
\end{align*}
Then we have 
\begin{align*}
\vert \boldsymbol{q}^{\tilde{\alpha}}(\boldy_{N_{0}}) - \boldsymbol{q}^{\alpha}(\boldy_{N_{0}})\vert 
< (d^{2}+d+2) 2^{-N_{0}}< \frac{\epsilon}{3}.
\end{align*}
We notice that 
\begin{align*}
\vert \boldsymbol{q}^{\alpha}(\boldy_{N_{0}}) - \boldsymbol{q}^{\alpha}(\boldx) \vert < \frac{\epsilon}{3}, \qquad 
\vert \boldsymbol{q}^{\tilde{\alpha}}(\boldy_{N_{0}}) - \boldsymbol{q}^{\tilde{\alpha}}(\boldx) \vert < \frac{\epsilon}{3}.
\end{align*}
Then we have that 
\begin{align*}
\vert \boldsymbol{q}^{\alpha}(\boldx) - \boldsymbol{q}^{\tilde{\alpha}}(\boldx) \vert < \epsilon.
\end{align*}
We can conclude that the proof is complete. 
\end{proof}

\subsection{Well-posedness of Hamilton-Jacobi-Bellman equation}

When the domain is convex and $p=2$, the condition (\ref{Cordes_coefficients_strong_hjb_pde_general}) 
for coefficients in Theorem~\ref{thm_hjb_pde_wellposedness} is identical to the Cordes condition 
used in \cite{SmearsSuli2014}.

\begin{theorem}
\label{thm_hjb_pde_wellposedness}
Let $\Omega$ be a Lipschitz polyhedral in $\mathbb{R}^{d}$ ($d=2,3$). 
We assume that $\sup_{\tilde{\alpha} \in \tilde{\Lambda}}\vert f^{\tilde{\alpha}} 
\vert \in L^{p}(\Omega)$, and (\ref{hjb_uniform_ellipticity},\ref{hjb_coeffs_bounds}) 
and the assumption~(\ref{hjb_coeffs_alternatives}) hold.  
There is a constant $0\leq \underline{\kappa}<1$ 
which may depend on $p$, such that if 
\begin{align}
\label{Cordes_coefficients_strong_hjb_pde_general}
& \dfrac{ \vert A^{\alpha}(\boldx)\vert^{2} 
+ (2\lambda)^{-1}\vert \boldsymbol{b}^{\alpha}(\boldx)\vert^{2} 
+ \lambda^{-2}\vert c^{\alpha}(\boldx)\vert^{2}}{\left( 
\text{Tr}A^{\alpha}(\boldx) + \lambda^{-1}
\vert c^{\alpha}(\boldx)\vert\right)^{2}} \\
\nonumber 
\leq & \dfrac{1}{d+ \epsilon}, \qquad
\forall \boldx \in \Omega \text{ almost everywhere},
\forall \alpha \in \Lambda,
\end{align}
for some $\epsilon \in (\underline{\kappa},1)$ and $\lambda > 0$,
then there is a unique solution $u \in W^{2,p}(\Omega)
\cap W_{0}^{1,p}(\Omega)$ satisfying (\ref{hjb_eqs_original}). 
Furthermore, there is a positive constant $C$ which may depend on $p$, such that
\begin{align*}
\Vert u\Vert_{W^{2,p}(\Omega)} \leq C \Vert \sup_{\tilde{\alpha}
\in \tilde{\Lambda}}\vert f^{\tilde{\alpha}}\vert\Vert_{L^{p}(\Omega)}.
\end{align*}
Here $\tilde{\Lambda}\subset \Lambda$ is introduced in 
the assumption~(\ref{hjb_coeffs_alternatives}). 
$p$ is valid in the range described in (\ref{p_range1}). 
In addition, $\underline{\kappa} = 0$ if $p=2$ and 
the domain is a convex polyhedra in $\mathbb{R}^{d}$ ($d=2,3$).
\end{theorem}

\begin{remark}
\label{remark_hjb_pde_wellposedness}
In special case $\boldsymbol{b}^{\alpha} = \boldsymbol{0}$ and $c^{\alpha} = 0$ 
for any $\alpha \in \Lambda$, Theorem~\ref{thm_hjb_pde_wellposedness} still holds 
if the condition (\ref{Cordes_coefficients_strong_hjb_pde_general}) is replaced by 
\begin{align}
\label{Cordes_coefficients_strong_hjb_pde_special}
& \dfrac{\vert A^{\alpha}(\boldx)\vert^{2}}{\left( 
\text{Tr}A^{\alpha}(\boldx) \right)^{2}} 
\leq \dfrac{1}{d -1 + \epsilon}, \qquad
\forall \boldx \in \Omega \text{ almost everywhere},
\forall \alpha \in \Lambda,
\end{align}
for some $\epsilon \in (\underline{\kappa}^{\prime},1)$ 
where $\underline{\kappa}^{\prime} \in [0,1)$. 
In addition, $\underline{\kappa}^{\prime} = 0$ if $p=2$ and 
the domain is a convex polyhedra in $\mathbb{R}^{d}$ ($d=2,3$).
The proof is similar.
We would like to point out that the condition (\ref{Cordes_coefficients_strong_hjb_pde_general}) 
and the condition (\ref{Cordes_coefficients_strong_hjb_pde_special}) are exactly 
the Cordes conditions in \cite{SmearsSuli2014}.
\end{remark}

\begin{proof}
Let $p$ be any number in the range described in this Theorem.

We denote by 
\begin{align*}
\Vert v\Vert_{W_{\lambda}^{2,p}(\Omega)}= 
\big( \Vert D^{2}v\Vert_{L^{p}(\Omega)}^{p} 
+ 2\lambda\Vert \nabla v\Vert_{L^{p}(\Omega)}^{p} 
+ \lambda^{2}\Vert v\Vert_{L^{p}(\Omega)}^{p} \big)^{1/p}, 
\qquad \forall v \in W^{2,p}(\Omega)\cap W_{0}^{1,p}(\Omega).
\end{align*}

According to the assumption (\ref{hjb_coeffs_alternatives}), it is easy to see 
that for any $\boldx \in \Omega \setminus E_{0}$ and for any $\alpha\in\Lambda$, 
$a_{ij}^{\alpha}(\boldx),b_{i}^{\alpha}(\boldx),c^{\alpha}(\boldx), 
f^{\alpha}(\boldx)$ are all finite numbers for any $1\leq i,j \leq d$. 

For any $\alpha \in \Lambda$, we define a function $\gamma^{\lambda, \alpha}$ 
on $\Omega$ such that for any $\boldx\in\Omega$ almost everywhere,
\begin{align}
\label{def_gamma_lambda_alpha}
\gamma^{\lambda,\alpha}(\boldx) 
= \dfrac{\Sigma_{1\leq i \leq d}a_{ii}^{\alpha}
(\boldx)+ \lambda^{-1}\vert c^{\alpha}(\boldx)\vert}{\vert A^{\alpha}(\boldx)
\vert^{2} + (2\lambda)^{-1}\vert \boldsymbol{b}^{\alpha}(\boldx)\vert^{2} 
+ (\lambda^{-1}c^{\alpha}(\boldx))^{2}}. 
\end{align}
Here $\vert B\vert = \big(\Sigma_{1\leq i,j \leq d}\vert b_{ij} \vert^{2}
\big)^{1/2}$ for any $B \in \mathbb{R}^{d\times d}$ and 
$\vert \boldsymbol{\beta} \vert = \big(\Sigma_{1\leq i \leq d}  
\vert \beta_{i}\vert^{2} \big)^{1/2}$ for any $\boldsymbol{\beta}
\in \mathbb{R}^{d}$.

By (\ref{hjb_uniform_ellipticity},\ref{hjb_coeffs_bounds}) and the fact 
$\tilde{\Lambda}$ is a countable set, there is a set $E_{1}
\supset E_{0}$ with zero $d$-dimensional Lebesgue measure such that   

\begin{align*}
& \gamma^{\lambda,\tilde{\alpha}}(\boldx) 
= \dfrac{\Sigma_{1\leq i \leq d}a_{ii}^{\tilde{\alpha}}
(\boldx)+ \lambda^{-1}\vert c^{\tilde{\alpha}}(\boldx)
\vert}{\vert A^{\tilde{\alpha}}(\boldx)\vert^{2}+(2\lambda)^{-1}
\vert \boldsymbol{b}^{\tilde{\alpha}}(\boldx)\vert^{2} 
+ (\lambda^{-1}c^{\tilde{\alpha}}(\boldx))^{2}}, \\
& C_{1} \leq \gamma^{\lambda,\tilde{\alpha}}(\boldx) \leq C_{2},
\end{align*}
for any $\boldx \in \Omega \setminus E_{1}$ and 
any $\tilde{\alpha} \in \tilde{\Lambda}$. 
Here $0< C_{1} \leq C_{2}<+\infty$ are two constants independent of 
$\boldx \in \Omega \setminus E_{1}$ and 
$\tilde{\alpha} \in \tilde{\Lambda}$.
Then by the assumption (\ref{hjb_coeffs_alternatives}), we have 
\begin{subequations}
\label{weight_func_props_one}
\begin{align}
& \gamma^{\lambda,\alpha}(\boldx) 
= \dfrac{\Sigma_{1\leq i \leq d}a_{ii}^{\alpha}
(\boldx)+ \lambda^{-1}\vert c^{\alpha}(\boldx)\vert}{ 
\vert A^{\alpha}(\boldx)\vert^{2}+(2\lambda)^{-1}
\vert \boldsymbol{b}^{\alpha}(\boldx)\vert^{2}
+(\lambda^{-1}c^{\alpha}(\boldx))^{2}},\\
& C_{1} \leq \gamma^{\lambda,\alpha}(\boldx) \leq C_{2},
\end{align}
\end{subequations}
for any $\boldx\in\Omega\setminus E_{1}$ and 
any $\alpha \in \Lambda$.
By (\ref{weight_func_props_one}) and the assumption (\ref{hjb_coeffs_alternatives}) 
for $(A^{\alpha},\boldsymbol{b}^{\alpha}, c^{\alpha},f^{\alpha})$,  
we have that the assumption (\ref{hjb_coeffs_alternatives}) is satisfied by 
$(\gamma^{\lambda,\alpha}A^{\alpha},\gamma^{\lambda,\alpha}\boldsymbol{b}^{\alpha},
\gamma^{\lambda,\alpha}c^{\alpha},\gamma^{\lambda,\alpha}f^{\alpha})$ 
with $\tilde{\Lambda}\subset \Lambda$ and the set $E_{1}$.
Then by Lemma~\ref{lemma_hjb_measurable} and (\ref{hjb_uniform_ellipticity},\ref{hjb_coeffs_bounds}), we 
have that for any $w\in W^{2,p}(\Omega) \cap W_{0}^{1,p}(\Omega)$,
\begin{align}
\label{hjb_sol_bound}
& \sup_{\alpha \in \Lambda} \gamma^{\lambda,\alpha}[A^{\alpha}:D^{2}w 
+ \boldsymbol{b}^{\alpha}\cdot \nabla w 
+ c^{\alpha}w - f^{\alpha}] \in L^{p}(\Omega),\\
\nonumber
& \Vert \sup_{\alpha \in \Lambda} \gamma^{\lambda,\alpha}[A^{\alpha}:D^{2}w 
+ \boldsymbol{b}^{\alpha}\cdot \nabla w 
+ c^{\alpha}w - f^{\alpha}]\Vert_{L^{p}(\Omega)} \\
\nonumber
\leq & C \big( \Vert w\Vert_{W^{2,p}(\Omega)} + \Vert \sup_{\tilde{\alpha}
\in \tilde{\Lambda}} \vert f^{\tilde{\alpha}} \vert \Vert_{L^{p}(\Omega)}\big).
\end{align}
Then by (\ref{weight_func_props_one}) again, Lemma~\ref{lemma_hjb_equivalent} implies that 
for any $w \in W^{2,p}(\Omega)\cap W_{0}^{1,p}(\Omega)$, 
$\sup_{\alpha \in \Lambda} [A^{\alpha}:D^{2}w 
+ \boldsymbol{b}^{\alpha}\cdot \nabla w 
+ c^{\alpha}w - f^{\alpha}] = 0$ in $\Omega$ is equivalent to 
\begin{align*}
\sup_{\alpha \in \Lambda} \gamma^{\lambda,\alpha}[A^{\alpha}:D^{2}w 
+ \boldsymbol{b}^{\alpha}\cdot \nabla w 
+ c^{\alpha}w - f^{\alpha}] = 0 \quad \text{ in } \Omega,
\end{align*}
which is equivalent to 
\begin{align}
\label{hjb_pde_equivalent}
\sup_{\tilde{\alpha} \in \tilde{\Lambda}} 
\gamma^{\lambda,\tilde{\alpha}}[A^{\tilde{\alpha}}:D^{2}w 
+ \boldsymbol{b}^{\tilde{\alpha}}\cdot \nabla w 
+ c^{\tilde{\alpha}}w - f^{\tilde{\alpha}}] = 0 \quad \text{ in } \Omega,
\end{align}
due to (\ref{hjb_uncountable_countable_equivalent}).

We define a mapping $M_{\lambda}: L^{p}(\Omega)\rightarrow 
W^{2,p}(\Omega)\cap W_{0}^{1,p}(\Omega)$ by 
\begin{align*}
(\Delta - \lambda )(M_{\lambda} g) = g \text{ in } \Omega,
\end{align*}
for any $g \in L^{p}(\Omega)$.
According to Theorem~\ref{thm_A_continuous_complete}, 
we have that for any $g\in L^{p}(\Omega)$,
\begin{align}
\label{coercivity_poisson_lambda_ineq}
\Vert M_{\lambda}g\Vert_{W_{\lambda}^{2,p}(\Omega)}
 \leq \underline{C}_{p,\lambda} \Vert g\Vert_{L^{p}(\Omega)},
\end{align}
where the constant $\underline{C}_{p,\lambda}$ may depend on $p$ and $\lambda$. 
In fact, $\underline{C}_{2,\lambda} = 1$ for any $\lambda \geq 1$ with convex 
domain due to \cite[Theorem~$2$]{SmearsSuli2014}.

By (\ref{hjb_pde_equivalent}) and the definition of $M_{\lambda}$, $u 
\in W^{2,p}(\Omega)\cap W_{0}^{1,p}(\Omega)$ satisfies the 
Hamilton-Jacobi-Bellman equation (\ref{hjb_eqs_original}) if and only if
\begin{align*}
- u = M_{\lambda}\big( \sup_{\tilde{\alpha} \in \tilde{\Lambda}} 
\gamma^{\lambda,\tilde{\alpha}}[A^{\tilde{\alpha}}:D^{2}u 
+ \boldsymbol{b}^{\tilde{\alpha}}\cdot \nabla u 
+ c^{\tilde{\alpha}}u - f^{\tilde{\alpha}}] 
- (\Delta u - \lambda u)\big).
\end{align*} 
On the other hand, (\ref{hjb_sol_bound}) and 
(\ref{coercivity_poisson_lambda_ineq}) imply that
\begin{align}
\label{hjb_pde_contraction}
w \rightarrow T_{\lambda}(w):= 
M_{\lambda}\big( \sup_{\tilde{\alpha} \in 
\Lambda} \gamma^{\lambda,\tilde{\alpha}} [A^{\tilde{\alpha}}:D^{2}w 
+ \boldsymbol{b}^{\tilde{\alpha}}\cdot \nabla w 
+ c^{\tilde{\alpha}}w - f^{\tilde{\alpha}}] - (\Delta - \lambda)w \big)
\end{align}
is a well defined mapping from $W^{2,p}(\Omega)\cap W_{0}^{1,p}(\Omega)$ into 
itself. 
Therefore, in order to to the existence and uniqueness of the solution 
of (\ref{hjb_eqs_original}), it is sufficient to prove that 
the mapping (\ref{hjb_pde_contraction}) is a contraction with respect to 
the norm $\Vert \cdot\Vert_{W_{\lambda}^{2,p}(\Omega)}$.

By (\ref{Cordes_coefficients_strong_hjb_pde_general}) and the proof 
of \cite[Lemma~$1$]{SmearsSuli2014} ($\underline{\kappa}$ 
will be determined later), we have that 
for any $w_{1},w_{2}\in W^{2,p}(\Omega)\cap W_{0}^{1,p}(\Omega)$,
\begin{align}
\label{hjb_operator_diff_ineq}
& \Vert \sup_{\tilde{\alpha} \in 
\Lambda} \gamma^{\lambda,\tilde{\alpha}} [A^{\tilde{\alpha}}:D^{2}w_{1} 
+ \boldsymbol{b}^{\tilde{\alpha}}\cdot \nabla w_{1} 
+ c^{\tilde{\alpha}}w_{1} - f^{\tilde{\alpha}}] \\
\nonumber
& \qquad - \sup_{\tilde{\alpha} \in 
\Lambda} \gamma^{\lambda,\tilde{\alpha}} [A^{\tilde{\alpha}}:D^{2}w_{2} 
+ \boldsymbol{b}^{\tilde{\alpha}}\cdot \nabla w_{2} 
+ c^{\tilde{\alpha}}w_{2} - f^{\tilde{\alpha}}] 
- (\Delta - \lambda)(w_{1} - w_{2})\Vert_{L^{p}(\Omega)} \\
\nonumber
\leq & \underline{C}_{p,\lambda}^{\prime}\sqrt{1 - \epsilon}\Vert 
(w_{1}-w_{2})\Vert_{W_{\lambda}^{2,p}(\Omega)}.
\end{align}
In fact, $\underline{C}_{2,\lambda}^{\prime} = 1$ on convex domain.

By (\ref{coercivity_poisson_lambda_ineq},\ref{hjb_operator_diff_ineq}), we have 
\begin{align}
\label{hjb_pde_contraction_bound1}
& \Vert M_{\lambda}\big(\sup_{\tilde{\alpha} \in 
\Lambda} \gamma^{\lambda,\tilde{\alpha}} [A^{\tilde{\alpha}}:D^{2}w_{1} 
+ \boldsymbol{b}^{\tilde{\alpha}}\cdot \nabla w_{1} 
+ c^{\tilde{\alpha}}w_{1} - f^{\tilde{\alpha}}]\big) \\
\nonumber
& \qquad - M_{\lambda}\big(\sup_{\tilde{\alpha} \in 
\Lambda} \gamma^{\lambda,\tilde{\alpha}} [A^{\tilde{\alpha}}:D^{2}w_{2} 
+ \boldsymbol{b}^{\tilde{\alpha}}\cdot \nabla w_{2} 
+ c^{\tilde{\alpha}}w_{2} - f^{\tilde{\alpha}}]\big) 
- (w_{1} - w_{2})\Vert_{W_{\lambda}^{2,p}(\Omega)} \\
\nonumber
\leq & \underline{C}_{p,\lambda}\underline{C}_{p,\lambda}^{\prime}
\sqrt{1 - \epsilon}\Vert (w_{1}-w_{2})\Vert_{W_{\lambda}^{2,p}(\Omega)}.
\end{align}

We can choose 
\begin{align}
\label{def_underline_kappa}
\underline{\kappa}= \max \left( 0, 1 - \frac{1}{(\underline{C}_{p,\lambda}
\underline{C}_{p,\lambda}^{\prime})^{2}} \right).
\end{align}  
By (\ref{coercivity_poisson_lambda_ineq}) and (\ref{hjb_operator_diff_ineq}), 
if (\ref{Cordes_coefficients_strong_hjb_pde_general}) is satisfied 
with $\epsilon \in (\underline{\kappa},1)$, $T_{\lambda}$ is a contraction. 
Then there is a unique solution $u\in W^{2,p}(\Omega)\cap W_{0}^{1,p}
(\Omega)$ satisfying the Hamilton-Jacobi-Bellman equation (\ref{hjb_eqs_original}) 
which is equivalent to (\ref{hjb_pde_equivalent}). 
Since $\underline{C}_{2,\lambda} = \underline{C}_{2,\lambda}^{\prime} = 1$, 
$\underline{\kappa} = 0$ if $p = 2$ with convex domain.

By (\ref{hjb_pde_contraction_bound1}) 
with $w_{1}=u$  and $w_{2}=0$, we have 
\begin{align*}
\Vert u - M_{\lambda}(\sup_{\tilde{\alpha}\in \tilde{\Lambda}}
\gamma^{\tilde{\alpha}}f^{\tilde{\alpha}}) \Vert_{W_{\lambda}^{2,p}(\Omega)} 
\leq \underline{C}_{p,\lambda}\underline{C}_{p,\lambda}^{\prime}
\sqrt{1 - \epsilon}\Vert u\Vert_{W_{\lambda}^{2,p}(\Omega)}.
\end{align*}
Since $\underline{C}_{p,\lambda}
\underline{C}_{p,\lambda}^{\prime}\sqrt{1-\epsilon}<1$, we have 
\begin{align*}
\Vert u\Vert_{W^{2,p}(\Omega)} \leq C \Vert \sup_{\tilde{\alpha}
\in \tilde{\Lambda}}\vert f^{\tilde{\alpha}}\vert\Vert_{L^{p}(\Omega)}.
\end{align*}
Therefore, the proof is complete.
\end{proof}

\subsection{Optimal convergence of numerical solution of FEM for Hamilton-Jacobi-Bellman equation}

When the domain is convex and $p=2$,  the condition~(\ref{Cordes_coefficients_strong_hjb_fem_general}) 
for coefficients in Theorem~\ref{thm_hjb_conv} is more restrictive than the Cord{\`{e}}s condition 
used in \cite{SmearsSuli2014}. The stronger restriction of the condition 
(\ref{Cordes_coefficients_strong_hjb_fem_general}) is due to two reasons. The first reason 
is that the constant of the discrete Miranda–Talenti inequality of the discrete Laplacian operator 
$L_{I,h}$ used in the analysis of the FEM (\ref{hjb_fem}) 
is not known explicitly and probably bigger than $1$. The second reason is that we don't want to include 
the parameter $\lambda$ used in \cite{SmearsSuli2014}, which may not be known beforehand, in the design of the 
weight function $\gamma^{\alpha}$ (see the difference between (\ref{def_gamma_alpha}) and the weight function 
used in \cite{SmearsSuli2014}). Though it is more restrictive, Remark~\ref{remark_new_corders_nonempty} 
shows that the condition~(\ref{Cordes_coefficients_strong_hjb_fem_general}) allows non-trivial examples.  
On the other hand, Theorem~\ref{thm_hjb_conv} is valid for a wide range of $p$ (see (\ref{p_range2})) 
and is valid on two dimensional non-convex polygon for a neighbourhood of $\frac{4}{3}$.

\begin{theorem}
\label{thm_hjb_conv}
Let $\Omega$ be a Lipschitz polyhedral in $\mathbb{R}^{d}$ ($d=2,3$). 
We assume that $\sup_{\tilde{\alpha}\in \tilde{\Lambda}}\vert f^{\tilde{\alpha}}\vert 
\in L^{p}(\Omega)$, and (\ref{hjb_uniform_ellipticity},\ref{hjb_coeffs_bounds}) 
and the assumption~(\ref{hjb_coeffs_alternatives}) hold.
There is a constant $0 \leq \kappa<1$ 
which may depend on $p$, such that if 
\begin{align}
\label{Cordes_coefficients_strong_hjb_fem_general}
& \dfrac{\vert A^{\alpha}(\boldx)\vert^{2}(\text{Tr} A^{\alpha}
(\boldx))^{-2}}{1 + 2\lambda^{-1}\vert c^{\alpha}(\boldx)\vert
(\text{Tr}A^{\alpha}(\boldx))^{-1}
- \vert A^{\alpha}(\boldx)\vert^{-2}
\big( \lambda^{-2} \vert c^{\alpha}(\boldx)\vert^{2} 
+ (2\lambda)^{-1}\vert \boldsymbol{b}^{\alpha}(\boldx)\vert^{2}\big)} \\
\nonumber 
\leq & \dfrac{1}{d + \epsilon}, \quad
\forall \boldx \in \Omega \text{ almost everywhere},
\forall \alpha \in \Lambda,
\end{align}
for some constants $\epsilon \in (\kappa,1)$ and $\lambda > 0$, then 
there is a unique numerical 
solution $u_{h}\in V_{h}$ of the finite element method (\ref{hjb_fem}) for any 
$0< h < h_{3}$.
In addition, there is a positive constant $C$ which may depend on $p$, such that 
\begin{align*}
\Vert u_{h}\Vert_{W^{1,p}(\Omega)}+ \Vert u_{h}\Vert_{W^{2,p}_{h}(\Omega)} 
\leq C \Vert \sup_{\tilde{\alpha}\in \tilde{\Lambda}}\vert f^{\tilde{\alpha}}\vert 
\Vert_{L^{p}(\Omega)}.
\end{align*}
Here $\tilde{\Lambda}\subset \Lambda$ is introduced in the 
assumption~(\ref{hjb_coeffs_alternatives}). Furthermore, if there are 
$\epsilon \in [0,1)$ and $\lambda>0$ such that (\ref{Cordes_coefficients_strong_hjb_fem_general}) holds, then 
(\ref{Cordes_coefficients_strong_hjb_pde_general}) in Theorem~\ref{thm_hjb_pde_wellposedness} holds as well for the same 
$\epsilon$ and $\lambda$. We also have the convergent result that 
for any $\chi_{h}\in V_{h}$, 
\begin{align}
\label{hjb_fem_con_ineq}
& \Vert u_{h} - u\Vert_{W^{1,p}(\Omega)} + \Vert u_{h} - u\Vert_{W_{h}^{2,p}(\Omega)} \\ 
\nonumber
\leq & C \big(\Vert u - \chi_{h}\Vert_{W^{1,p}(\Omega)} + \Vert u - \chi_{h}\Vert_{W_{h}^{2,p}(\Omega)} 
+ \Vert \overline{P}_{h}(D^{2}u) - D^{2}u\Vert_{L^{p}(\Omega)} \big),
\end{align}
if $0 < h < h_{3}$. Here $\overline{P}_{h}$ is the standard $L^{2}$-orthogonal projection 
onto $[\overline{V}_{h}]^{d\times d}$. 
$p$ is valid in the range described in (\ref{p_range2}).
\end{theorem}

\begin{remark}
\label{remark_hjb_conv}
In special case $\boldsymbol{b}^{\alpha} = \boldsymbol{0}$ and $c^{\alpha} = 0$ 
for any $\alpha \in \Lambda$, Theorem~\ref{thm_hjb_conv} still holds (no $\lambda$ any more)
if the condition (\ref{Cordes_coefficients_strong_hjb_pde_general}) is replaced by 
\begin{align}
\label{Cordes_coefficients_strong_hjb_fem_special}
& \dfrac{\vert A^{\alpha}(\boldx)\vert^{2}}{\left( 
\text{Tr}A^{\alpha}(\boldx) \right)^{2}} 
\leq \dfrac{1}{d -1 + \epsilon}, \qquad
\forall \boldx \in \Omega \text{ almost everywhere},
\forall \alpha \in \Lambda,
\end{align}
for some $\epsilon \in (\kappa^{\prime},1)$
where $\kappa^{\prime} \in [0,1)$. 
The proof is much simpler. We would like to emphasize that 
even for $p=2$, $\kappa^{\prime}$ is not known explicitly 
and probably bigger than $0$, since the constant of the discrete 
Miranda–Talenti inequality of the discrete Lapalcian operator $L_{I,h}$ 
associated with the FEM (\ref{hjb_fem}) is not known explicitly and probably bigger than $1$.
\end{remark}

\begin{remark}
\label{remark_new_corders_nonempty}
(The condition (\ref{Cordes_coefficients_strong_hjb_fem_general}) includes non-trivial example) 
We investigate that under which restrictions, the condition (\ref{Cordes_coefficients_strong_hjb_fem_general}) 
can be valid for \cite[Example~$1$]{SmearsSuli2014}.  

We recall that \cite[Example~$1$]{SmearsSuli2014} is defined in two dimensional domain, such that 
\begin{align*}
\vert A^{\alpha}\vert^{2} = \dfrac{1 + (\sin \theta)^{2}}{2}, \quad 
\text{Tr}A^{\alpha} = 1, \quad \boldsymbol{b}^{\alpha} = \boldsymbol{0}, 
\quad c^{\alpha} = -c_{0} \text{ with } c_{0} > 0, 
\end{align*}
for any $\alpha \in \Lambda$ with $\theta \in [0,\frac{\pi}{3}]$.
Then the condition (\ref{Cordes_coefficients_strong_hjb_fem_general}) is equivalent to 
\begin{align}
\label{remark_new_corders_nonempty_ineq}
\dfrac{\frac{1}{2}(1 + (\sin \theta)^{2})}{1 + 2(\frac{c_{0}}{\lambda}) 
- \frac{2}{1+(\sin \theta)^{2}}(\frac{c_{0}}{\lambda})^{2}} \leq \frac{1}{2+\epsilon}.
\end{align}

It is easy to verify that (\ref{remark_new_corders_nonempty_ineq}) holds for any $\epsilon \in (0,1)$ 
if $\theta = 0$ and $\lambda = 2c_{0}$, and holds also for any $\epsilon \in (0, \frac{1}{7}]$ 
if $\theta = \frac{\pi}{3}$ and $\lambda = \frac{8}{7}c_{0}$. 

We notice that $\epsilon \in (\kappa ,1)$ where $\kappa \in [0,1)$ is not known explicitly. 
Then by the continuity argument, there exists $\beta \in (0, \frac{\pi}{3}]$ such that 
(\ref{remark_new_corders_nonempty_ineq}) with $\epsilon = \kappa$ holds for some $\lambda 
\in (\frac{8}{7}c_{0}, 2c_{0})$ and for any $\theta \in [0, \beta]$.
Therefore, we can conclude that conservatively speaking, the condition 
(\ref{Cordes_coefficients_strong_hjb_fem_general}) can be satisfied by \cite[Example~$1$]{SmearsSuli2014} 
with probably more restricted region 
of the parameter $\theta$ (a non-trivial closed subinterval contained in $[0,\frac{\pi}{3}]$). 
\end{remark}

\begin{proof}
Let $p$ be any number in the range described in this Theorem.

Since (\ref{hjb_uniform_ellipticity},\ref{hjb_coeffs_bounds}) 
and the assumption~(\ref{hjb_coeffs_alternatives}) hold, 
it is obvious to mimic the proof of Theorem~\ref{thm_hjb_pde_wellposedness} 
(the argument to obtain (\ref{hjb_pde_equivalent})) to 
have that the finite element method (\ref{hjb_fem}) is equivalent to find 
$u_{h}\in V_{h}$ satisfying 
\begin{align}
\label{hjb_fem_equivalent}
(\sup_{\tilde{\alpha}\in \tilde{\Lambda}}\gamma^{\tilde{\alpha}}
[A^{\tilde{\alpha}}:\overline{\nabla}_{h}
(\nabla u_{h}) + \boldsymbol{b}^{\tilde{\alpha}}\cdot\nabla u_{h} 
+ c^{\tilde{\alpha}}u_{h} - f^{\tilde{\alpha}}], v_{h})_{\Omega} = 0, 
\qquad \forall v_{h} \in V_{h}.
\end{align}
We denote by 
\begin{align*}
\Vert v_{h}\Vert_{W_{h,\lambda}^{2,p}(\Omega)}= 
\big( \Vert v_{h}\Vert_{W_{h}^{2,p}(\Omega)}^{p} 
+ 2\lambda\Vert \nabla v_{h}\Vert_{L^{p}(\Omega)}^{p} 
+ \lambda^{2}\Vert v_{h}\Vert_{L^{p}(\Omega)}^{p} \big)^{1/p}, 
\qquad \forall v_{h} \in V_{h}.
\end{align*}

For any $w_{h} \in V_{h}$, we define $F_{h}(w_{h}) \in V_{h}$ by 
\begin{align*}
&(F_{h}(w_{h}), v_{h})_{\Omega} \\ 
= &(\sup_{\tilde{\alpha}\in \tilde{\Lambda}}\gamma^{\tilde{\alpha}}
[A^{\tilde{\alpha}}:\overline{\nabla}_{h}
(\nabla w_{h}) + \boldsymbol{b}^{\tilde{\alpha}}\cdot\nabla w_{h} 
+ c^{\tilde{\alpha}}w_{h} - f^{\tilde{\alpha}}], v_{h})_{\Omega}, 
\quad \forall v_{h} \in V_{h}.
\end{align*}
We define a mapping $M_{\lambda,h}:V_{h} 
\rightarrow V_{h}$ such that for any $g_{h} \in V_{h}$, 
\begin{align*}
-(\nabla (M_{\lambda,h}(g_{h})), \nabla v_{h})_{\Omega} 
- \lambda (M_{\lambda,h}(g_{h}),v_{h})_{\Omega} = 
(g_{h}, v_{h})_{\Omega},\qquad \forall v_{h} \in V_{h}.
\end{align*} 
It is easy to see that for any $w_{h}\in V_{h}$,
\begin{align*}
M_{\lambda,h}(\mathcal{L}_{I_{d},h}w_{h}-\lambda w_{h}) = w_{h},
\end{align*}
where $\mathcal{L}_{I_{d},h}w_{h} \in V_{h}$ satisfying  
$(\mathcal{L}_{I_{d},h}w_{h}, v_{h})_{\Omega}
=(I_{d}:\overline{\nabla}_{h}(\nabla w_{h}),v_{h})_{\Omega}$ 
for any $v_{h} \in V_{h}$.
Then by Theorem~\ref{thm_global_estimate}, we have that 
\begin{align}
\label{coercivity_poisson_fem_lambda_ineq}
\Vert M_{\lambda,h}(g_{h}) \Vert_{W_{h,\lambda}^{2,p}(\Omega)} 
\leq C_{p,\lambda} \Vert g_{h}\Vert_{L^{p}(\Omega)}, \qquad \forall g_{h}\in V_{h},
\forall 0 < h < h_{3},
\end{align}
where the constant $C_{p,\lambda}$ may depend on $p$ and $\lambda$.

We notice that $u_{h}\in V_{h}$ is a numerical solution of the finite element 
method (\ref{hjb_fem_equivalent}) if and only if $F_{h}(u_{h}) = 0$ in $\Omega$, 
which is equivalent to 
\begin{align*}
-u_{h} = M_{\lambda,h}\big(F_{h}(u_{h}) - (\mathcal{L}_{I_{d},h}u_{h}
-\lambda u_{h}) \big).
\end{align*}
Therefore, to prove the existence and uniqueness of (\ref{hjb_fem_equivalent}), 
it is sufficient to prove the mapping 
\begin{align}
\label{hjb_fem_contraction} 
w_{h} \rightarrow M_{\lambda,h}\big( F_{h}(w_{h}) - (\mathcal{L}_{I_{d},h}w_{h}
-\lambda w_{h}) \big)
\end{align}
is a contraction from $V_{h}$ into itself with respect to the 
$\Vert \cdot \Vert_{W_{h,\lambda}^{2,p}(\Omega)}$ norm.

We choose $w_{1,h},w_{2,h} \in V_{h}$ arbitrarily. 
We notice that for any $v_{h}\in V_{h}$
\begin{align*}
& F_{h}(v_{h}) - (\mathcal{L}_{I_{d},h}v_{h}-\lambda v_{h}) \\
= & P_{h}\big( \sup_{\tilde{\alpha}\in \tilde{\Lambda}}\gamma^{\tilde{\alpha}}
[A^{\tilde{\alpha}}:\overline{\nabla}_{h}
(\nabla v_{h}) + \boldsymbol{b}^{\tilde{\alpha}}\cdot\nabla v_{h} 
+ c^{\tilde{\alpha}}v_{h} - f^{\tilde{\alpha}}] 
- (\overline{\nabla}_{h}(\nabla v_{h})-\lambda v_{h})\big).
\end{align*}
Here $P_{h}$ is the standard $L^{2}$-orthogonal projection onto $V_{h}$. 
Therefore, we have 
\begin{align}
\label{hjb_fem_opreator_projection}
& \Vert \big(F_{h}(w_{1,h})-(\mathcal{L}_{I_{d},h}w_{1,h}-\lambda 
w_{1,h})\big) \\
\nonumber
&\qquad - \big(F_{h}(w_{2,h})-(\mathcal{L}_{I_{d},h}w_{2,h}-\lambda 
w_{2,h})\big) \Vert_{L^{p}(\Omega)} \\
\nonumber
\leq & C_{p}^{\prime} \Vert \sup_{\tilde{\alpha}\in \tilde{\Lambda}}\gamma^{\tilde{\alpha}}
[A^{\tilde{\alpha}}:\overline{\nabla}_{h}
(\nabla w_{1,h}) + \boldsymbol{b}^{\tilde{\alpha}}\cdot\nabla w_{1,h} 
+ c^{\tilde{\alpha}}w_{1,h} - f^{\tilde{\alpha}}] \\
\nonumber 
& \qquad - \sup_{\tilde{\alpha}\in \tilde{\Lambda}}\gamma^{\tilde{\alpha}}
[A^{\tilde{\alpha}}:\overline{\nabla}_{h}
(\nabla w_{2,h}) + \boldsymbol{b}^{\tilde{\alpha}}\cdot\nabla w_{2,h} 
+ c^{\tilde{\alpha}}w_{2,h} - f^{\tilde{\alpha}}] \\
\nonumber 
& \qquad -(\overline{\nabla}_{h}(\nabla (w_{1,h}-w_{2,h}))-\lambda(w_{1,h}-w_{2,h}))
\Vert_{L^{p}(\Omega)}
\end{align}
According to the proof of \cite[Lemma~$1$]{SmearsSuli2014}, we have that 
for any $\boldx \in \Omega$ almost everywhere, 
\begin{align}
\label{hjb_fem_operator_diff}
& \vert \sup_{\tilde{\alpha}\in \tilde{\Lambda}}\gamma^{\tilde{\alpha}}
[A^{\tilde{\alpha}}:\overline{\nabla}_{h}
(\nabla w_{1,h}) + \boldsymbol{b}^{\tilde{\alpha}}\cdot\nabla w_{1,h} 
+ c^{\tilde{\alpha}}w_{1,h} - f^{\tilde{\alpha}}](\boldx) \\
\nonumber 
& \qquad - \sup_{\tilde{\alpha}\in \tilde{\Lambda}}\gamma^{\tilde{\alpha}}
[A^{\tilde{\alpha}}:\overline{\nabla}_{h}
(\nabla w_{2,h}) + \boldsymbol{b}^{\tilde{\alpha}}\cdot\nabla w_{2,h} 
+ c^{\tilde{\alpha}}w_{2,h} - f^{\tilde{\alpha}}](\boldx) \\
\nonumber 
& \qquad -(\overline{\nabla}_{h}(\nabla (w_{1,h}-w_{2,h}))-\lambda(w_{1,h}-w_{2,h}))
(\boldx)\vert \\
\nonumber 
\leq & \big(\sup_{\tilde{\alpha}\in \tilde{\Lambda}}\sqrt{E^{\tilde{\alpha}}(\boldx)}
\big)\sqrt{\vert \overline{\nabla}_{h}(\nabla w_{h}(\boldx))\vert^{2} 
+2\lambda \vert \nabla w_{h}(\boldx)\vert^{2} 
+ \lambda^{2}\vert w_{h}(\boldx)\vert^{2}},
\end{align}
where $w_{h}=w_{1,h} - w_{2,h}$ and 
\begin{align*}
E^{\tilde{\alpha}}(\boldx)
= & \vert \gamma^{\tilde{\alpha}}(\boldx)
A^{\tilde{\alpha}}(\boldx) 
- I_{d} \vert^{2} + \vert \gamma^{\tilde{\alpha}}(\boldx) \vert^{2} 
\frac{\vert \boldsymbol{b}^{\tilde{\alpha}}(\boldx) \vert^{2}}{2\lambda} 
+ \dfrac{\vert \lambda + c^{\tilde{\alpha}}(\boldx) \vert^{2}}{\lambda^{2}}\\
= & \frac{\vert \gamma^{\tilde{\alpha}}(\boldx) \vert^{2}
\vert c^{\tilde{\alpha}}(\boldx)\vert^{2}}{\lambda^{2}}
+\gamma^{\tilde{\alpha}}(\boldx)\big( \frac{
\gamma^{\tilde{\alpha}}(\boldx)\vert \boldsymbol{b}^{\tilde{\alpha}}(\boldx)
\vert}{2\lambda} - \frac{2\vert c^{\tilde{\alpha}}(\boldx)\vert}{\lambda}\big)\\
\nonumber  
& \qquad + \big(d+1 +\vert \gamma^{\tilde{\alpha}}(\boldx) \vert^{2}
\vert A^{\tilde{\alpha}}(\boldx)\vert^{2} - 2 \text{Tr} 
A^{\tilde{\alpha}}(\boldx)\big).
\end{align*}
The last equation above is due to the fact $c^{\tilde{\alpha}}(\boldx)\leq 0$ 
for any $\tilde{\alpha}\in \tilde{\Lambda}$.
By the definition of $\gamma^{\tilde{\alpha}}$ (see (\ref{def_gamma_alpha})), 
\begin{align*}
E^{\tilde{\alpha}}(\boldx) = & \frac{(\text{Tr}A^{\tilde{\alpha}}(\boldx))^{2}
\vert c^{\tilde{\alpha}}(\boldx)\vert^{2}}{\vert A^{\tilde{\alpha}}
(\boldx) \vert^{4}}\lambda^{-2} 
+ \frac{\text{Tr}A^{\tilde{\alpha}}(\boldx)}{\vert A^{\tilde{\alpha}}
(\boldx) \vert^{2}}\big( \frac{\text{Tr}A^{\tilde{\alpha}}(\boldx)}{\vert 
A^{\tilde{\alpha}}(\boldx) \vert^{2}} \frac{\vert \boldsymbol{b}^{\tilde{\alpha}} 
(\boldx)\vert^{2}}{2} - 2 \vert c^{\tilde{\alpha}}(\boldx)\vert\big)\lambda^{-1}\\
& \qquad + (d+1 - \frac{(\text{Tr}A^{\tilde{\alpha}}
(\boldx))^{2}}{\vert A^{\tilde{\alpha}}(\boldx) \vert^{2}}).
\end{align*}
Then for any $\epsilon \in (0,1]$, $E^{\tilde{\alpha}}(\boldx)\leq 1 - \epsilon$ 
is equivalent to 
\begin{align*}
& -\frac{(\text{Tr}A^{\tilde{\alpha}}(\boldx))^{2}
\vert c^{\tilde{\alpha}}(\boldx)\vert^{2}}{\vert A^{\tilde{\alpha}}
(\boldx) \vert^{4}}\lambda^{-2} 
- \frac{\text{Tr}A^{\tilde{\alpha}}(\boldx)}{\vert A^{\tilde{\alpha}}
(\boldx) \vert^{2}}\big( \frac{\text{Tr}A^{\tilde{\alpha}}(\boldx)}{\vert 
A^{\tilde{\alpha}}(\boldx) \vert^{2}} \frac{\vert \boldsymbol{b}^{\tilde{\alpha}} 
(\boldx)\vert^{2}}{2} - 2 \vert c^{\tilde{\alpha}}(\boldx)\vert\big)\lambda^{-1}\\
& \qquad + \frac{(\text{Tr}A^{\tilde{\alpha}}
(\boldx))^{2}}{\vert A^{\tilde{\alpha}}(\boldx) \vert^{2}} \geq d+\epsilon,
\end{align*}
which is equivalent to
\begin{align*}
\dfrac{\vert A^{\alpha}(\boldx)\vert^{2}(\text{Tr} A^{\alpha}
(\boldx))^{-2}}{1 + 2\lambda^{-1}\vert c^{\alpha}(\boldx)\vert
(\text{Tr}A^{\alpha}(\boldx))^{-1}
- \vert A^{\alpha}(\boldx)\vert^{-2}
\big( \lambda^{-2} \vert c^{\alpha}(\boldx)\vert^{2} 
+ (2\lambda)^{-1}\vert \boldsymbol{b}^{\alpha}(\boldx)\vert^{2}\big)}
\leq \dfrac{1}{d + \epsilon}.
\end{align*}
By (\ref{Cordes_coefficients_strong_hjb_fem_general}), we have that 
for any $\epsilon \in (\kappa,1)$,
\begin{align}
\label{hjb_fem_operator_constant_bound}
E^{\tilde{\alpha}}(\boldx) \leq &  1 - \epsilon,
\end{align}
for any $\boldx \in \Omega$ almost everywhere and for any $\tilde{\alpha}
\in \tilde{\Lambda}$. Here the constant $\kappa \in [0,1)$ will be determined 
later.

By (\ref{hjb_fem_opreator_projection},\ref{hjb_fem_operator_diff},\ref{hjb_fem_operator_constant_bound}), we have 
\begin{align}
\label{hjb_fem_operator_diff_ineq}
& \Vert \big(F_{h}(w_{1,h})-(\mathcal{L}_{I_{d},h}w_{1,h}-\lambda 
w_{1,h})\big) \\
\nonumber
&\qquad - \big(F_{h}(w_{2,h})-(\mathcal{L}_{I_{d},h}w_{2,h}-\lambda 
w_{2,h})\big)\Vert_{L^{p}(\Omega)} \\
\nonumber
\leq & C_{p,\lambda}^{\prime}C_{p}^{\prime}\sqrt{1 - \epsilon} 
\big( \Vert \overline{\nabla}_{h}(\nabla (w_{1,h}-w_{2,h})) 
\Vert_{L^{p}(\Omega)}^{p} \\
\nonumber
& \qquad + 2\lambda \Vert \nabla (w_{1,h}-w_{2,h}) \Vert_{L^{p}(\Omega)}^{p} 
+ \lambda^{2}\Vert w_{1,h}-w_{2,h} \Vert_{L^{p}(\Omega)}^{p}\big)^{1/p}.
\end{align}

By (\ref{def_discrete_deri}), it is easy to see that 
\begin{align}
\label{discrete_deri_bound1}
\Vert \overline{\nabla}_{h}(\nabla (w_{1,h}-w_{2,h})) 
\Vert_{L^{p}(\Omega)}^{p} \leq C_{p}^{\prime\prime} 
\Vert w_{1,h}-w_{2,h}\Vert_{W_{h}^{2,p}(\Omega)}.
\end{align}
By (\ref{coercivity_poisson_fem_lambda_ineq}, \ref{hjb_fem_operator_diff_ineq},
\ref{discrete_deri_bound1}), we have 
\begin{align}
\label{hjb_fem_contraction_ineq}
& \Vert M_{\lambda,h}\big(F_{h}(w_{1,h})-(\mathcal{L}_{I_{d},h}w_{1,h}-\lambda 
w_{1,h})\big)\\
\nonumber
&\qquad - M_{\lambda,h}\big(F_{h}(w_{2,h})-(\mathcal{L}_{I_{d},h}w_{2,h}-\lambda 
w_{2,h})\big)\Vert_{W_{h,\lambda}^{2.p}(\Omega)} \\
\nonumber
\leq & \big( C_{p,\lambda}C_{p,\lambda}^{\prime}C_{p}^{\prime}
C_{p}^{\prime\prime}\big)
\sqrt{1-\epsilon}\Vert w_{1,h}-w_{2,h}\Vert_{W_{h,\lambda}^{2,p}(\Omega)}.
\end{align}
In order to make the mapping (\ref{hjb_fem_contraction}) a contraction, 
it is sufficient to choose 
\begin{align}
\label{def_kappa}
\kappa = 1 - \dfrac{1}{\big( C_{p,\lambda}
C_{p,\lambda}^{\prime}C_{p}^{\prime}C_{p}^{\prime\prime}\big)^{2}},
\end{align}
such that $\big( C_{p,\lambda}C_{p,\lambda}^{\prime}C_{p}^{\prime}
C_{p}^{\prime\prime}\big)
\sqrt{1-\epsilon}<1$ for any $\epsilon \in (\kappa,1]$.
Therefore, the existence and uniqueness of numerical solution of 
the finite element method (\ref{hjb_fem}) is proven. 

Let $u_{h}\in V_{h}$ be the numerical solution of (\ref{hjb_fem}). 
Then we have 
\begin{align*}
F_{h}(u_{h}) = 0 \text{ and }
u_{h} = - M_{\lambda,h}\big(F_{h}(u_{h})-(\mathcal{L}_{I_{d},h}u_{h}-\lambda 
u_{h})\big) \text{ in }\Omega.
\end{align*}
By (\ref{hjb_fem_contraction_ineq}) with $w_{1,h}=u_{h}$ and 
$w_{2,h} = 0$, we have 
\begin{align*}
\Vert u_{h} - M_{\lambda,h}\big( \sup_{\tilde{\alpha}\in \tilde{\Lambda}}
\gamma^{\tilde{\alpha}} f^{\tilde{\alpha}}\big)\Vert_{W_{h,\lambda}^{2,p}(\Omega)} 
\leq \big( C_{p,\lambda}C_{p,\lambda}^{\prime}C_{p}^{\prime\prime}\big)
\sqrt{1-\epsilon}\Vert u_{h}\Vert_{W_{h,\lambda}^{2,p}(\Omega)}.
\end{align*}
Since $\big( C_{p,\lambda}C_{p,\lambda}^{\prime}C_{p}^{\prime}
C_{p}^{\prime\prime}\big)
\sqrt{1-\epsilon}<1$ for any $\epsilon \in (\kappa,1]$, we have that 
\begin{align}
\label{hjb_numerical_sol_bound1}
\Vert u_{h}\Vert_{W^{2,p}_{h}(\Omega)} \leq C \Vert 
\sup_{\tilde{\alpha}\in \tilde{\Lambda}}\vert f^{\tilde{\alpha}}\vert 
\Vert_{L^{p}(\Omega)}.
\end{align}

It is straightforward to verify that if there are 
$\epsilon \in (0,1]$ and $\lambda>0$ such that (\ref{Cordes_coefficients_strong_hjb_fem_general}) holds, then 
(\ref{Cordes_coefficients_strong_hjb_pde_general}) in Theorem~\ref{thm_hjb_pde_wellposedness} holds as well for the same 
$\epsilon$ and $\lambda$. Furthermore, it is easy to check that 
$\kappa \geq \underline{\kappa}$ which is introduced in (\ref{def_underline_kappa}). 
Then by Theorem~\ref{thm_hjb_pde_wellposedness}, there is a unique solution 
$u\in W_{2,p}(\Omega)\cap W_{0}^{1,p}(\Omega)$ of the Hamilton-Jacobi-Bellman 
equation (\ref{hjb_eqs_original}). 

We choose $\chi_{h} \in V_{h}$ arbitrarily. By (\ref{hjb_fem_contraction_ineq}) 
with $w_{1,h} = u_{h}$ and $w_{2,h}=\chi_{h}$, we have 
\begin{align*}
\Vert M_{\lambda,h}(F_{h}(\chi_{h})) + (u_{h} - \chi_{h}) \Vert_{W_{h,\lambda}^{2,p}
(\Omega)}\leq \big( C_{p,\lambda}C_{p,\lambda}^{\prime}C_{p}^{\prime}
C_{p}^{\prime\prime}\big)\sqrt{1-\epsilon} 
\Vert u_{h} - \chi_{h}\Vert_{W_{h,\lambda}^{2,p}(\Omega)}.
\end{align*}
By (\ref{coercivity_poisson_fem_lambda_ineq}) and the fact 
$\big( C_{p,\lambda}C_{p,\lambda}^{\prime}C_{p}^{\prime}
C_{p}^{\prime\prime}\big)\sqrt{1-\epsilon}<1$, we have 
\begin{align*}
\Vert u_{h} - \chi_{h}\Vert_{W_{h}^{2,p}(\Omega)} 
\leq C \Vert F_{h}(\chi_{h})\Vert_{L^{p}(\Omega)}.
\end{align*}

By the same argument in the proof of Theorem~\ref{thm_hjb_pde_wellposedness}, 
it is easy to see that 
\begin{align*}
\sup_{\tilde{\alpha}\in \tilde{\Lambda}}\gamma^{\alpha}
[A^{\tilde{\alpha}}:D^{2}u+\boldsymbol{b}^{\tilde{\alpha}}\cdot\nabla u 
+ c^{\tilde{\alpha}}u - f^{\tilde{\alpha}}] = 0 \text{ in } \Omega.
\end{align*}

We notice that 
\begin{align*}
&\Vert F_{h}(\chi_{h})\Vert_{L^{p}(\Omega)} 
\leq C \Vert \sup_{\tilde{\alpha}\in \tilde{\Lambda}}\gamma^{\tilde{\alpha}}
[A^{\tilde{\alpha}}:\overline{\nabla}_{h}
(\nabla \chi_{h}) + \boldsymbol{b}^{\tilde{\alpha}}\cdot\nabla \chi_{h} 
+ c^{\tilde{\alpha}}\chi_{h} - f^{\tilde{\alpha}}]\Vert_{L^{p}(\Omega)}\\
= & C \Vert \sup_{\tilde{\alpha}\in \tilde{\Lambda}}\gamma^{\tilde{\alpha}}
[A^{\tilde{\alpha}}:\overline{\nabla}_{h}
(\nabla \chi_{h}) + \boldsymbol{b}^{\tilde{\alpha}}\cdot\nabla \chi_{h} 
+ c^{\tilde{\alpha}}\chi_{h} - f^{\tilde{\alpha}}] \\
& \qquad - \sup_{\tilde{\alpha}\in \tilde{\Lambda}}\gamma^{\alpha}
[A^{\tilde{\alpha}}:D^{2}u+\boldsymbol{b}^{\tilde{\alpha}}\cdot\nabla u 
+ c^{\tilde{\alpha}}u - f^{\tilde{\alpha}}]
\Vert_{L^{p}(\Omega)} \\
\leq & C \Vert \sup_{\tilde{\alpha}\in \tilde{\Lambda}}  
\vert \gamma^{\tilde{\alpha}}
[A^{\tilde{\alpha}}:\overline{\nabla}_{h}
(\nabla \chi_{h}) + \boldsymbol{b}^{\tilde{\alpha}}\cdot\nabla \chi_{h} 
+ c^{\tilde{\alpha}}\chi_{h} - f^{\tilde{\alpha}}] \\
& \qquad -  \gamma^{\alpha}
[A^{\tilde{\alpha}}:D^{2}u+\boldsymbol{b}^{\tilde{\alpha}}\cdot\nabla u 
+ c^{\tilde{\alpha}}u - f^{\tilde{\alpha}}]\vert \Vert_{L^{p}(\Omega)}.
\end{align*}
Then by (\ref{hjb_uniform_ellipticity},\ref{hjb_coeffs_bounds}), we have 
\begin{align*}
&\Vert F_{h}(\chi_{h})\Vert_{L^{p}(\Omega)} 
\leq C \big( \Vert \overline{\nabla}_{h}(\nabla \chi_{h}) - D^{2}u\Vert_{L^{p}
(\Omega)} + \Vert \chi_{h} - u\Vert_{W^{1,p}(\Omega)}\big) \\
\leq & C \big( \Vert \overline{\nabla}_{h}(\nabla (\chi_{h}-u))
\Vert_{L^{p}(\Omega)} + \Vert \overline{\nabla}_{h}(\nabla u) 
- \nabla (\nabla u)  \Vert_{L^{p}(\Omega)}  
+ \Vert \chi_{h} - u\Vert_{W^{1,p}(\Omega)}\big).
\end{align*}
By Definition~\ref{def_discrete_deri} of the discrete gradient operator 
$\overline{\nabla}_{h}$, we have 
\begin{align*}
& \Vert \overline{\nabla}_{h}(\nabla \chi_{h}) - D^{2}u\Vert_{L^{p}
(\Omega)} \leq C \Vert u - \chi_{h}\Vert_{W_{h}^{2,p}(\Omega)}, \\
& \Vert \overline{\nabla}_{h}(\nabla u) 
- \nabla (\nabla u)  \Vert_{L^{p}(\Omega)} 
= \Vert \overline{P}_{h}D^{2}u - D^{2}u\Vert_{L^{p}(\Omega)}.
\end{align*} 
Therefore, (\ref{hjb_fem_con_ineq}) holds.

We can conclude that proof is complete.
\end{proof}

{\bf Acknowledgements} The author would like to thank Professor Rob Stevenson for the fruitful discussion 
of this paper between them when he visited University of Amsterdam in 2025. The author was supported by 
Liu Bie Ju Centre for Mathematical Sciences (project no. 9600007, 9360109, 9360020).

\end{document}